\pdfoutput=1

\documentclass[preprint,10pt]{elsarticle}

\usepackage{fullpage}
\usepackage[colorlinks=true]{hyperref}

\usepackage{amsmath,amssymb,amsfonts,amsthm}
\theoremstyle{definition}
\newtheorem{definition}{Definition}
\theoremstyle{lemma}
\newtheorem{lemma}{Lemma}

\theoremstyle{theorem}
\newtheorem{theorem}{Theorem}
\theoremstyle{assumption}
\newtheorem{assumption}{Assumption}

\usepackage[titletoc,toc,title]{appendix}

\usepackage{array} 
\usepackage{listings}
\usepackage{mathtools}
\usepackage{pdfpages}
\usepackage[textsize=footnotesize,color=green]{todonotes}
\usepackage{bm}
\usepackage{bbm}

\usepackage{tikz}
\usepackage[normalem]{ulem}
\usepackage{hhline}

\usepackage{graphicx}
\usepackage{subfig}
\usepackage{color}

\usepackage{pgfplots}
\usepackage{pgfplotstable}
\definecolor{markercolor}{RGB}{124.9, 255, 160.65}
\pgfplotsset{
compat=1.3,
width=10cm,
tick label style={font=\small},
label style={font=\small},
legend style={font=\small}
}

\usetikzlibrary{calc}
\usetikzlibrary{intersections} 

\newcommand{\logLogSlopeTriangle}[5]
{
    \pgfplotsextra
    {
        \pgfkeysgetvalue{/pgfplots/xmin}{\xmin}
        \pgfkeysgetvalue{/pgfplots/xmax}{\xmax}
        \pgfkeysgetvalue{/pgfplots/ymin}{\ymin}
        \pgfkeysgetvalue{/pgfplots/ymax}{\ymax}

        \pgfmathsetmacro{\xArel}{#1}
        \pgfmathsetmacro{\yArel}{#3}
        \pgfmathsetmacro{\xBrel}{#1-#2}
        \pgfmathsetmacro{\yBrel}{\yArel}
        \pgfmathsetmacro{\xCrel}{\xArel}

        \pgfmathsetmacro{\lnxB}{\xmin*(1-(#1-#2))+\xmax*(#1-#2)} % in [xmin,xmax].
        \pgfmathsetmacro{\lnxA}{\xmin*(1-#1)+\xmax*#1} % in [xmin,xmax].
        \pgfmathsetmacro{\lnyA}{\ymin*(1-#3)+\ymax*#3} % in [ymin,ymax].
        \pgfmathsetmacro{\lnyC}{\lnyA+#4*(\lnxA-\lnxB)}
        \pgfmathsetmacro{\yCrel}{\lnyC-\ymin)/(\ymax-\ymin)} % THE IMPROVED EXPRESSION WITHOUT 'DIMENSION TOO LARGE' ERROR.

        \coordinate (A) at (rel axis cs:\xArel,\yArel);
        \coordinate (B) at (rel axis cs:\xBrel,\yBrel);
        \coordinate (C) at (rel axis cs:\xCrel,\yCrel);

        \draw[#5]   (A)-- node[pos=0.5,anchor=north] {}
                    (B)-- 
                    (C)-- node[pos=0.5,anchor=west] {\textcolor{black}{#4}}
                    cycle;
    }
}
\newcommand{\logLogSlopeTriangleFlip}[5]
{
    \pgfplotsextra
    {
        \pgfkeysgetvalue{/pgfplots/xmin}{\xmin}
        \pgfkeysgetvalue{/pgfplots/xmax}{\xmax}
        \pgfkeysgetvalue{/pgfplots/ymin}{\ymin}
        \pgfkeysgetvalue{/pgfplots/ymax}{\ymax}

        \pgfmathsetmacro{\xBrel}{#1-#2}
        \pgfmathsetmacro{\yBrel}{#3}
        \pgfmathsetmacro{\xCrel}{#1}

        \pgfmathsetmacro{\lnxB}{\xmin*(1-(#1-#2))+\xmax*(#1-#2)} % in [xmin,xmax].
        \pgfmathsetmacro{\lnxA}{\xmin*(1-#1)+\xmax*#1} % in [xmin,xmax].
        \pgfmathsetmacro{\lnyA}{\ymin*(1-#3)+\ymax*#3} % in [ymin,ymax].
        \pgfmathsetmacro{\lnyC}{\lnyA+#4*(\lnxA-\lnxB)}
        \pgfmathsetmacro{\yCrel}{\lnyC-\ymin)/(\ymax-\ymin)} % THE IMPROVED EXPRESSION WITHOUT 'DIMENSION TOO LARGE' ERROR.

	\pgfmathsetmacro{\xArel}{\xBrel}
        \pgfmathsetmacro{\yArel}{\yCrel}

        \coordinate (A) at (rel axis cs:\xArel,\yArel);
        \coordinate (B) at (rel axis cs:\xBrel,\yBrel);
        \coordinate (C) at (rel axis cs:\xCrel,\yCrel);

        \draw[#5]   (A)-- node[pos=0.5,anchor=east] {\textcolor{black}{#4}}
                    (B)-- 
                    (C)-- node[pos=0.5,anchor=south] {}
                    cycle;
    }
}
\usepackage{stmaryrd}

\renewcommand{\hat}[1]{\widehat{#1}}

\newcommand{\td}[2]{\frac{{\rm d}#1}{{\rm d}{\rm #2}}}
\newcommand{\pd}[2]{\frac{\partial#1}{\partial#2}}

\newcommand{\nor}[1]{\left\| #1 \right\|}
\newcommand{\LRp}[1]{\left( #1 \right)}
\newcommand{\LRs}[1]{\left[ #1 \right]}
\newcommand{\LRa}[1]{\left\langle #1 \right\rangle}
\newcommand{\LRb}[1]{\left| #1 \right|}
\newcommand{\LRc}[1]{\left\{ #1 \right\}}

\newcommand{\LRl}[1]{\left. #1 \right|}

\newcommand{\jump}[1] {\ensuremath{\llbracket#1\rrbracket}}
\newcommand{\avg}[1] {\ensuremath{\LRc{\!\{#1\}\!}}}

\newcommand{\eval}[2][\right]{\relax
  \ifx#1\right\relax \left.\fi#2#1\rvert}

\newcommand{\note}[1]{#1}
\newcommand{\noteOne}[1]{#1}
\newcommand{\noteTwo}[1]{#1}

\newcommand{\diag}[1]{{\rm diag}\LRp{#1}}

\newcolumntype{C}[1]{>{\centering\let\newline\\\arraybackslash\hspace{0pt}}m{#1}}

\newcommand*\diff[1]{\mathop{}\!{\mathrm{d}#1}}

\makeatletter
\renewcommand\d[1]{\mspace{6mu}\mathrm{d}#1\@ifnextchar\d{\mspace{-3mu}}{}}
\makeatother

\graphicspath{{./figs/}}

\begin{document}

%\maketitle

\begin{frontmatter}
\title{On discretely entropy conservative and entropy stable discontinuous Galerkin methods}

\author[rice]{Jesse Chan\corref{cor1}}
\ead{Jesse.Chan@caam.rice.edu}
\address[rice]{Department of Computational and Applied Mathematics, Rice University, 6100 Main St, Houston, TX, 77005}

\begin{abstract}
High order methods based on diagonal-norm summation by parts operators can be shown to satisfy a discrete conservation or dissipation of entropy for nonlinear systems of hyperbolic PDEs \cite{fisher2013high, carpenter2014entropy}.  These methods can also be interpreted as nodal discontinuous Galerkin methods with diagonal mass matrices \cite{gassner2016split, gassner2016well, wintermeyer2017entropy, chen2017entropy}.  In this work, we describe how use flux differencing, quadrature-based projections, and SBP-like operators to construct discretely entropy conservative schemes for DG methods under more arbitrary choices of volume and surface quadrature rules.  The resulting methods are semi-discretely entropy conservative or entropy stable with respect to the volume quadrature rule used.  Numerical experiments confirm the stability and high order accuracy of the proposed methods for the compressible Euler equations in one and two dimensions.  
\end{abstract}
\end{frontmatter}

%\tableofcontents

\section{Introduction}

Numerical simulations in engineering increasingly require higher accuracy without sacrificing computational efficiency.  Because they are more accurate than low order methods per degree of freedom for sufficiently regular solutions, high order methods provide one avenue towards improving fidelity in numerical simulations while maintaining reasonable computational costs.  High order methods which can accommodate unstructured meshes are desirable for problems with complex geometries, and among such methods, high order discontinuous Galerkin (DG) methods are particularly well-suited to the solution of time-dependent hyperbolic problems on modern computing architectures \cite{hesthaven2007nodal, klockner2009nodal}.  

The accuracy of high order methods can be attributed in part to their low numerical dissipation and dispersion compared to low order schemes \cite{ainsworth2004dispersive}.  This accuracy has made them advantageous for the simulation of wave propagation \cite{hesthaven2007nodal, wilcox2010high}.  However, while high order methods can be applied in a stable manner to linear wave propagation problems, instabilities are observed when applying them to nonlinear hyperbolic problems.  This is contrast to low order schemes, whose high numerical dissipation tends to apply a stabilizing effect \cite{wang2013high}.  As a result, most high order schemes for nonlinear conservation laws typically require additional stabilization procedures, including  filtering \cite{hesthaven2007nodal}, slope limiting \cite{krivodonova2007limiters}, artificial viscosity \cite{persson2006sub}, and polynomial de-aliasing through over-integration \cite{kirby2003aliasing}.  Moreover, stabilized numerical methods can still fail, requiring user intervention or heuristic modifications to achieve non-divergent solutions.  

For linear wave propagation problems, semi-discretely energy stable numerical methods can be constructed, even in the presence of curvilinear coordinates or variable coefficients \cite{warburton2013low, chan2016weight1, chan2016weight2, chan2017weight}.  This semi-discrete stability implies that, under a stable timestep restriction (CFL condition), discrete solutions do not suffer from non-physical growth in time.  However, for nonlinear systems of conservation laws, traditional methods do not admit theoretical proofs of semi-discrete stability.  This was addressed for low order methods with the introduction of discretely entropy conservative and entropy stable finite volume schemes by Tadmor \cite{tadmor1987numerical}.   These schemes rely on a specific entropy conservative flux which satisfies a condition involving the entropy variables and entropy potential, and were extended to low order finite volume methods on unstructured grids in \cite{ray2016entropy}.  High order entropy stable methods were also developed for structured grids in \cite{fjordholm2012arbitrarily} based on an entropy conservative essentially non-oscillatory (ENO) reconstruction.  

The extension of entropy conservative schemes to unstructured high order methods was done much more recently for the compressible Euler and Navier-Stokes equations in \cite{fisher2013high, carpenter2014entropy} based on a spectral collocation approach on tensor product elements, which can also be interpreted as a mass-lumped DG spectral element (DG-SEM) scheme.  The proof of entropy conservation relies on the presence of a diagonal mass matrix, the summation by parts (SBP) property \cite{gassner2013skew}, and  a concept referred to as \emph{flux differencing}.  Similar entropy-stable schemes have been constructed for the shallow water and MHD equations \cite{gassner2016well,  wintermeyer2017entropy, winters2017uniquely}.  Finally, high order entropy conservative and entropy stable schemes have been extended to unstructured triangular meshes in \cite{crean2017high, chen2017entropy}.  

{It is possible to construct energy preserving schemes for certain conservation laws based on split forms of conservation laws, which involve both conservative and non-conservative derivative terms.  Split formulations have been shown to recover kinetic energy preserving schemes for the compressible Euler and Navier-Stokes equations under diagonal norm SBP operators \cite{gassner2014kinetic, ortleb2016kinetic, ortleb2017kinetic}. 
Additionally, for dense norm and generalized SBP operators, stable schemes for Burgers' equation can be constructed based on the split form of the underlying equations \cite{gassner2013skew, ranocha2017extended, ranocha2017generalised}.  However, entropy conservative and entropy stable schemes for the compressible Euler or Navier-Stokes equations do not correspond to split formulations \cite{chen2017entropy}, and} (to the authors knowledge) the construction of unstructured high order entropy conservative and entropy stable schemes for these equations has required diagonal norm SBP operators.\footnote{Entropy stable high order finite element and DG methods which do not fall under the diagonal norm SBP-DG category have been proposed \cite{hughes1986new}, but the proofs are often given at the continuous level, relying on exact integration or the chain rule, which do hold at the discrete level.}  We refer to DG methods with these properties as diagonal norm SBP-DG methods.  

Appropriate diagonal norm SBP operators are straightforward to construct on tensor product elements based on a DG-SEM discretization.  Diagonal-norm SBP operators can also be constructed for triangles and tetrahedra \cite{chin1999higher,  hicken2016multidimensional, chen2017entropy}; however, the number of nodal points for such operators is typically greater than the dimension of the natural polynomial approximation space, {and the resulting diagonal norm SBP-DG operators do not correspond to any basis \cite{hicken2016multidimensional}}.  Furthermore, to the author's knowledge, appropriate point sets have only been constructed for $N \leq 4$ in three dimensions \cite{zhebel2014comparison}, and the construction of high order diagonal 
norm SBP-DG methods has not yet been performed for uncommon elements such as pyramids \cite{chan2016orthogonal}.  

% consider non-lumped (dense) mass matrices and more general quadrature rules (e.g.\ rules without boundary points

This work focuses on the construction of entropy conservative high order DG schemes for systems of conservation laws.  In order to generalize beyond diagonal norm SBP-DG methods, we will consider DG discretizations using {over-integrated} quadrature rules with more points than the dimension of the approximation space, {which are commonly used for non-tensor product elements in two and three dimensions \cite{xiao2010quadrature}}.  {These quadrature rules induce DG schemes which are related to dense norm and generalized SBP operators \cite{fernandez2014generalized, ranocha2017extended, ranocha2017comparison}, for which discretely entropy stable schemes for the compressible Euler equations have not yet been constructed}.  {We present proofs of discrete entropy stability using both a matrix formulation involving a ``decoupled'' SBP-like operator and continuous formulations involving projection and lifting operators.  In both cases, the proofs rely only on properties which hold under quadrature-based integration.  We also focus on ensuring discrete entropy stability for conservation laws which do not admit a nonlinearly stable split formulation. }

%We will present proofs of discrete entropy stability in terms of integrals and $L^2$ projections, which we assume are computed using an appropriate quadrature rule.  We stress that, while this work adopts a continuous notation involving integrals rather than discrete sums, the proofs rely only on properties of quadrature-based integration and quadrature-based $L^2$ projection.  

The outline of the paper is as follows: Section~\ref{sec:intro} will briefly review the construction of entropy conservative diagonal norm SBP-DG methods on a single element.  Section~\ref{sec:ecdg} will describe how to construct analogous entropy conservative methods on single element in a continuous setting.  Section~\ref{sec:ecdg2} will discuss extensions to multiple elements, including comparisons of different coupling terms and entropy stable fluxes.  Finally, Section~\ref{sec:num} presents numerical results which verify the high order accuracy and discrete entropy conservation of the proposed methods {in one and two spatial dimensions.}  %We note that, for simplicity and clarity of presentation, the proposed methods, proofs, and numerical experiments are given in one dimension.  However, the extension to multiple dimensions (including simplicial elements) is straightforward and will be described in an upcoming manuscript.  

%\note{The construction of such operators on non-tensor product elements (such as simplices or pyramids) is non-trivial.} \cite{hicken2016multidimensional, chen2017entropy}.  \note{Constructing such operators for non-polynomial spaces (such as splines) is also difficult within the current framework.}

\section{Entropy stability for systems of hyperbolic PDEs}
%\section{{Review of} entropy conservative diagonal norm SBP-DG methods}
\label{sec:intro}

%We will begin by reviewing continuous entropy theory and existing high order methods which are provably entropy conservative at the discrete level.  These methods were introduced in \cite{fisher2013high, carpenter2014entropy} for tensor product elements, though we adopt methods of proof introduced in \cite{gassner2017br1, chen2017entropy}.  The proofs are based on diagonal norm SBP operators, which can be derived from a nodal discontinuous Galerkin framework through the collocation of nodal points and quadrature points \cite{gassner2013skew}.  

%\note{Specify that all analysis is done in 1D but is extendable to higher dimensions!!  For simplicity and clarity of presentation, we will begin by focusing on the case $d = 1$.  }
We will begin by reviewing continuous entropy theory.  We consider systems of nonlinear conservation laws in one dimension with $n$ variables
\begin{align}
%\pd{\bm{u}}{t} + \sum_{k=1}^d \pd{\bm{f_k(\bm{u})}}{x} &= 0, \qquad \bm{u}(x,t) = (u_1(x,t),\ldots,u_n(x,t)).
\pd{\bm{u}}{t} + \pd{\bm{f(\bm{u})}}{x} &= 0, \qquad \bm{u}(x,t) = (u_1(x,t),\ldots,u_n(x,t)).
\label{eq:pde}
\end{align}
where the fluxes $f(\bm{u})$ are smooth functions of the vector of conservative variables $\bm{u}({x},t)$.  We are interested in systems for which there exists a convex entropy function $U(\bm{u})$ such that  
\begin{equation}
%U''(\bm{u})\bm{A}_k(\bm{u}) = \LRp{U''(\bm{u}) \bm{A}_k(\bm{u})}^T.
U''(\bm{u})\bm{A}(\bm{u}) = \LRp{U''(\bm{u}) \bm{A}(\bm{u})}^T, \qquad \LRp{\bm{A}(\bm{u})}_{ij} = \LRp{\pd{\bm{f}(\bm{u})}{{u}_j}}_i,
\label{eq:entropysym}
\end{equation}
where $\bm{A}(\bm{u})$ is the Jacobian matrix.  
For systems with convex entropy functions, one can define entropy variables $\bm{v} = U'(\bm{u})$.  The convexity of the $U(\bm{u})$ guarantees that the mapping between conservative and entropy variables is invertible.  

It can be shown (see, for example, \cite{mock1980systems}) that (\ref{eq:entropysym}) is equivalent to the existence of an entropy flux function $F(\bm{u})$ and entropy potential $\psi$ such that
\begin{equation}
%\psi_k(\bm{v}) = \bm{v}^T\bm{f}_k(\bm{u}(\bm{v})) - F_k(\bm{u}(\bm{v})), \qquad \psi_k'(\bm{v}) = \bm{f}_k(\bm{u}(\bm{v})).
\bm{v}^T \pd{\bm{f}}{\bm{u}} = \pd{F(\bm{u})}{\bm{u}}^T, \qquad \psi(\bm{v}) = \bm{v}^T\bm{f}(\bm{u}(\bm{v})) - F(\bm{u}(\bm{v})), \qquad \psi'(\bm{v}) = \bm{f}(\bm{u}(\bm{v})).
\end{equation}

When $\bm{u}$ is smooth, multiplying (\ref{eq:pde}) on the left by $\bm{v}^T = U'(\bm{u})^T$, applying the definition of the entropy flux and using the chain rule yields the conservation of entropy
\begin{equation}
\pd{U(\bm{u})}{t} + \pd{F(\bm{u})}{{x}} = 0.
\label{eq:consentropystrong}
\end{equation}
\noteTwo{We assume now that the domain is the interval $[-1,1]$.  Integrating (\ref{eq:consentropystrong}) over this interval and using the definition of the entropy potential then yields a statement of entropy conservation for smooth solutions
\begin{equation}
\int_{-1}^1 \pd{U(\bm{u})}{t} + \left.\LRp{\bm{v}^T\bm{f}(\bm{u}(\bm{v}))-\psi(\bm{v})}\right|_{-1}^1 = 0.  
\label{eq:consentropy}
\end{equation}}

More generally, it can be shown that physically relevant solutions of (\ref{eq:pde}) (defined as the limiting solution for an appropriately defined vanishing viscosity) satisfy the inequality
\begin{equation}
\pd{U(\bm{u})}{t} + \pd{F(\bm{u})}{{x}} \leq 0.  
\end{equation}
\noteTwo{Integrating over $[-1,1]$ then yields a more general statement of entropy inequality
\begin{equation}
\int_{-1}^1 \pd{U(\bm{u})}{t} + \left.\LRp{\bm{v}^T\bm{f}(\bm{u}(\bm{v}))-\psi(\bm{v})}\right|_{-1}^1 \leq 0.  
\label{eq:ineqentropy}
\end{equation}}
\noteTwo{The focus of this work is the construction of {high order polynomial DG} methods which satisfy a discrete analogue of the conservation of entropy (\ref{eq:consentropy}) and the dissipation of entropy (\ref{eq:ineqentropy}).}  %We begin by reviewing existing discretely entropy conservative methods based on diagonal norm SBP operators. 

\section{Discrete differential operators and quadrature-based matrices}
\label{sec:ops}

%{In this section, we define some 
\subsection{Mathematical assumptions and notations}

We begin with a $d$-dimensional reference element $\widehat{D}$ with boundary $\partial \hat{D}$.  
We denote the $i$th component of the outward normal vector on the boundary of the reference element $\partial \hat{D}$ as $\hat{n}_i$.  For this work, we assume that $\hat{n}_i$ is constant; i.e., that the faces of the reference element are planar, which is true for most commonly used reference elements in two and three dimensions \cite{chan2015gpu}.  

We define an approximation space using degree $N$ polynomials on the reference element.  In one dimension, this space is defined as
\begin{equation}
P^N\LRp{\widehat{D}} = \LRc{\widehat{{x}}^i, \quad \widehat{x} \in \widehat{D}, \quad 0\leq i \leq N}.
\end{equation}
In higher dimensions, the choice of approximation space depends on the type of element \cite{chan2015gpu}, but generally contains the space of total degree $N$ polynomials.  We denote the dimension of the approximation space $P^N$ as $N_p = {\rm dim}\LRp{P^N\LRp{\widehat{D}}}$ (with $N_p = N+1$ in one dimension).  

We define the $L^2$ norm and inner products over the reference element $\hat{D}$ and the surface of the reference element $\partial \hat{D}$
\begin{equation}
\LRp{\bm{u},\bm{v}}_{\hat{D}} = \int_{\hat{D}} \bm{u}\cdot\bm{v}\diff{\bm{x}} =  \int_{\widehat{D}} \bm{u}\cdot\bm{v} J^k \diff{\widehat{x}}, \qquad \nor{\bm{u}}^2_{\hat{D}} = (\bm{u},\bm{u})_{\hat{D}}, \qquad \LRa{\bm{u},\bm{v}}_{\partial \hat{D}} = \int_{\partial \hat{D}} \bm{u} \cdot \bm{v} \diff{\bm{x}},
\end{equation}

%In one dimension, $D^k$ is some interval and the $L^2$ inner product over the surface $\LRc{-1,1}$ reduces to 
%\begin{equation}
%\LRa{u,v}_{\partial D^k} = u(x_L)v(x_L) + u(x_R)v(x_R).
%\end{equation}  
%However, the presented proofs in this work will not directly utilize this simplification in order to remain generalizable to higher dimensions in a straightforward manner.  

\subsection{{Interpolation and differentiation matrices}}
\label{sec:matrix}

In most implementations, integrals and $L^2$ inner products are approximated using a quadrature rule which is exact for polynomials of a certain degree.  This defines a discrete $L^2$ inner product, which in turn can be used to construct operators which obey a property resembling summation-by-parts for diagonal norm SBP operators.
%  consisting of $N_q \geq N_p$ points $\bm{x}_i$ and weights $w_i$ over the reference element $\widehat{D}$, such that
%\begin{equation}
%\LRp{u,v}_{D^k} \coloneqq \sum_{i=1}^{N_q} u(\widehat{x}_i)v(\widehat{x}_i)w_i J^k.     
%\end{equation}
%This also defines an analogous quadrature-based $L^2$ projection operator.  

{We now introduce several quadrature-based matrices for the $d$-dimensional reference element $\widehat{D}$, which we will use to construct matrix-vector formulations of DG methods.   Assuming $u(\bm{x}) \in P^N\LRp{\widehat{D}}$, it can be represented in some polynomial basis $\phi_i$ of degree $N$ and dimension $N_p$ in terms of the vector of coefficients $\bm{u}$
\begin{equation}
u(\bm{x}) = \sum_{j=1}^{N_p}\bm{u}_j \phi_j(\widehat{\bm{x}}), \qquad P^N\LRp{\widehat{D}} = {\rm span}\LRc{\phi_i(\widehat{x})}_{i=1}^{N_p}.
\end{equation}
The construction of these quadrature-based matrices utilizes $\phi_i$, as well as volume and surface quadrature rules with $N_q$ and $N^f_q$ points, respectively.  We make the following assumptions on the strength of the quadrature rules:
\begin{assumption}
The volume quadrature rule $\LRc{(\bm{x}_i, w_i)}_{i=1}^{N_q}$ exactly integrates polynomials of degree at least $(2N-1)$ on the reference element $\hat{D}$, and the surface quadrature $\LRc{(\bm{x}^f_i, w^f_i)}_{i=1}^{N^f_q}$ exactly integrates polynomials of at least degree $2N$ on the boundary of the reference element $\partial \hat{D}$.  
\label{ass:quad}
\end{assumption}
These assumptions impose a minimum strength of the quadrature rule; however, the polynomial degree for which these quadrature rules are accurate can be taken arbitrarily high.  These conditions are imposed to ensure that integration by parts holds for any two polynomials of degree $N$ when integrals are approximated using quadrature.\footnote{We note that these conditions are sufficient, but not always necessary.  For example, appropriate SBP operators can be constructed using tensor product Gauss-Legendre-Lobatto quadratures, though the GLL surface quadrature is only exact for degree $2N-1$ polynomials.  More general conditions can be formulated by requiring that integration by parts holds when volume and surface integrals are approximated using volume and surface quadrature rules.}

Let $\bm{W} \in \mathbb{R}^{N_q\times N_q}$ denote the diagonal matrix whose entries are quadrature weights
\begin{equation}
\bm{W}_{ij} = \begin{cases}
w_i & i=j\\
0 & \text{otherwise}.
\end{cases}
\end{equation}
We also define $\bm{W}_f$ as the diagonal matrix of surface quadrature weights.  We define the quadrature interpolation matrix $\bm{V}_q$ 
\begin{equation}
\LRp{\bm{V}_q}_{ij} = \phi_j(\bm{x}_i), \qquad 1 \leq j \leq N_p, \qquad 1 \leq i \leq N_q,
\end{equation}
which maps coefficients $\bm{u}$ to evaluations of $u$ at quadrature points
\begin{equation}
\bm{u}_q = \bm{V}_q \bm{u}, \qquad \LRp{\bm{u}_q}_i = u(x_i), \quad 1 \leq i \leq N_q.
\end{equation}
Let $\bm{V}_f$ denote the matrix which interpolates to boundary values
\footnote{
In one dimension, $\hat{D} = [-1,1]$, and the surface quadrature rule $(\bm{x}^f_i, w^f_i)$ consists only of boundary points $-1,1$, with normal directions $\hat{n} = \pm 1$ and both weights $w^f_1 = w^f_2 = 1$, such that $\bm{W}_f$ is simply the $2\times 2$ identity matrix.  In one dimension, $\bm{V}_f$ reduces to
\begin{equation}
\LRp{\bm{V}_f}_{1j} = \phi_j(-1), \qquad \LRp{\bm{V}_f}_{2j} = \phi_j(1), \qquad 1 \leq j \leq N_p.  
\end{equation}}
\begin{equation}
\noteTwo{\LRp{\bm{V}_f}_{ij} = \phi_j(\hat{\bm{x}}^f_i), \qquad 1 \leq j \leq N_p, \qquad 1 \leq i \leq N^f_q},
\end{equation}
Next, let ${\bm{D}}_i$ denote the differentiation matrix with respect to the $i$th coordinate, defined implicitly through
\begin{equation}
u(\hat{\bm{x}}) = \sum_{j=1}^{N_p} \bm{u}_j \phi_j(\hat{\bm{x}}), \qquad \pd{u}{\hat{\bm{x}}_i} = \sum_{j=1}^{N_p} \LRp{{\bm{D}}_i \bm{u}}_j\phi_j(\hat{\bm{x}})
\end{equation}
In other words, ${\bm{D}}_i$ maps basis coefficients of some polynomial $u \in P^N$ to coefficients of its $i$th derivative with respect to the reference coordinate $\hat{\bm{x}}$, and is sometimes referred to as a ``modal''\footnote{The term ``modal'' refers to bases which are not necessarily defined in terms of nodal points, and should not be confused with modal bases which are orthonormal with respect to an $L^2$ inner product.} differentiation matrix with respect to a general non-nodal basis \cite{hicken2016multidimensional}.  

\subsection{Quadrature-based projection matrices and lifting matrices}

Given $\bm{V}_q$, we introduce the mass matrix, whose entries are the evaluation of inner products of different basis functions using quadrature
\begin{equation}
\bm{M} = \bm{V}_q^T\bm{W}\bm{V}_q, \qquad \bm{M}_{ij} = \noteTwo{\sum_{k=1}^{N_q} w_k \phi_j(\hat{\bm{x}}_k)\phi_i(\hat{\bm{x}}_k) \approx \int_{\hat{D}}\phi_j\phi_i \diff{\hat{\bm{x}}} = \LRp{\phi_j,\phi_i}_{\hat{D}}}.
\end{equation}
The approximation in the formula for the mass matrix becomes an equality if the volume quadrature rule is exact for polynomials of degree $2N$.  The mass matrix is symmetric and positive definite under Assumption~\ref{ass:quad}, and is also referred to in the SBP literature as a ``norm'' matrix, with distinctions made between dense and diagonal norm matrices.  We do not make any distinctions between diagonal and dense $\bm{M}$ in this work.  

The mass matrix appears in the computation of the $L^2$ projection both when integrals are computed exactly and when integrals are computed using quadrature.  The continuous $L^2$ projection operator is defined as $\Pi_N: L^2\LRp{\hat{D}}\rightarrow P^N\LRp{\hat{D}}$ such that 
\begin{equation}
\LRp{\Pi_N f,v}_{\hat{D}} = \LRp{f,v}_{\hat{D}}, \qquad \forall v\in P^N\LRp{\hat{D}}.  
\label{eq:l2proj}
\end{equation}
In other words, $\Pi_N$ is the operator which maps an $L^2$ integrable function $f$ to a polynomial $\Pi_N f  \in P^N\LRp{\hat{D}}$.  Assuming some polynomial basis \noteTwo{$\phi_j(\hat{\bm{x}})$} for $P^N$, the $L^2$ projection reduces to the determination of coefficients of $\Pi_N f$ in the \noteTwo{$\phi_j(\hat{\bm{x}})$}.  Additionally, when integrals within the $L^2$ inner products present in (\ref{eq:l2proj}) are computed using quadrature, the \noteTwo{discrete quadrature-based} $L^2$ projection of a function $f(x)$ (\ref{eq:l2proj}) can be reduced to the following matrix problem:
\begin{equation}
\bm{M} \bm{u} = \bm{V}_q^T\bm{W}\bm{f}, \qquad \bm{f}_i = f(\hat{\bm{x}}_i), \quad i = 1,\ldots,N_q.
\end{equation}
where $u$ is the vector of coefficients of the \noteTwo{quadrature-based} $L^2$ projection of $\bm{f}$.  Inverting the mass matrix allows us to define the quadrature-based $L^2$ projection matrix $\bm{P}_q$ as a discretization of the $L^2$ projection operator $\Pi_N$
\begin{equation}
\bm{P}_q = \bm{M}^{-1}\bm{V}_q^T\bm{W}.
\end{equation}
The matrix $\bm{P}_q$ which maps a function (in terms of its evaluation at quadrature points) to coefficients of the $L^2$ projection in the basis $\phi_j(\hat{\bm{x}})$.  Note that, since $\bm{M} = \bm{V}_q^T\bm{W}\bm{V}_q$, 
\begin{equation}
\bm{P}_q\bm{V}_q = \bm{M}^{-1}\bm{V}_q^T\bm{W}\bm{V}_q = \bm{I}.  
\label{eq:pqvq}
\end{equation}
In other words, the projection operator reduces to the identity matrix when applied to any linear combination of basis functions $\phi_j(\hat{\bm{x}})$ evaluated at quadrature points.  This implies that, when applied to the evaluation of any polynomial at volume quadrature points, the matrix $\bm{P}_q$ simply returns the coefficients of the polynomial in the basis $\phi_j(\hat{\bm{x}})$.\footnote{It is worth noting that $\bm{V}_q\bm{P}_q \neq \bm{I}$ in general.  However, if both matrices are square such that $N_p = N_q$ and the number of basis functions coincides with the number of quadrature points, then the mass matrix inverse can be explicitly written as $\bm{V}_q^{-T}\bm{W}^{-1}\bm{V}_q^{-1}$, and we can simplify $\bm{V}_q\bm{P}_q = \bm{V}_q\bm{M}^{-1}\bm{V}_q^T\bm{W} = \bm{I}$.  When $N_q > N_p$, the matrix $\bm{V}_q$ cannot be inverted and the mass matrix does not have an explicit inverse in terms of $\bm{V}_q$. }

We also introduce the quadrature-based lifting matrix 
\begin{equation}
\bm{L}_q = \bm{M}^{-1}\bm{V}_f^T \bm{W}_f,
\end{equation}
which ``lifts'' a function (evaluated at surface quadrature points) from the boundary of an element to coefficients of a basis defined in the interior of the element.  This is a quadrature-based discretization of the lifting operator  ${L}: L^2\LRp{\partial \hat{D}} \rightarrow P^N$ \cite{hesthaven2007nodal, di2011mathematical}
\begin{equation}
\LRp{L u,v}_{\hat{D}} = \LRa{u,v}_{\partial \hat{D}}, \qquad \forall v \in P^N.
\end{equation}

%The lift operator can be used to rewrite surface terms which involve the volume $L^2$ projection
%\begin{equation}
%\LRa{u,\Pi_N(vw)}_{\partial D^k} = \LRp{Lu,\Pi_N (vw)}_{D^k} = \LRp{Lu,vw}_{D^k}, \qquad u,v\in V_h, 
%\end{equation}
%where $w(x)$ is smooth and bounded over $D^k$.  

\subsection{Quadrature-based differentiation matrices and a ``decoupled'' SBP-like operator}

The projection and lifting matrices map between function values at volume or surface quadrature points and approximations which can be represented in the basis $\LRc{\phi_i}_{i=1}^{N_p}$.  By combining them with the polynomial differentiation matrix ${\bm{D}}_i$, we can construct a differencing operator for functions defined at quadrature points.  For example, given function evaluations at volume quadrature points, we can project the function to $P^N$ using $\bm{P}_q$, differentiate the resulting polynomial, and evaluate the result at quadrature points.  This sequence of operations can be concisely expressed as the product of three matrices
\begin{equation}
\bm{D}^i_q = \bm{V}_q \bm{D}_i \bm{P}_q.  
\label{eq:Dq}
\end{equation}
Define $\diag{\bm{u}}$ as the diagonal matrix with the entries of $\bm{u}$ on the diagonal.  
The quadrature-based differentiation matrix $\bm{D}^i_q$ obeys the following lemma:
\begin{lemma}
The operator $\bm{D}^i_q = \bm{V}_q \bm{D}_i \bm{P}_q$ satisfies the SBP property with respect to the diagonal matrix $\bm{W}$ such that
\begin{equation}
\bm{W}\bm{D}^i_q + \LRp{\bm{W}\bm{D}^i_q}^T = \bm{P}_q^T\bm{V}_f^T \bm{W}_f{\rm diag}(\hat{\bm{n}}_i) \bm{V}_f\bm{P}_q.
\end{equation}
Additionally, $\bm{D}^i_q$ is a degree $N$ approximation to the derivative $\pd{}{\hat{\bm{x}}_i}$.  
\label{lemma:qsbp}
\end{lemma}
\begin{proof}
We first note that the ``modal'' differentiation matrix $\bm{D}_i$ satisfies the following property with respect to the mass matrix
\begin{equation}
\bm{M}\bm{D}_i  + \bm{D}_i^T\bm{M} = \bm{V}_f^T \bm{W}_f{\rm diag}(\hat{\bm{n}}_i) \bm{V}_f.
\label{eq:modalsbp}
\end{equation}
This is simply a restatement of integration by parts for polynomials \cite{hesthaven2007nodal, chan2015gpu, hicken2016multidimensional} 
\begin{equation}
\int_{\hat{D}} \pd{\phi_j}{\hat{\bm{x}}_i}\phi_i \diff{\hat{\bm{x}}} = 
\int_{\partial \hat{D}} \phi_j\phi_i \hat{n}_i \diff{\hat{\bm{x}}} - \int_{\hat{D}} \phi_j \pd{\phi_i}{\hat{\bm{x}}_i} \diff{\hat{\bm{x}}} 
\end{equation}
using the fact that the above integrals are exact for volume quadratures of degree $2N-1$ and surface quadratures of degree $2N$, which we have from Assumption~\ref{ass:quad}.  Multiplying equation (\ref{eq:modalsbp}) by $\bm{P}_q^T$ on the left, $\bm{P}_q$ on the right, and using $\bm{P}_q = \bm{M}^{-1}\bm{V}_q^T\bm{W}$ yields
\begin{align}
\bm{P}_q^T\bm{M}\bm{D}_i\bm{P}_q  + \bm{P}_q^T\bm{D}_i^T\bm{M}\bm{P}_q^T &= \bm{W}\bm{V}_q\bm{D}_i\bm{P}_q
+ \bm{P}_q^T\bm{D}_i^T\bm{V}_q^T\bm{W} \nonumber\\
&= \bm{W}\bm{D}^i_q + \LRp{\bm{W}\bm{D}^i_q}^T  = \bm{P}_q^T\bm{V}_f^T \bm{W}_f{\rm diag}(\hat{\bm{n}}_i) \bm{V}_f\bm{P}_q.
\end{align}
The accuracy of $\bm{D}^i_q$ results from the fact that $\bm{D}^i_q$ recovers the exact derivative of polynomials up to degree $N$.  Let $\bm{u}_q$ be the values of some polynomial $u \in P^N$ at quadrature points, such that $\bm{u}_q = \bm{V}_q\bm{u}$ for coefficients $\bm{u}$.  Then by (\ref{eq:pqvq}),
\begin{equation}
\bm{D}^i_q\bm{u}_q = \bm{V}_q\bm{D}_i\bm{P}_q\bm{V}_q\bm{u} = \bm{V}_q\bm{D}_i\bm{u}
\end{equation}
which is the evaluation of the exact derivative of $u(\bm{x})$ at quadrature points.
\end{proof}
Lemma~\ref{lemma:qsbp} shows how the projection matrix $\bm{P}_q$ can transform a ``modal'' SBP property (involving the norm matrix $\bm{M}$, which can be dense) to a quadrature-based SBP property involving the diagonal norm matrix $\bm{W}$.  The boundary term in the resulting SBP property includes the matrix $\bm{V}_f\bm{P}_q$, which can be interpreted as taking the projection of a function (defined through values at volume quadrature points) and evaluating the result at surface quadrature points.  

We can also use the quadrature-based differentiation operator $\bm{D}^i_q$ to define a ``decoupled'' operator $\bm{D}^i_N$  which maps from a vector of both volume and surface quadrature points to values at both volume and surface quadrature points
\begin{align}
\bm{D}^i_N  &= \LRs{
\begin{array}{cc}
\bm{D}^i_q - \frac{1}{2}\bm{V}_q \bm{L}_q {\rm diag}(\widehat{\bm{n}}_i) \bm{V}_f\bm{P}_q &  \frac{1}{2}\bm{V}_q\bm{L}_q{\rm diag}(\widehat{\bm{n}}_i)\\
-\frac{1}{2}{\rm diag}(\hat{\bm{n}}_i)\bm{V}_f\bm{P}_q & \frac{1}{2}{\rm diag}(\widehat{\bm{n}}_i)
\end{array}},
\label{eq:DN}
\end{align}
where $\hat{\bm{n}}_i$ is the vector containing the $i$th normal component $\hat{n}_i$ evaluated at surface quadrature points.  

Let $\bm{W}_N$ be defined as the matrix of both volume and surface quadrature weights
\begin{equation}
\bm{W}_N = \LRp{
\begin{array}{cc}
\bm{W} & \\
& \bm{W}_f
\end{array}
}.
\label{eq:WN}
\end{equation}
We can show that the ``weak'' or integrated version of the differentiation matrix $\bm{D}^i_N$ satisfies the following properties:
\begin{theorem}
The matrix $\bm{Q}^i_N = \bm{W}_N \bm{D}^i_N$ satisfies the SBP-like property
\begin{equation}
\bm{Q}^i_N + \LRp{\bm{Q}^i_N}^T = \bm{B}^i_N, \qquad \bm{B}^i_N = 
\LRp{\begin{array}{cc}
0 & \\
& \bm{W}_f {\rm diag}(\hat{\bm{n}}_i)
\end{array}}.
\end{equation}
Additionally, $\bm{Q}^i_N\bm{1} = 0$, where $\bm{1}$ is the vector of all ones.  
\label{thm:sbp}
\end{theorem}
\begin{proof}
The matrix $\bm{Q}^i_N$ is given explicitly as
\begin{align}
\bm{Q}^i_N  &= \LRs{
\begin{array}{cc}
\bm{W}\bm{D}^i_q - \frac{1}{2}\bm{W}\bm{V}_q \bm{L}_q {\rm diag}(\widehat{\bm{n}}_i) \bm{V}_f\bm{P}_q &  \frac{1}{2}\bm{W}\bm{V}_q\bm{L}_q{\rm diag}(\widehat{\bm{n}}_i)\\
-\frac{1}{2}\bm{W}_f{\rm diag}(\widehat{\bm{n}})\bm{V}_f\bm{P}_q & \frac{1}{2}\bm{W}_f{\rm diag}(\widehat{\bm{n}}_i)
\end{array}}.
\end{align}
The bottom right block of $\bm{Q}^i_N + \LRp{\bm{Q}^i_N}^T$ is $\bm{W}_f{\rm diag}(\widehat{\bm{n}}_i)$, as both $\bm{W}_f$ and ${\rm diag}(\widehat{\bm{n}}_i)$ are diagonal and symmetric.  We will show that all remaining blocks of $\bm{Q}^i_N + \LRp{\bm{Q}^i_N}^T$ are zero.  We first deal with the off-diagonal blocks, which are equal to
\begin{equation}
\frac{1}{2}\bm{W}\bm{V}_q\bm{L}_q{\rm diag}(\widehat{\bm{n}}_i) - \frac{1}{2}\LRp{\bm{W}_f{\rm diag}(\widehat{\bm{n}})\bm{V}_f\bm{P}_q}^T
\end{equation}
and its transpose.  These off-diagonal blocks reduce to zero by noting that
\begin{align}
\bm{W}\bm{V}_q\bm{L}_q{\rm diag}(\widehat{\bm{n}}_i) &= \bm{W}\bm{V}_q\bm{M}^{-1}\bm{V}_f \bm{W}_f {\rm diag}(\widehat{\bm{n}}_i)\\
&=  \bm{P}_q^T \bm{V}_f \bm{W}_f {\rm diag}(\widehat{\bm{n}}_i)= \LRp{\bm{W}_f{\rm diag}(\widehat{\bm{n}})\bm{V}_f\bm{P}_q}^T.\nonumber
\end{align}

To show the top left block of $\bm{Q}^i_N +\LRp{\bm{Q}^i_N}^T$ is zero, we first use the fact that $\bm{L}_q = \bm{M}^{-1}\bm{V}_f^T\bm{W}_f$ to rewrite Lemma~\ref{lemma:qsbp} as
\begin{align}
\bm{W}\bm{D}^i_q &= \bm{P}_q^T\bm{V}_f^T \bm{W}_f{\rm diag}(\hat{\bm{n}}_i) \bm{V}_f\bm{P}_q - \LRp{\bm{W}\bm{D}^i_q}^T \nonumber\\
&= \bm{W} \bm{V}_q \bm{M}^{-1}\bm{V}_f^T \bm{W}_f {\rm diag}(\hat{\bm{n}}_i) \bm{V}_f\bm{P}_q - \LRp{\bm{W}\bm{D}^i_q}^T = \bm{W} \bm{V}_q \bm{L}_q {\rm diag}(\hat{\bm{n}}_i) \bm{V}_f\bm{P}_q - \LRp{\bm{W}\bm{D}^i_q}^T.
\label{eq:thmsbp}
\end{align}
Next, we show that
\begin{equation}
\bm{W} \bm{V}_q \bm{L}_q {\rm diag}(\hat{\bm{n}}_i) \bm{V}_f\bm{P}_q = \bm{W} \bm{V}_q\bm{M}^{-1} \bm{V}_f^T\bm{W}_f {\rm diag}(\hat{\bm{n}}_i) \bm{V}_f\bm{M}^{-1}\bm{V}_q^T\bm{W},
\label{eq:thmsym}
\end{equation}
from which we can see that $\bm{W} \bm{V}_q \bm{L}_q {\rm diag}(\hat{\bm{n}}_i) \bm{V}_f\bm{P}_q$ is symmetric.  Combining (\ref{eq:thmsbp}) and (\ref{eq:thmsym}), we have that the top left block of $\bm{Q}^i_N +\LRp{\bm{Q}^i_N}^T$ is also zero
\begin{align}
&\bm{W}\bm{D}^i_q - \frac{1}{2}\bm{W}\bm{V}_q \bm{L}_q {\rm diag}(\widehat{\bm{n}}_i) \bm{V}_f\bm{P}_q + \LRp{\bm{W}\bm{D}^i_q - \frac{1}{2}\bm{W}\bm{V}_q \bm{L}_q {\rm diag}(\widehat{\bm{n}}_i) \bm{V}_f\bm{P}_q}^T \nonumber\\
&= \bm{W}\bm{D}^i_q - \bm{W}\bm{V}_q \bm{L}_q {\rm diag}(\widehat{\bm{n}}_i) \bm{V}_f\bm{P}_q + \LRp{\bm{W}\bm{D}^i_q}^T = 0.
 \end{align}
 
 Showing $\bm{Q}^i_N \bm{1} = 0$ is equivalent to showing $\bm{D}^i_N\bm{1} = 0$.  Direct multiplication gives
 \begin{align}
\bm{D}^i_N \bm{1}
&= \LRs{
\begin{array}{c}
\bm{D}^i_q\bm{1} - \frac{1}{2}\bm{V}_q \bm{L}_q {\rm diag}(\widehat{\bm{n}}_i) \bm{V}_f\bm{P}_q \bm{1} + \frac{1}{2}\bm{V}_q\bm{L}_q{\rm diag}(\widehat{\bm{n}}_i)\bm{1} \\
-\frac{1}{2}{\rm diag}(\hat{\bm{n}}_i)\bm{V}_f\bm{P}_q\bm{1} + \frac{1}{2}{\rm diag}(\widehat{\bm{n}}_i)
\end{array}} = 0.
\end{align}
Here, we have used that (\ref{eq:pqvq}) implies $\bm{V}_f\bm{P}_q\bm{1} = \bm{1}$, and that $\bm{D}^i_q\bm{1} = 0$ by Lemma~\ref{lemma:qsbp}.  
\end{proof}

%We emphasize that, throughout this work, we will only use properties of quadrature-based inner products and projection operators.  We assume only that the quadrature is sufficiently accurate to integrate degree $2N-1$ polynomials, such that integration by parts holds for any two polynomials on the reference element, while allowing for quadratures with an arbitrary number of points.  

\subsection{Discrete differentiation operators and approximating weighted derivatives}

While $\bm{D}^i_N$ satisfies an SBP-like property, it cannot be used directly as a differentiation operator.  To see this, let $\bm{u}$ denote the coefficients of some function $u \in P^N$.  Then, $\bm{V}_q \bm{u}$ and $\bm{V}_f \bm{u}$ are the evaluations of $u(\bm{x})$ at volume and surface quadrature points, respectively.  We expect a high order accurate derivative operator to exactly differentiate polynomials.  However, applying $\bm{D}^i_N$ to the concatenated vector containing evaluations at both volume and surface quadrature points yields
\begin{align}
\bm{D}^i_N \LRs{\begin{array}{c}\bm{V}_q\bm{u}\\ \bm{V}_f\bm{u}\end{array}}
&= \LRs{
\begin{array}{cc}
\bm{D}^i_q - \frac{1}{2}\bm{V}_q \bm{L}_q {\rm diag}(\widehat{\bm{n}}_i) \bm{V}_f\bm{P}_q &  \frac{1}{2}\bm{V}_q\bm{L}_q{\rm diag}(\widehat{\bm{n}}_i)\\
-\frac{1}{2}{\rm diag}(\hat{\bm{n}}_i)\bm{V}_f\bm{P}_q & \frac{1}{2}{\rm diag}(\widehat{\bm{n}}_i)
\end{array}}
\LRs{\begin{array}{c}\bm{V}_q\\ \bm{V}_f\end{array}}\bm{u} \nonumber\\
&= \LRs{
\begin{array}{c}
\bm{V}_q\bm{D}_i\bm{P}_q\bm{V}_q - \frac{1}{2}\bm{V}_q \bm{L}_q {\rm diag}(\widehat{\bm{n}}_i) \bm{V}_f\bm{P}_q\bm{V}_q +  \frac{1}{2}\bm{V}_q\bm{L}_q{\rm diag}(\widehat{\bm{n}}_i) \bm{V}_f\\
-\frac{1}{2}{\rm diag}(\hat{\bm{n}}_i)\bm{V}_f\bm{P}_q\bm{V}_q + \frac{1}{2}{\rm diag}(\widehat{\bm{n}}_i) \bm{V}_f
\end{array}}\bm{u} = \LRs{
\begin{array}{c}
\bm{V}_q\bm{D}_i\bm{u} \\
0
\end{array}}.
\end{align}
where have used that $\bm{P}_q\bm{V}_q = \bm{I}$ from (\ref{eq:pqvq}).  This indicates that the rows of $\bm{D}^i_N$ corresponding to volume quadrature points recover the exact derivatives of a polynomial, but that the rows of $\bm{D}^i_N$ corresponding to surface quadrature points return zero when applied to a polynomial.  Thus, while $\bm{D}^i_N$ has an SBP-like property, it is not an SBP operator due to the fact that it is not a high order accurate approximation of the derivative \cite{fernandez2014generalized}.  

However, while $\bm{D}^i_N$ by itself is not a differentiation operator, we can recover the polynomial differentiation operator $\bm{D}_i$ by contracting the output of $\bm{D}^i_N$ with projection and lifting matrices.  Straightforward computations show that
\begin{equation}
\bm{D}_i = \LRs{\begin{array}{cc}
\bm{P}_q & \bm{L}_q\end{array}} \bm{D}^i_N \LRs{\begin{array}{c}
\bm{V}_q\\ \bm{V}_f\end{array}}.
\label{eq:Drecovery}
\end{equation}
Motivated by this fact, let $u(\bm{x})$ be a non-polynomial function, and let $[\bm{u}_q, \bm{u}_f]^T$ be the vector whose entries $\bm{u}_q, \bm{u}_f$ are the evaluations of $u$ at volume and surface quadrature points respectively.  A polynomial approximation of the derivative of $u$ can be computed by first applying $\bm{D}^i_N$ to $[\bm{u}_q, \bm{u}_f]^T$, then applying the projection and lifting matrices to the result 
\begin{equation}
\tilde{\bm{u}} = \LRs{\begin{array}{cc}
\bm{P}_q & \bm{L}_q\end{array}} \bm{D}^i_N \LRs{\begin{array}{c}\bm{u}_q\\ \bm{u}_f\end{array}}, \qquad \sum_{j=1}^{N_p} \tilde{\bm{u}}_j \phi_j(\bm{x}) \approx \pd{u}{\hat{\bm{x}}_i}.
\end{equation}
In the above case, the application of the projection and lifting matrices can be combined with the application of $\bm{D}^i_N$ into a single matrix-vector product.  However, applying $\bm{P}_q$ and $\bm{L}_q$ separately from $\bm{D}^i_N$ makes it possible to approximate the product of a function and a function derivative while maintaining an SBP-like property.  Let $u(\hat{\bm{x}})$ and $w(\hat{\bm{x}})$ denote two non-polynomial functions, whose evaluations at volume and surface quadrature points are denoted $[\bm{u}_q, \bm{u}_f]^T$ and $[\bm{w}_q, \bm{w}_f]^T$, respectively.  \noteOne{The utility of the operator $\bm{D}^i_N$ is that, when combined with $\bm{P}_q,\bm{L}_q$, it can be used to construct some polynomial with coefficients $\tilde{\bm{u}}_w$ which} approximates the product $w(\hat{\bm{x}}) \pd{u}{\hat{\bm{x}}_i}$ 
\begin{equation}
\tilde{\bm{u}}_w = \LRs{\begin{array}{cc}
\bm{P}_q & \bm{L}_q\end{array}} {\rm diag}\LRp{\LRs{\begin{array}{c}\bm{w}_q\\ \bm{w}_f\end{array}}}\bm{D}^i_N \LRs{\begin{array}{c}\bm{u}_q\\ \bm{u}_f\end{array}}, \qquad  \sum_{j=1}^{N_p} (\tilde{\bm{u}}_w)_j \phi_j(\hat{\bm{x}}) \approx w(\hat{\bm{x}})\pd{u}{\hat{\bm{x}}_i}.
\label{eq:wdsbp}
\end{equation}
We note that the polynomial approximation resulting from (\ref{eq:wdsbp}) is equivalent to a quadrature discretization of the following approximation of $w(\bm{x})\pd{u}{\hat{\bm{x}}_i}$ involving the continuous $L^2$ projection and lifting operators $\Pi_N$ and $L$ 
\begin{equation}
w(\hat{\bm{x}})\pd{u}{\hat{\bm{x}}_i} \approx \Pi_N \LRp{ w(\hat{\bm{x}}) \LRp{\pd{\Pi_N u}{\hat{\bm{x}}_i} + L(u - \Pi_N u)}} + L\LRp{w(u-\Pi_N u)}.
\end{equation}
\noteOne{In the context of discretizations for (\ref{eq:pde}), (\ref{eq:wdsbp}) provides a way to approximate the spatial term involving the nonlinear flux functions when it is in non-conservative form.  The connection to the entropy conservative discretization of (\ref{eq:pde}) is made explicit using Burgers' equation as an example in \ref{appendix:A}.  }

%This formulation can also be used to construct DG schemes which are discretely entropy conservative.  However, these proofs require more technical assumptions than the matrix-based proofs presented here.  The proofs based on continuous operators are provided in \ref{appendix:B} for interested readers.  
}

\section{Entropy conservative DG methods on a single element}
\label{sec:ecdg}

In this section, we focus on constructing an entropy conservative DG scheme on a single element using the matrix operators defined in Section~\ref{sec:ops}.  %These operators can be constructed for quadrature rules which satisfy Assumption~\ref{ass:quad} and general ``modal'' bases.  %These include (as special cases) dense norm SBP operators and generalized SBP (GSBP) operators which do not contain boundary points.  Difficulties in generalizing entropy conservative SBP-DG schemes present themselves primarily in two areas: the assumption of a diagonal mass matrix, which is used in the proof of entropy stability, and the generalization of flux differencing to a continuous formulation.  We address both areas of difficulty by extending the entropy conservative DG-SEM scheme (\ref{eq:dgsem}) from a discrete formulation to a continuous formulation involving flux differencing.  
 We first prove that the proposed scheme is entropy conservative on a single one dimensional element, then generalize the proof to higher dimensions.  %, and describe its application to Burgers' equation as an illustrative example.  

% to a larger class of quadratures and choices of basis.  Additionally, we prove that the generalized scheme is entropy conservative (stable) on multiple elements in the following sections.  

%We will get around these difficulties for dense-norm DG methods by satisfying three requirements for entropy stability
%\begin{itemize}
%\item $\bm{D}\bm{1} = 0$ for use of the flux formulation.  
%\item For some norm matrix $\bm{W}$, $\bm{W}\bm{D} = \bm{B}-\bm{D}^T\bm{W}$ where $\bm{B}$ is a suitably defined boundary operator.
%\item Contraction with the \textit{projection} of the entropy variables.  
%\end{itemize}

\subsection{A continuous interpretation of flux differencing}

In this section, we discuss a continuous interpretation of flux differencing \cite{fisher2013high,gassner2017br1,chen2017entropy}, which also encompasses the split form methodology proposed in \cite{gassner2016split}.  This interpretation will guide the construction of entropy conservative DG schemes.  We first introduce the definition of an entropy conservative finite volume numerical flux given by Tadmor \cite{tadmor1987numerical}:
\begin{definition}
Let $\bm{f}_S(\bm{u}_L,\bm{u}_R)$ be a bivariate function which is symmetric and consistent with the flux function $\bm{f}(\bm{u})$
\begin{equation}
\bm{f}_S(\bm{u}_L,\bm{u}_R) = \bm{f}_S(\bm{u}_R,\bm{u}_L), \qquad \bm{f}_S(\bm{u},\bm{u}) = \bm{f}(\bm{u})
\end{equation}
The numerical flux $\bm{f}_S(\bm{u}_L, \bm{u}_R)$ is entropy conservative if, for entropy variables $\bm{v}_L = \bm{v}(\bm{u}_L), \bm{v}_R = \bm{v}(\bm{u}_R)$
\begin{equation}
\LRp{\bm{v}_L - \bm{v}_R}^T \bm{f}_S(\bm{u}_L,\bm{u}_R) = (\psi_L - \psi_R), \qquad \psi_L = \psi(\bm{v}(\bm{u}_L)), \quad \psi_R = \psi(\bm{v}(\bm{u}_R)).  
\end{equation}
Similarly, a flux $\bm{f}_S(\bm{u}_L, \bm{u}_R)$ is referred to as entropy stable if $\LRp{\bm{v}_L - \bm{v}_R}^T \bm{f}_S(\bm{u}_L,\bm{u}_R) \leq (\psi_L - \psi_R)$.
\label{def:tadmor}
\end{definition}
This numerical flux can be used to construct entropy conservative finite volume methods, and was generalized in \cite{fjordholm2012arbitrarily} for the construction of high order finite volume schemes.  This flux was later used for the construction of discretely entropy conservative schemes using an approach referred to as \emph{flux differencing} \cite{fisher2013high, carpenter2014entropy, gassner2017br1, chen2017entropy}.  The key factor enabling the construction of discretely entropy conservative schemes is that (unlike the continuous proof of entropy conservation), the proof of discrete entropy conservation avoids the use of the chain rule, which does not hold in general at the discrete level.  Entropy conservation can also be extended beyond a single element by defining interface fluxes between two elements using the entropy conservative flux $f_S(u_L,u_R)$ \cite{carpenter2014entropy, gassner2017br1, chen2017entropy}.  Similarly, employing an entropy stable flux at element interfaces results in an entropy stable method which satisfies a global entropy inequality.  

%\note{MOVE AROUND----------------------------------------------}
%\note{END MOVE AROUND----------------------------------------------}

Let $\bm{f}_S(\bm{u}_L,\bm{u}_R)$ now be the symmetric and consistent two-point flux defined in (\ref{def:tadmor}).  Consider the bivariate function
\begin{equation}
\bm{f}_S\LRp{\bm{u}({x}),\bm{u}({y})}, \qquad x,{y} \in \mathbb{R}.
\end{equation}
{We consider an interpretation of flux differencing motivated by the derivative projection operator introduced by Gassner, Winters, Hindenlang, and Kopriva in \cite{gassner2017br1}}.  
\noteTwo{
Here, the two-point flux $\bm{f}_S = \bm{f}_S(\bm{u}_1,\bm{u}_2)$ is a bivariate function of $\bm{u}_1,\bm{u}_2$.  Then, using the consistency of $\bm{f}_S$, 
\begin{align}
\pd{\bm{f}(\bm{u})}{\bm{u}} &= \pd{\bm{f}_S(\bm{u},\bm{u})}{\bm{u}_1} = \LRl{\LRp{{\pd{\bm{f}_S(\bm{u}_1,\bm{u}_2)}{\bm{u}_1}} +  {\pd{\bm{f}_S(\bm{u}_1,\bm{u}_2)}{\bm{u}_2}}}}_{\bm{u}_1,\bm{u}_2 = \bm{u}} = \LRl{2\pd{\bm{f}_S(\bm{u}_1,\bm{u}_2)}{\bm{u}_1}}_{\bm{u}_1,\bm{u}_2 = \bm{u}},
\label{eq:chainrulefluxdiff}
\end{align}
where we have used the consistency of $\bm{f}_S$ in the first step and the symmetry of $\bm{f}_S$ in the last.  
Let $\bm{u}_1 = \bm{u}(x)$ and $\bm{u}_2 = \bm{u}(y)$, where $x, y$ are independent spatial coordinates over the domain.  Combining (\ref{eq:chainrulefluxdiff}) with the chain rule yields that
\begin{equation}
\pd{\bm{f}(\bm{u}(x))}{x} = 2\left.\pd{\bm{f}_S\LRp{\bm{u}(x),\bm{u}(y)}}{x}\right|_{y=x}.
\label{eq:fluxdiff}
\end{equation}
The accuracy of (\ref{eq:fluxdiff}) was shown using this same approach in \cite{chen2017entropy,crean2018entropy}.  
%Consider the differentiation of $\bm{f}_S\LRp{\bm{u}(x),\bm{u}(y)}$ along one coordinate $x$, then evaluation of the result with $y=x$
}
\noteOne{Flux differencing was first used to systematically recover split formulations \cite{gassner2016split}, which we describe in more detail in \ref{appendix:A} for Burgers' equation.  }

\subsection{{Quadrature-based operators and flux differencing}}
{
We first simplify the evaluation of (\ref{eq:fluxdiff}) using quadrature-based operators.  For simplicity of notation, we have dropped the superscript $i$ from one-dimensional operators.  The matrix $\bm{D}_q = \bm{V}_q \bm{D}\bm{P}_q$ is a discretization of the one-dimensional operator $\pd{}{\hat{x}}\Pi_N$, and maps from volume quadrature points to volume quadrature points.  Let ${f}_S(\bm{u}({x}),\bm{u}({y}))$ denote a scalar bivariate function.  Define the matrix $\bm{F}_S$ as the evaluation of ${f}_S(\bm{u}({x}),\bm{u}({y}))$ at quadrature points 
\begin{equation}
\LRp{\bm{F}_S}_{ij} = {f}_S(\bm{u}(\hat{x}_i),\bm{u}(\hat{x}_j)), \qquad 1\leq i,j \leq N_q.  
\end{equation}
The columns of the matrix-matrix product $\bm{D}_q\bm{F}_S$ then correspond to the projection and differentiation of the univariate function ${f}_S(\bm{u}(\hat{x}),\bm{u}(\hat{x}_j))$ for fixed quadrature points $\hat{x}_j$ for $j = 1,\ldots, N_q$.  Thus, evaluating (\ref{eq:fluxdiff}) at different quadrature points is equivalent to evaluating
\begin{equation}
\LRl{\pd{\Pi_N f_S(\bm{u}(\hat{x}),\bm{u}(\hat{y}))}{\hat{{x}}}}_{\hat{y}=\hat{x}_i} = \LRp{\bm{D}_q\bm{F}_S}_{ii} = \LRp{{\rm diag}\LRp{\bm{D}_q\bm{F}_S}}_i.
\label{eq:qfluxdiff}
\end{equation}
In other words, performing flux differencing and evaluating (\ref{eq:fluxdiff}) reduces to the computation of the diagonal entries of a matrix-matrix product when using quadrature-based matrices.  

We can further simplify (\ref{eq:qfluxdiff}) using the Hadamard product, which is defined as the matrix  operation $\circ$ such that
\begin{equation}
\LRp{\bm{A}\circ \bm{B}}_{ij} = \bm{A}_{ij} \bm{B}_{ij}, 
\end{equation}
where $\bm{A},\bm{B}$ are two matrices with the same row and column dimensions.  The Hadamard product obeys the following properties, whose proofs can be found in \cite{horn2012matrix}.
\begin{lemma}
Let $\bm{A},\bm{B},\bm{C}$ be square matrices of the same dimension.
\begin{enumerate}
\item The Hadamard product is commutative, with $\bm{A}\circ\bm{B} = \bm{B}\circ\bm{A}$.  
\item The Hadamard product is linear with respect to addition and the transpose operation
\begin{equation}
\LRp{\bm{A} + \bm{B}}\circ \bm{C} = \bm{A} \circ\bm{C} + \bm{B}\circ \bm{C}, \qquad  \LRp{\bm{A}\circ \bm{B}}^T = \bm{A}^T\circ \bm{B}^T.
\end{equation}
\item The Hadamard product is related to the ``$\rm diag$'' operation as follows:
\begin{equation}
\diag{\bm{A}\bm{B}^T} = \LRp{\bm{A}\circ \bm{B}}\bm{1}.
\end{equation}
where $\bm{1}$ is the vector of all ones.  
\end{enumerate}
\label{lemma:hadamard}
\end{lemma}
Using part 2 of Lemma~\ref{lemma:hadamard} and the fact that $\bm{F}_S$ is symmetric, we rewrite the ``$\rm diag$'' operation as the row sum of a Hadamard product, or equivalently the multiplication of a Hadamard product with the vector of all ones
\begin{equation}
{\rm diag}\LRp{\bm{D}_q\bm{F}_S} = \LRp{\bm{D}_q\circ \bm{F}_S}\bm{1}.  
\label{eq:qfluxdiffhadamard}
\end{equation}
This yields a generalization of the derivative projection operator introduced in \cite{gassner2017br1}.

%We note that rewriting (\ref{eq:qfluxdiff}) using the Hadamard product is also more computationally efficient to evaluate, as 

\subsection{{A one-dimensional entropy conservative DG method on a single element}}

Motivated by the observations in the previous section, we now construct a semi-discretely entropy conservative formulation on a single element in one spatial dimension based on the interpretation of flux differencing (\ref{eq:fluxdiff}) in the previous section.  This formulation involves projection and lifting matrices along with the decoupled SBP-like operator $\bm{D}_N$ used in (\ref{eq:wdsbp}).  We seek degree $N$ polynomial approximations $\bm{u}_N({x},t)$ to the conservative variables $\bm{u}({x},t)$, with coefficients $\bm{u}_h(t)$ such that
\begin{equation}
\bm{u}_N(\hat{\bm{x}},t) = \sum_{j=1}^{N_p} \LRp{\bm{u}_h(t)}_j \phi_j(\hat{\bm{x}}), \qquad \LRp{\bm{u}_h(t)}_j \in \mathbb{R}^n.
\end{equation}
Because $\bm{u}_h$ consists of vectors of coefficients for each scalar component of $\bm{u}_N(\hat{\bm{x}},t)$, from this point onward we understand the product of matrices applied to component-wise vectors like $\bm{u}_h$ in a Kronecker product sense (as in \cite{chen2017entropy}), e.g.\ $\bm{A} \bm{u}_h$ should be understood as applying $\bm{A}$ to each component of $\bm{u}_h$, or applying $\bm{A}\otimes\bm{I}$ (with $\bm{I}$ the ${n\times n}$ identity matrix) to the full vector $\bm{u}_h$.  

We now introduce $\bm{v}_h$ as the $L^2$ projection of the entropy variables and $\tilde{\bm{u}}$ as the evaluation (at volume and surface quadrature points) of the conservative variables in terms of the $L^2$ projected entropy variables.  
\begin{align}
&\bm{u}_q = \bm{V}_q\bm{u}_h,& &\bm{v}_q = \bm{v}\LRp{\bm{u}_q},&  &\bm{v}_h = \bm{P}_q \bm{v}_q,&\nonumber\\
&\tilde{\bm{v}} = \LRs{\begin{array}{c}
\tilde{\bm{v}}_q\\
\tilde{\bm{v}}_f
 \end{array}} = 
\LRs{\begin{array}{c}
 \bm{V}_q \\
 \bm{V}_f 
 \end{array}}\bm{v}_h,& &\tilde{\bm{u}} = \LRs{\begin{array}{c}
\tilde{\bm{u}}_q\\
\tilde{\bm{u}}_f
 \end{array}} = \bm{u}\LRp{\tilde{\bm{v}}}.&
% %, \qquad
% \LRs{\begin{array}{c}
%\bm{u}^q_{\bm{v}}\\
%\tilde{\bm{u}}_f
%\end{array}}
\label{eq:uvvars}
\end{align}
Here, $\bm{u}_q, \bm{v}_q$ denote the conservative variables and entropy variables (as a function of the conservation variables) evaluated at volume quadrature points.  The vector $\tilde{\bm{v}}$ denotes the evaluations of the $L^2$ projection of the entropy variables at volume and surface quadrature points, while $\tilde{\bm{u}}$ denotes the evaluation of the conservative variables in terms of the projected entropy variables $\bm{u}\LRp{\Pi_N\bm{v}}$, which will be crucial to proving discrete conservation of entropy.  For the remainder of this paper, $\bm{u}\LRp{\Pi_N\bm{v}}$ and $\tilde{\bm{u}}$ will be referred to as entropy-projected conservative variables.  \note{We note that this approach closely resembles that of \cite{parsani2016entropy}, where entropy stable schemes were constructed for generalized SBP operators based on Gauss nodes.  In \cite{parsani2016entropy}, the flux is computed by first interpolating the entropy variables at Gauss nodes, then evaluating the conservative variables in terms of the interpolated entropy variables at a separate set of nodes. }

Let $\bm{f}^*$ be a numerical flux which is used for the imposition of boundary conditions.  Motivated by the quadrature discretization of flux differencing (\ref{eq:qfluxdiff}) and its reformulation using the Hadamard product (\ref{eq:qfluxdiffhadamard}), we introduce the semi-discrete formulation for $\bm{u}_h(t)$ as
\begin{align}
\td{\bm{u}_h}{t} &= - \LRs{\begin{array}{cc} 
\bm{P}_q & \bm{L}_q\end{array}} \LRp{2\bm{D}_N \circ \bm{F}_S} \bm{1} - \bm{L}_q\diag{\hat{\bm{n}}}\LRp{\bm{f}^* - \bm{f}(\tilde{\bm{u}}_f)},  \label{eq:oneelemformulation}\\
\LRp{\bm{F}_S}_{ij} &= \bm{f}_S\LRp{\tilde{\bm{u}}_{i},\tilde{\bm{u}}_{j}}, \qquad 1\leq i,j \leq N_q+N^f_q,
\nonumber
\end{align}
where $\bm{F}_S$ is a symmetric matrix (symmetry is a result of the symmetry condition of Definition~\ref{def:tadmor})
%, $f^k_S$ is the $k$th component of $\bm{f}_S$, 
and $\tilde{\bm{u}}_{j}$ denotes the evaluation of the entropy-projected conservative variables at the $i$th quadrature point, where $i$ indexes into the combined set of both volume and surface quadrature points.  

We have the following theorem: 
%Let $\bm{A},\bm{B}$ be two matrices with the same dimension.  The Hadamard product $\bm{A}\circ \bm{B}$ is defined as the matrix whose entries are
%\begin{equation}
%\LRp{\bm{A}\circ \bm{B}}_{ij} = \bm{A}_{ij} \bm{B}_{ij}.
%\end{equation}
%We can now show the following theorem:

%Here, $\bm{F}^m_S$ is a symmetric block matrix, where the entries of each block consist of evaluations of $\bm{f}_S$ at different volume and surface quadrature points.  

%While both the conservative and split forms of Burgers' equation are equivalent at the continuous level (assuming smoothness of $u$), they are not equivalent at the discrete level.  Suppose now that the solution is approximated by $u \in P^N\LRp{\widehat{D}}$ on the reference element.  Since $f(u) = u^2/2 \not\in P^N\LRp{\widehat{D}}$ for general choices of $u$, it is not possible to differentiate $f(u)$ exactly without prior knowledge of the form of the nonlinearity.  It is more common to first project $f(u)$ to $P^N\LRp{\widehat{D}}$, then differentiate, such that the discrete scheme is equivalent to solving for $u \in P^N\LRp{\widehat{D}}$ such that
%\begin{equation}
%\int_{-1}^1 \LRp{\pd{u}{t} + \pd{\Pi_N f(u)}{x}}w \diff{x}  
%%\int_{-1}^1 \LRp{\pd{u}{t} +D_N f(u)}w \diff{x} 
%=0, \qquad \forall w \in P^N\LRp{\widehat{D}}.  
%\end{equation}
\begin{theorem}
\noteTwo{Let $\bm{f}_S$ be an entropy conservative flux from Definition~\ref{def:tadmor}}.  Then, assuming continuity in time, the semi-discrete formulation defined by (\ref{eq:uvvars}) and (\ref{eq:oneelemformulation}) \noteTwo{satisfies}
\begin{equation}
\bm{1}^T \bm{W}\td{U(\bm{u}_q)}{t} = \bm{1}^T\bm{W}_f\diag{\hat{\bm{n}}}\LRp{\psi(\tilde{\bm{u}}_f) - \tilde{\bm{v}}_f^T\bm{f}^*}
\end{equation}
which is an approximation \noteTwo{of the statement of conservation of entropy (\ref{eq:consentropy}) involving numerical quadrature and the numerical flux $\bm{f}^*$}
\begin{equation}
\int_{\hat{D}} \pd{U(\bm{u}_N)}{t} \diff{\hat{x}}= \int_{\partial \hat{D}} \LRp{\psi\LRp{\bm{u}\LRp{\Pi_N\bm{v}}}-\LRp{\Pi_N \bm{v}}^T\bm{f}^*%\LRp{\bm{u}\LRp{\Pi_N\bm{v}}} 
}  \hat{n} \diff{\hat{x}}.
\end{equation}
\label{thm:ec1D}
\end{theorem}
\begin{proof}
The proof borrows concepts from \cite{gassner2017br1, chen2017entropy} in combination with properties of the decoupled SBP-like operator $\bm{D}_N$.  We first note that we can rewrite the projection and lifting operation as
\begin{equation}
\LRs{\begin{array}{cc} 
\bm{P}_q & \bm{L}_q\end{array}} = \bm{M}^{-1}\LRs{\begin{array}{c} 
\bm{V}_q\\ \bm{V}_f\end{array}}^T \LRs{\begin{array}{cc} 
\bm{W} &\\ &\bm{W}_f\end{array}}.
\end{equation}
Multiplying (\ref{eq:oneelemformulation}) by $\bm{M}$ yields the variational form of the equation
\begin{align}
\label{eq:matrixvarform}
\bm{M}\td{\bm{u}_h}{t} 
&= - \LRs{\begin{array}{c} 
\bm{V}_q\\ \bm{V}_f\end{array}}^T \LRp{2 \bm{Q}_N \circ \bm{F}_S}\bm{1} - \bm{V}_f^T\bm{W}_f\diag{\hat{\bm{n}}} \LRp{\bm{f}^* - \bm{f}(\tilde{\bm{u}}_f)}.
\end{align}
where we have moved the matrix of quadrature weights $\bm{W}$ inside the Hadamard product because multiplication of a matrix by a diagonal matrix is identical to taking the Hadamard product of both matrices \cite{horn2012matrix}.  We now test with the projection of the entropy variables $\bm{v}_h^T$ on both sides and note that
\begin{equation}
\bm{v}_h^T\bm{M} = \bm{v}_q^T\bm{P}_q^T\bm{M} = \bm{v}_q^T \bm{W}\bm{V}_q \bm{M}^{-1}\bm{M}
= \bm{v}_q \bm{W}\bm{V}_q.  
\end{equation}
Plugging this into the left hand side of (\ref{eq:matrixvarform}), assuming continuity in time (such that the chain rule holds), and using that $\bm{v} = \pd{U(\bm{u})}{\bm{u}}$ then yields
\begin{align}
\bm{v}_h^T\bm{M}\td{\bm{u}_h}{t} &= \bm{v}_q^T \bm{W}\td{\bm{u}_q}{t} = \sum_{i=1}^{N_q} w_i\bm{v}(\bm{u}_N({x}_i,t))^T \pd{\bm{u}_N({x}_i,t)}{t} = \sum_{i=1}^{N_q} w_i\pd{U\LRp{\bm{u}_N({x}_i,t)}}{t} \nonumber\\
&= \bm{1}^T\bm{W}\td{U(\bm{u}_q)}{t} \approx \int_{\hat{D}} \pd{U(\bm{u}_N)}{t}\diff{\hat{x}}.
\end{align}
We treat the right hand side next.  The contribution involving the numerical flux $\bm{f}^*$ yields
\begin{equation}
-\bm{v}_h^T\bm{V}_f^T\bm{W}_f\diag{\hat{\bm{n}}} \LRp{ \bm{f}^* -\bm{f}(\tilde{\bm{u}})} = -\tilde{\bm{v}}_f^T\bm{W}_f\diag{\hat{\bm{n}}} \LRp{ \bm{f}^* -\bm{f}(\tilde{\bm{u}}_f)}.
\label{eq:numflux}
\end{equation}
For the right hand side terms involving $\bm{F}_S$, testing with $\bm{v}_h^T$ and using Theorem~\ref{thm:sbp} yields 
\begin{equation}
- \LRp{\LRs{\begin{array}{c} 
\bm{V}_q\\ \bm{V}_f\end{array}}\bm{v}_h}^T \LRp{ \LRp{
\bm{B}_N + 
\bm{Q}_N - \bm{Q}_N^T} \circ \bm{F}_S}\bm{1} = - \tilde{\bm{v}}^T \LRp{ \LRp{
\bm{B}_N + 
\bm{Q}_N - \bm{Q}_N^T} \circ \bm{F}_S}\bm{1}.
\label{eq:rhs}
\end{equation}
The first term involving $\bm{B}_N$ is simplified by noting that $\bm{B}_N$ is diagonal, and that  $\bm{B}_N \circ \bm{F}_S$ extracts the values of $\bm{F}_S$ at face quadrature points.  These values correspond to evaluations of the entropy conservative numerical flux $\bm{f}_S(\tilde{\bm{u}}_f,\tilde{\bm{u}}_f)$, where $\tilde{\bm{u}}_f = \bm{u}\LRp{\bm{V}_f\bm{v}_h}$ is the evaluation of the entropy-projected conservative variables at surface quadrature points.  By the consistency condition of Definition~\ref{def:tadmor}, $\bm{f}_S(\tilde{\bm{u}}_f,\tilde{\bm{u}}_f) = \bm{f}(\tilde{\bm{u}}_f)$.  As a result, 
\begin{equation}
\LRp{\bm{B}_N \circ \bm{F}_S}\bm{1} =\LRp{\LRs{
\begin{array}{cc}
0&\\
& \bm{W}_f {\rm diag}\LRp{\hat{\bm{n}}} {\rm diag}\LRp{\bm{f}(\tilde{\bm{u}}_f)}
\end{array}
}}\bm{1} =  \bm{W}_f {\rm diag}\LRp{\hat{\bm{n}}} \bm{f}(\tilde{\bm{u}}_f).
\end{equation}
Combining this with (\ref{eq:rhs}) and (\ref{eq:numflux}) yields 
\begin{align}
-\tilde{\bm{v}}^T \LRp{ 
\bm{B}_N \circ \bm{F}_S}\bm{1} - \tilde{\bm{v}}_f^T\bm{W}_f\diag{\hat{\bm{n}}} \LRp{ \bm{f}^* -\bm{f}(\tilde{\bm{u}})} &= -\tilde{\bm{v}}_f^T \bm{W}_f {\rm diag}\LRp{\hat{\bm{n}}} \bm{f}(\tilde{\bm{u}}_f)  - \tilde{\bm{v}}_f^T\bm{W}_f\diag{\hat{\bm{n}}} \LRp{ \bm{f}^* -\bm{f}(\tilde{\bm{u}}_f)} \nonumber\\
&=  - \bm{1}^T\bm{W}_f\diag{\hat{\bm{n}}}\tilde{\bm{v}}_f^T\bm{f}^* \approx -\int_{\partial \hat{D}} \LRp{\Pi_N \bm{v}}^T\bm{f}^*  \hat{n} \diff{\hat{x}}.
\end{align}
where we have used the fact that $\bm{W}_f\diag{\hat{\bm{n}}}$ is diagonal to commute $\tilde{\bm{v}}_f^T$.  
The remaining terms in (\ref{eq:rhs}) involving $\bm{Q}_N$ are 
\begin{equation}
\tilde{\bm{v}}^T\LRp{ \LRp{
\bm{Q}_N - \bm{Q}_N^T} \circ \bm{F}_S}\bm{1} = 
\tilde{\bm{v}}^T \LRp{\bm{Q}_N \circ \bm{F}_S - \bm{Q}_N^T \circ \bm{F}_S}\bm{1} = \tilde{\bm{v}}^T \LRp{\bm{Q}_N \circ \bm{F}_S}\bm{1} - \bm{1}^T \LRp{\bm{Q}_N \circ \bm{F}_S}\tilde{\bm{v}},
\label{eq:rhsQN}
\end{equation}
where we have used 
\begin{equation}
\tilde{\bm{v}}^T\LRp{\bm{Q}_N^T \circ \bm{F}_S}\bm{1} = 
\bm{1}^T\LRp{\bm{Q}_N^T \circ \bm{F}_S}^T\tilde{\bm{v}}
= 
\bm{1}^T\LRp{\bm{Q}_N \circ \bm{F}_S}\tilde{\bm{v}}
\end{equation}
by the transpose property of Lemma~\ref{lemma:hadamard} and symmetry of $\bm{F}_S$.  Writing out (\ref{eq:rhsQN}) using sum notation 
\begin{align}
\tilde{\bm{v}}^T \LRp{\bm{Q}_N \circ \bm{F}_S}\bm{1} - \bm{1}^T \LRp{\bm{Q}_N \circ \bm{F}_S}\tilde{\bm{v}}
&= \sum_{i,j=1}^{N_q}  \LRp{\bm{Q}_N}_{ij} \tilde{\bm{v}}_i^T\LRp{\bm{F}_S}_{ij} - \LRp{\bm{Q}_N}_{ij} \tilde{\bm{v}}_j^T\LRp{\bm{F}_S}_{ij} \nonumber\\
&= \sum_{i,j=1}^{N_q} \LRp{\bm{Q}_N}_{ij} \LRp{\tilde{\bm{v}}_i - \tilde{\bm{v}}_j}^T \bm{f}_S(\tilde{\bm{u}}_i,\tilde{\bm{u}}_j).
\label{eq:QNsum}
\end{align}
Using the conservation condition in Definition~\ref{def:tadmor} of an entropy conservative flux and the fact that $\tilde{\bm{u}}_i = \bm{u}\LRp{\tilde{\bm{v}}_i}$, 
\begin{equation}
\LRp{\tilde{\bm{v}}_i - \tilde{\bm{v}}_j}^T \bm{f}_S(\tilde{\bm{u}}_i,\tilde{\bm{u}}_j) = \psi(\tilde{\bm{u}}_i) -  \psi(\tilde{\bm{u}}_j).
\label{eq:psi}
\end{equation}
Let $\bm{\psi}_i = \psi(\tilde{\bm{u}}_i)$.  Inserting (\ref{eq:psi}) into (\ref{eq:QNsum}) yields
\begin{align}
\sum_{i,j=1}^{N_q} \LRp{\bm{Q}_N}_{ij} \LRp{\tilde{\bm{v}}_i - \tilde{\bm{v}}_j}^T \bm{f}_S(\tilde{\bm{u}}_i,\tilde{\bm{u}}_j) 
&= \bm{1}^T\bm{Q}_N \bm{\psi} - \bm{\psi}^T\bm{Q}_N \bm{1} = \bm{1}^T\bm{Q}_N \bm{\psi} \nonumber\\
&= \bm{1}^T\LRp{\bm{B}_N-\bm{Q}_N^T} \bm{\psi} = \bm{1}^T\bm{B}_N\bm{\psi}\nonumber\\
&= \bm{1}^T\bm{W}_f\diag{\hat{\bm{n}}}\psi(\tilde{\bm{u}}_f) \approx \int_{\partial \hat{D}} \hat{n}\psi(\tilde{\bm{u}}_f)\diff{\hat{x}}, 
\end{align}
where Theorem~\ref{thm:sbp} implies that $\bm{Q}_N\bm{1} = 0$.  

\end{proof}
}

It is important to emphasize that conservation of entropy does not, in general, imply stability of the numerical scheme unless the solution satisfies additional constraints.  For example, the transformation between conservative and entropy variables for the compressible Euler and Navier-Stokes equations is well-defined only if the density and internal energy are positive, and steps must be taken in any numerical scheme to guarantee that the discrete solution satisfies such constraints.  For DG methods, this is most commonly done using bound and positivity-preserving limiters \cite{zhang2010positivity, zhang2012maximum}.  

It is also worth pointing out that, for specific choices of basis and quadrature, the formulation (\ref{eq:oneelemformulation}) can be significantly simplified.  For example, the co-location of nodal and quadrature points (assuming that the dimension of the approximation space is identical to the number of quadrature points) reduces $\bm{V}_q, \bm{P}_q$ to identity matrices.  Furthermore, if quadrature points coincide with boundary points (as with DG-SEM methods), the lift matrix $\bm{L}_q$ is zero except for entries which correspond to those boundary points.  These assumptions greatly simplify both the formulation and implementation of entropy conservative/stable DG methods, and will be discussed more thoroughly in a future manuscript.  

\subsection{Local conservation}
\label{sec:localconserv}
We next show that (\ref{eq:oneelemformulation}) is locally conservative.  {The proof is very similar to that of \cite{chen2017entropy}.  Recall the matrix variational form (\ref{eq:matrixvarform})
\begin{align}
\bm{M}\td{\bm{u}_h}{t} 
&= - \LRs{\begin{array}{c} 
\bm{V}_q\\ \bm{V}_f\end{array}}^T \LRp{2 \bm{Q}_N \circ \bm{F}_S}\bm{1} - \bm{V}_f^T\bm{W}_f\diag{\hat{\bm{n}}} \LRp{\bm{f}^* - \bm{f}(\tilde{\bm{u}}_f)}.
\end{align}
To prove local conservation, we test with $\LRp{\bm{P}_q\bm{1}}^T$, which gives
\begin{align}
\bm{1}^T\bm{W}\td{\bm{u}_q}{t} 
&= - \bm{1}^T\LRp{2 \bm{Q}_N \circ \bm{F}_S}\bm{1} - \bm{1}^T\bm{W}_f\diag{\hat{\bm{n}}} \LRp{\bm{f}^* - \bm{f}(\tilde{\bm{u}}_f)}.
\label{eq:localconserv}
\end{align}
Using the commutativity of the Hadmard product (Lemma~\ref{lemma:hadamard}), symmetry of $\bm{F}_S$, and Theorem~\ref{thm:sbp},
\begin{align}
\bm{1}^T\LRp{2 \bm{Q}_N \circ \bm{F}_S}\bm{1} &= \bm{1}^T\LRp{\bm{Q}_N \circ \bm{F}_S}\bm{1} + \bm{1}^T\LRp{\bm{Q}_N \circ \bm{F}_S}\bm{1} = 
\bm{1}^T\LRp{\bm{Q}_N \circ \bm{F}_S}\bm{1} + \bm{1}^T\LRp{\bm{Q}_N^T \circ \bm{F}_S}\bm{1} \nonumber\\
&= \bm{1}^T\LRp{\LRp{\bm{Q}_N+\bm{Q}_N^T} \circ \bm{F}_S}\bm{1} =  \bm{1}^T\LRp{\bm{B}_N \circ \bm{F}_S}\bm{1} = \bm{1}^T\bm{W}_f\diag{\hat{\bm{n}}} \bm{f}(\tilde{\bm{u}}_f).
\end{align}
Plugging this back into (\ref{eq:localconserv}) yields
\begin{equation}
\bm{1}^T\bm{W}\td{\bm{u}_q}{t} + \bm{1}^T\bm{W}_f\diag{\hat{\bm{n}}}\bm{f}^* = 0 
\end{equation}
which is a quadrature approximation to the conservation condition
\begin{equation}
\int_{\hat{D}} \pd{\bm{u}}{t}\diff{\hat{x}} + \int_{\partial \hat{D}} \bm{f}^*\diag{\hat{n}} \diff{\hat{x}} = 0.
\end{equation}
%It is not immediately clear that this statement of local conservation satisfies the conditions of the classic Lax-Wendroff theorem.  However, this form of local conservation does satisfy a generalized definition of local conservation, which is sufficient to guarantee convergence to a weak solution under mesh refinement \cite{shi2017local}.  

\section{{Entropy stable DG methods on multiple elements and in higher dimensions} }
\label{sec:ecdg2}
\label{sec:multipleelems}

{
In this section, we discuss the extension of entropy conservative DG schemes to multiple elements, the addition of interface dissipation, and the construction of entropy stable schemes in higher dimensions.   

\subsection{Multiple elements}

We now extend entropy conservative schemes to multiple elements in one dimension, where the domain $\Omega$ is broken up into $K$ non-overlapping intervals $D^k$ with outward normals $n$.  Each interval can be represented as the affine mapping $\Phi^k$ of the reference interval $\hat{D}$.  Because this mapping is affine, $J^k$ (the determinant of the Jacobian of $\Phi^k$) is constant over each element. 

An entropy conservative formulation can be constructed by modifying the lifting matrix for dimensional consistency.  Let $\bm{L}_q = \frac{1}{J^k}{\hat{\bm{L}}}_q$, where $\hat{\bm{L}}_q$ is the lifting matrix over the reference element, and let $\bm{n} = [-1,1]$ be the vector containing values of the outward normal at face quadrature points on the element $D^k$.  An entropy stable formulation over $D^k$ is given by 
\begin{align}
\td{\bm{u}_h}{t} &= - \LRs{\begin{array}{cc} 
\bm{P}_q & \bm{L}_q\end{array}} \LRp{2\bm{D}_N \circ \bm{F}_S} \bm{1} - \bm{L}_q\diag{{\bm{n}}}\LRp{\bm{f}^* - \bm{f}(\tilde{\bm{u}}_f)},  \label{eq:multielemformulation}\\
\LRp{\bm{F}_S}_{ij} &= \bm{f}_S\LRp{\tilde{\bm{u}}_{i},\tilde{\bm{u}}_{j}}, \qquad 1\leq i,j \leq N_q+N^f_q, \nonumber\\
\bm{f}^* &= \bm{f}_S\LRp{\tilde{\bm{u}}_f^+,\tilde{\bm{u}}_f}, \qquad \text{ at interior interfaces},
\nonumber
\end{align}
where $\tilde{\bm{u}}_f^+$ is the value of the entropy-projected conservative variables on the neighboring element.

Let $\hat{\bm{x}}^f_i$ denote face quadrature points on the reference element $\hat{D}$, and let $\bm{x}^f_i = \Phi^k\LRp{\hat{\bm{x}}^f_i}$ denote the mapping of the reference face points to physical face points of $D^k$.  Let $\bm{W}_{\partial \Omega}$ be the diagonal boundary matrix such that 
\begin{equation}
\LRp{\bm{W}_{\partial \Omega}}_{ii} = \begin{cases}
\LRp{\bm{W}_f}_{ii}, & \text{ if $\bm{x}^f_i$ is on the domain boundary } \partial \Omega\\
0, & \text{ otherwise.}
\end{cases}
\label{eq:bmatrix}
\end{equation}
In other words, $\bm{W}_{\partial \Omega}$ is zero for any interior elements and is equal to $\bm{W}_f$ at face quadrature points $\bm{x}^f_i$ which coincide with the boundary $\partial \Omega$.  Then, (\ref{eq:multielemformulation}) satisfies the following theorem:
\begin{theorem}
\noteTwo{Let $\bm{f}_S$ be an entropy conservative flux from Definition~\ref{def:tadmor}}.
The scheme (\ref{eq:multielemformulation}) is locally conservative and \noteTwo{satisfies}
\begin{equation}
\sum_{k=1}^K \bm{1}^T J^k\bm{W}\td{U(\bm{u}_q)}{t} = \sum_{k=1}^K \bm{1}^T\bm{W}_{\partial \Omega} \diag{{\bm{n}}}\LRp{\psi(\tilde{\bm{u}}_f) - \tilde{\bm{v}}_f^T\bm{f}^*},
\end{equation}
which is an approximation of \noteTwo{(\ref{eq:consentropy}) involving quadrature and the numerical flux on the boundary $\bm{f}^*$}
\begin{equation}
\int_{\Omega} \pd{U(\bm{u}_N)}{t}\diff{x} = \int_{\partial \Omega} \LRp{\psi\LRp{\bm{u}\LRp{\Pi_N\bm{v}}}-\LRp{\Pi_N \bm{v}}^T\bm{f}^*}  {n} \diff{x}.
\end{equation}
\label{thm:ecmultielem}
\end{theorem}
\begin{proof}
The proof of local conservation over each element $D^k$ is the same as the one-element case shown in Section~\ref{sec:localconserv}.  For multiple elements, showing conservation of entropy is done by first applying the one-element proof of conservation of entropy over each element $D^k$, summing over all elements, then cancelling shared interface terms.  Without loss of generality, we assume a periodic domain such that all interfaces are interior interfaces.  Scaling by $J^k$ on both sides, applying the one-element proof of conservation of entropy in Theorem~\ref{thm:ec1D}, and summing the results gives
\begin{equation}
\sum_{k=1}^K \bm{1}^TJ^k\bm{W}\td{U(\bm{u}_q)}{t} = \sum_{k=1}^K \LRp{\bm{1}^T\bm{W}_f \diag{{\bm{n}}} \psi\LRp{\tilde{\bm{u}}_f} - \bm{1}^T\bm{W}_f\diag{{\bm{n}}} \tilde{\bm{v}}_f^T\bm{f}_S\LRp{\tilde{\bm{u}}^+,\tilde{\bm{u}}}}. 
\end{equation}
Note that ${n}^+ = -{n}$, where ${n}^+$ denotes the outward normal on a neighboring element.  Then, splitting interface contributions between neighboring elements yields
\begin{align}
-\sum_{k=1}^K \tilde{\bm{v}}_f^T\bm{W}_f\diag{{\bm{n}}} \bm{f}_S\LRp{\tilde{\bm{u}}^+,\tilde{\bm{u}}}
&= \sum_{k=1}^K \frac{1}{2}\bm{1}^T\bm{W}_f\diag{{\bm{n}}} \LRp{\LRp{\tilde{\bm{v}}_f}^+-{\tilde{\bm{v}}_f}}^T\bm{f}_S\LRp{\tilde{\bm{u}}^+,\tilde{\bm{u}}},\label{eq:tadmornumflux}\\
&= \sum_{k=1}^K  \frac{1}{2}\bm{1}^T\bm{W}_f\diag{{\bm{n}}} \LRp{ \psi\LRp{\tilde{\bm{u}}^+} - \psi\LRp{\tilde{\bm{u}}}},\nonumber
\end{align}
where we have used the symmetry and conservation conditions of Definition~\ref{def:tadmor}.   Returning contributions involving $\psi\LRp{\tilde{\bm{u}}^+}$ to neighboring elements of $D^k$ cancels interface terms, such that
\begin{equation}
\sum_{k=1}^K \bm{1}^TJ^k\bm{W}\td{U(\bm{u}_q)}{t} = 0.  
\end{equation}
When the domain is not periodic, (\ref{eq:tadmornumflux}) is required only on interior interfaces, such that contributions from $\bm{f}^*$ remain on the boundaries of the domain.  
\end{proof}
}
We note that, since $\bm{f}^*$ is a function of $\tilde{\bm{u}}$ and not $\bm{u}$, it is not immediately clear that this statement of local conservation satisfies the conditions of the classic Lax-Wendroff theorem.  However, this form of local conservation does satisfy a generalized definition of local conservation, which is sufficient to guarantee convergence to a weak solution under mesh refinement \cite{shi2017local}.  

We also note that the analysis in previous sections has focused on the construction of entropy conservative schemes.  However, entropy is only conserved for smooth solutions and should be dissipated away in the presence of discontinuities and shocks.  To this end, one can construct discretely entropy stable schemes by adding additional dissipative terms (for example, by adding matrix dissipation terms \cite{chandrashekar2013kinetic, winters2017uniquely}) or Lax-Friedrichs penalization \noteTwo{in terms of either the entropy variables \cite{carpenter2014entropy} or the conservative variables \cite{chen2017entropy}.  For the numerical experiments presented in Section~\ref{sec:num}, we apply a local Lax-Friedrichs penalization in terms of the \noteTwo{entropy-projected conservative variables}, augmenting the flux function at element interfaces with the additional term
\begin{equation}
 \bm{f}_S\LRp{\bm{u}_L,\bm{u}_R}\Rightarrow \bm{f}_S\LRp{\bm{u}_L,\bm{u}_R} - \frac{\lambda}{2}\jump{\tilde{\bm{u}}},
\end{equation}
where $\lambda$ is an estimate of the maximum eigenvalue of the flux Jacobian.   It is not immediately obvious that the Lax-Friedrichs penalization dissipate entropy when multiplied by the entropy variables; however, it was shown in \cite[Corollary 3.2]{chen2017entropy} that the local Lax-Friedrichs flux is entropy dissipative when $\lambda$ is an appropriately chosen estimate of the average wave-speed.  The result can be extended to the current setting by noting that, since conservation of entropy requires testing with the projection of the entropy variables, the jump term should involve the evaluations of the entropy-projected conservative variables $\tilde{\bm{u}} = \bm{u}\LRp{\Pi_N \bm{v}}$ in order to guarantee entropy dissipation.  

\note{Finally, Theorem~\ref{thm:ecmultielem} implies that if boundary conditions are enforced in such a way that $\bm{f}^*$ is entropy stable at the boundaries (see, for example, \cite{svard2014entropy, parsani2015entropy, chen2017entropy}), then the semi-discrete scheme will satisfy a discrete analogue of the global entropy inequality 
\begin{equation}
\int_{\Omega} \pd{U(\bm{u}_N)}{t}\diff{x} \leq \int_{\partial \Omega} \LRp{\psi\LRp{\bm{u}\LRp{\Pi_N\bm{v}}}-\LRp{\Pi_N \bm{v}}^T\bm{f}^*}  {n} \diff{x} \leq 0.  
\end{equation}
A similar statement of global entropy dissipation also holds for periodic boundary conditions.}

\subsection{Higher dimensions}

In this section, we describe the construction of entropy stable DG methods for nonlinear conservation laws in $d$ dimensions on a domain $\Omega$
\begin{equation}
\pd{\bm{u}}{t} + \sum_{i=1}^d\pd{\bm{f}_i(\bm{u})}{\bm{x}_i} = 0.  
\end{equation}
We assume that the $d$-dimensional domain $\Omega \in \mathbb{R}^d$ is decomposed into non-overlapping elements $D^k$, and that $D^k$ is the image of the reference element $\widehat{D}$  under an affine mapping $\bm{x} = \Phi^k(\hat{\bm{x}})$ (where $\hat{\bm{x}}$ denotes coordinates on the reference element).  We approximate the solution over a physical element by mapping $P^N\LRp{\widehat{D}}$ to $D^k$ under $\Phi^k$
\begin{equation}
P^N\LRp{D^k} =  {\Phi}^k \circ P^N\LRp{\widehat{D}}.
\end{equation}
Volume integrals over each physical element $D^k$ can be mapped to the reference element.  In two and three dimensions, we also assume that each face of the element $D^k$ is the image of some reference face, such that surface integrals can be mapped from the physical element boundary $\partial D^k$ to reference element boundary $\partial \hat{D}$.  Thus, we have
\begin{equation}
\int_{D^k} f(\bm{x})\diff{\bm{x}} = \int_{\hat{D}}f(\bm{x}) J^k\diff{\hat{\bm{x}}}, \qquad \int_{\partial D^k} f(\bm{x})\diff{\bm{x}} = \int_{\partial \hat{D}} f(\bm{x}) J^k_f\diff{\hat{\bm{x}}},
\end{equation}
where  $J^k$ is the determinant of the Jacobian of $\Phi^k$ and $J^k_f$ is the Jacobian factor of the face mapping.  We assume both mappings to be affine in this work, such that both $J^k$ is constant over each element $D^k$ and $J^k_f$ is constant over each face.  

%We define the $L^2$ norm and inner products over the element $D^k$ and the surface of the element $\partial D^k$
%\begin{equation}
%\LRp{\bm{u},\bm{v}}_{D^k} = \int_{D^k} \bm{u}\cdot\bm{v}\diff{\bm{x}} =  \int_{\widehat{D}} \bm{u}\cdot\bm{v} J^k \diff{\widehat{x}}, \qquad \nor{\bm{u}}^2_{D^k} = (\bm{u},\bm{u})_{D^k}, \qquad \LRa{\bm{u},\bm{v}}_{\partial D^k} = \int_{\partial D^k} \bm{u} \cdot \bm{v} \diff{\bm{x}},
%\end{equation}

To construct entropy stable schemes in higher dimensions, we require a generalization of the entropy conservative fluxes defined in Definition~\ref{def:tadmor}.  
\begin{definition}
Let $\bm{f}_{i,S}(\bm{u}_L,\bm{u}_R)$ be a bivariate function which is symmetric and consistent with the $i$th coordinate flux function $\bm{f}_i(\bm{u})$.  The numerical flux $\bm{f}_{i,S}(\bm{u}_L, \bm{u}_R)$ is entropy conservative if, for entropy variables $\bm{v}_L = \bm{v}(\bm{u}_L), \bm{v}_R = \bm{v}(\bm{u}_R)$
\begin{align}
&\LRp{\bm{v}_L - \bm{v}_R}^T \bm{f}_{i,S}(\bm{u}_L,\bm{u}_R) = (\psi_{i,L} - \psi_{i,R}), \\
&\psi_{i,L} = \psi_i(\bm{v}(\bm{u}_L)), \quad \psi_{i,R} = \psi_i(\bm{v}(\bm{u}_R)), \qquad i = 1,\ldots,d.  \nonumber
\end{align}
\label{def:tadmor2}
\end{definition}

We now define SBP-like operators on mapped elements in multiple dimensions.  Let $\bm{G}^k = \pd{\Phi^k}{\bm{x}}$ be the matrix of geometric factors.  Since the mapping $\Phi^k$ is assumed to be affine, the entries of $\bm{G}^k$ are constant over each element.  Using these factors, we can define 
\begin{align}
\bm{D}^i_N = \sum_{j=1}^d \bm{G}^k_{ij} \hat{\bm{D}}^j_N, \qquad  \bm{L}_q = \hat{\bm{L}}_q \diag{\frac{\bm{J}^k_f}{J^k}},
\label{eq:sbpgeo}
\end{align}
where $\hat{\bm{D}}^j_N$ and $\hat{\bm{L}}_q$ are SBP-like operators defined on the reference element $\hat{D}$, and $\bm{J}^k_f$ is a vector of the values of $J^k_f$ at surface quadrature points.  One can show the relation between the geometric factors $\bm{G}_{ij}$ and components of the outward normals $\bm{n}_i$
\begin{equation}
\sum_{j=1}^d J^k\bm{G}_{ij} \hat{\bm{n}}_i\hat{J}_f = \bm{n}_i J^k_f,
\end{equation}
where $\hat{J}_f$ is the face Jacobian factor of the mapping from faces of the reference element to the reference face \cite{hesthaven2007nodal}.\footnote{The factor $\hat{J}_f$ appears, for example, for triangles, where the reference triangle is usually taken to be a right triangle.  For this reference triangle, two faces are of the same size, but the hypotenuse face is larger and will thus have a different value of $\hat{J}_f$ from the other two faces when mapping to the reference face.}  We assume $\hat{J}_f$ is pre-multiplied into the surface quadrature weights.  

Let $\bm{W}^k_N$ be the Jacobian-weighted diagonal matrix of volume and surface quadrature points
\begin{equation}
\bm{W}^k_N = 
\LRp{\begin{array}{cc}
\bm{W}J^k & \\
& \bm{W}_f \diag{\bm{J}^k_f/\hat{\bm{J}}_f}
\end{array}},
\end{equation}
where $\hat{\bm{J}}_f$ is the vector containing values of $\hat{J}_f$ at surface quadrature points, and ${\bm{J}^k_f/\hat{\bm{J}}_f}$ is a vector corresponding to the entry-wise division of $\bm{J}^k_f$ by $\hat{\bm{J}}_f$.  Then, $\bm{D}^i_N$ and $\bm{L}_q$ satisfy the following analogue of Theorem~\ref{thm:sbp} with respect to $\bm{W}^k_N$
\begin{equation}
\bm{W}^k_N \bm{D}^i_N + \LRp{\bm{W}^k_N \bm{D}^i_N}^T = \bm{B}^i_N, \qquad \bm{B}^i_N =
\LRp{\begin{array}{cc}
0 & \\
& \bm{W}_f \diag{{\bm{J}^k_f/\hat{\bm{J}}_f}} \diag{\bm{n}_i}
\end{array}}.
\label{eq:sbpgeoeq}
\end{equation}

A semi-discrete scheme can then be constructed over each element using (\ref{eq:sbpgeo})
\begin{align}
\td{\bm{u}_h}{t} &= - \sum_{i=1}^d \LRs{\begin{array}{cc} 
\bm{P}_q & \bm{L}_q\end{array}} \LRp{2\bm{D}^i_N \circ \bm{F}_{i,S}} \bm{1} - \bm{L}_q \diag{\bm{n}_i}\LRp{\bm{f}^*_i - \bm{f}_i(\tilde{\bm{u}}_f)},  \label{eq:multielemformulationnd}\\
\LRp{\bm{F}_{i,S}}_{ij} &= \bm{f}_{i,S}\LRp{\tilde{\bm{u}}_{i},\tilde{\bm{u}}_{j}}, \qquad 1\leq i,j \leq N_q+N^f_q.
\nonumber
\end{align}
where $\bm{f}^*_i$ is the $i$th component of the numerical flux and $\tilde{\bm{u}}$ is again the evaluation of the entropy-projected conservative variables.   

Let $\bm{W}^k_{\partial \Omega} = \bm{W}_{\partial \Omega}\diag{\bm{J}^k_f / \hat{\bm{J}}_f}$, where $\bm{W}_{\partial \Omega}$ is the diagonal boundary matrix defined in (\ref{eq:bmatrix}).  The multi-dimensional scheme satisfies the following theorem:
\begin{theorem}
\noteTwo{Let $\bm{f}_S$ be a higher dimensional entropy conservative flux from Definition~\ref{def:tadmor2}.}  
The scheme (\ref{eq:multielemformulationnd}) is locally conservative and \noteTwo{satisfies}
\begin{equation}
\sum_{k=1}^K \bm{1}^T J^k\bm{W}\td{U(\bm{u}_q)}{t} = \sum_{k=1}^K \sum_{i=1}^d \bm{1}^T\bm{W}_{\partial \Omega} \diag{\bm{n}_i}\LRp{\psi_i(\tilde{\bm{u}}_f) - \tilde{\bm{v}}_f^T\bm{f}_i^*},
\end{equation}
which is an \noteTwo{approximation of the higher dimensional generalization of the conservation of entropy (\ref{eq:consentropy}) involving numerical quadrature and the boundary numerical fluxes $\bm{f}_i^*$}
\begin{equation}
\int_{\Omega} \pd{U(\bm{u}_N)}{t}\diff{x} = \int_{\partial \Omega} \LRp{\psi_i\LRp{\bm{u}\LRp{\Pi_N\bm{v}}}-\LRp{\Pi_N \bm{v}}^T\bm{f}_i^*} \cdot \bm{n} \diff{\bm{x}}.
\end{equation}
\end{theorem}
\begin{proof}
The proofs of both stability and conservation are shown by applying (\ref{eq:sbpgeoeq}) and the one-dimensional proofs of Theorem~\ref{thm:ec1D} and Theorem~\ref{thm:ecmultielem} along each of the $i$th coordinate directions.  
\end{proof}
}

\section{Numerical experiments: the compressible Euler equations}
\label{sec:num}

In this section, we illustrate the entropy conservation and accuracy of the proposed schemes for the one dimensional compressible Euler equations.  
All numerical experiments utilize the fourth order five-stage low-storage Runge-Kutta method Carpenter and Kennedy \cite{carpenter1994fourth}.  Following the derivation of stable timestep restrictions in \cite{chan2015gpu}, we define the timestep $\Delta t$ to be 
\begin{equation}
\Delta t = {\rm CFL} \times \frac{h}{C_N}, \qquad C_N = \frac{(N+1)^2}{2},
\end{equation} 
where $C_N$ is the one-dimensional constant in the trace inequality \cite{warburton2003constants}, and ${\rm CFL}$ is a user-defined constant.  

We note that the numerical implementations used here are oblivious to the choice of basis, and the discretization is specified completely by the choice of quadrature.  For example, when GLL quadratures are used, an entropy conservative/stable DG-SEM discretization is recovered \cite{carpenter2014entropy, chen2017entropy, gassner2017br1}, and when a Gauss quadrature with $(N+1)$ points is used, generalized SBP-DG methods are recovered \cite{ranocha2017extended, ranocha2017comparison}.  

\subsection{One-dimensional experiments}

The one-dimensional compressible Euler equations, which correspond to the inviscid limit of the compressible Navier-Stokes equations, are given as follows:
\begin{align}
\pd{\rho}{t} + \pd{\LRp{\rho u}}{x} &= 0,\\
\pd{\rho u}{t} + \pd{\LRp{\rho u^2 + p }}{x} &= 0,\nonumber\\
\pd{E}{t} + \pd{\LRp{u(E+p)}}{x} &= 0.\nonumber
\end{align}
We assume an ideal gas, such that the pressure satisfies the constitutive relation $p = (\gamma-1)\LRp{E - \frac{1}{2}\rho u^2}$, where $\gamma = 1.4$ is the ratio of specific heat for a diatomic gas.    

The choice of convex entropy for the Euler equations is non-unique \cite{harten1983symmetric}.  However, a unique entropy can be chosen by restricting to choices of entropy variables which symmetrize the viscous heat conduction term in the compressible Navier-Stokes equations \cite{hughes1986new}.  This leads to $U(\bm{u})$ of the form
\begin{equation}
U(\bm{u}) = -\frac{\rho s}{\gamma-1},
\end{equation}
where $s = \log\LRp{\frac{p}{\rho^\gamma}}$ is the physical specific entropy.  The entropy variables under this choice of entropy are then
\begin{align}
%v_1 = \frac{\gamma-s}{\gamma-1} - \frac{\rho u^2}{2p}, \qquad v_2 = \frac{\rho u}{p}, \qquad v_3 = -\frac{\rho}{p}.
v_1 = \frac{E - \rho e(\gamma + 1 - s)}{\rho e}, \qquad v_2 = \frac{\rho u}{\rho e}, \qquad v_3 = -\frac{\rho}{\rho e},
\end{align}
where $\rho e = E - \frac{1}{2}\rho u^2$ is the specific internal energy.  
The inverse mapping is given by 
\begin{equation}
\rho = -(\rho e) v_3, \qquad \rho u = (\rho e) v_2, \qquad E = (\rho e)\LRp{1 - \frac{v_2^2}{2 v_3}},
\end{equation}
where $\rho e$ and $s$ in terms of the entropy variables are 
\begin{equation}
\rho e = \LRp{\frac{(\gamma-1)}{\LRp{-v_3}^{\gamma}}}^{1/(\gamma-1)}e^{\frac{-s}{\gamma-1}}, \qquad s = \gamma - v_1 + \frac{v_2^2}{2v_3}.
\end{equation}
In order for the entropy $U(\bm{u})$ to be well defined, we require the assumption that the discrete density and pressure solutions are bounded away from zero
\begin{equation}
\rho \geq \rho_0 > 0,  \qquad p \geq p_0 > 0.  \label{eq:assumption2}
\end{equation}
This can be achieved using positivity-preserving limiters \cite{zhang2010positivity, zhang2012maximum}.  However, these have not been implemented in our numerical simulations.  

Examples of entropy conservative flux functions can be found in \cite{ismail2009affordable, chandrashekar2013kinetic}.  In this work, we utilize the flux function $\bm{f}_S(\bm{u}_L,\bm{u}_R)$ introduced by Chandreshekar \cite{chandrashekar2013kinetic}, whose components are given as
\begin{align}
f^1_S(\bm{u}_L,\bm{u}_R) &= \avg{\rho}^{\log} \avg{u}\\
f^2_S(\bm{u}_L,\bm{u}_R) &= \frac{\avg{\rho}}{2\avg{\beta}} + \avg{u}f^1_S\nonumber\\
f^3_S(\bm{u}_L,\bm{u}_R) &= f^1_S\LRp{\frac{1}{2(\gamma-1)\avg{\beta}^{\log}} - \frac{1}{2}\avg{u^2}} + \avg{u}f^2_S,\nonumber
\end{align}
where we have introduced the inverse temperature $\beta$
\begin{equation}
\beta = \frac{\rho}{2p}
\end{equation}
and the logarithmic mean
\begin{equation}
\avg{u}^{\log} = \frac{u_L - u_R}{\log{u_L}- \log{u_R}}.  
\end{equation}
We note that, because the direct evaluation of the logarithmic mean is numerically sensitive for $u_L\approx u_R$, when $\LRb{u_L-u_R}<\epsilon$ we switch to evaluation using a high order accurate expansion introduced by Ismail and Roe \cite{ismail2009affordable}.  

These fluxes are both entropy conservative and kinetic energy preserving.  Unlike the shallow water equations, these entropy conservative fluxes do not correspond to stable split formulations of the Euler equations.  Thus, the compressible Euler equations serve as a test of the flux differencing formulation, as entropy stability cannot be achieved through skew-symmetry.  

We also present results which utilize the dissipative local Lax-Friedrichs flux described in Section~\ref{sec:ecdg2}, where the value of $\lambda$ is estimated by
\begin{equation}
\lambda = \max_{\bm{u}^+, \bm{u}} \LRc{\LRb{u} + c}, 
%\frac{1}{2}\LRp{\LRb{u}_R + c_R + \LRb{u}_L + c_L}, 
\qquad c = \sqrt{\frac{\gamma p}{\rho}}.
\end{equation}
We will refer to the combination of the entropy conservative flux with Lax-Friedrichs dissipation as the ``Lax-Friedrichs'' flux.  

\subsubsection{Smooth entropy wave solution}

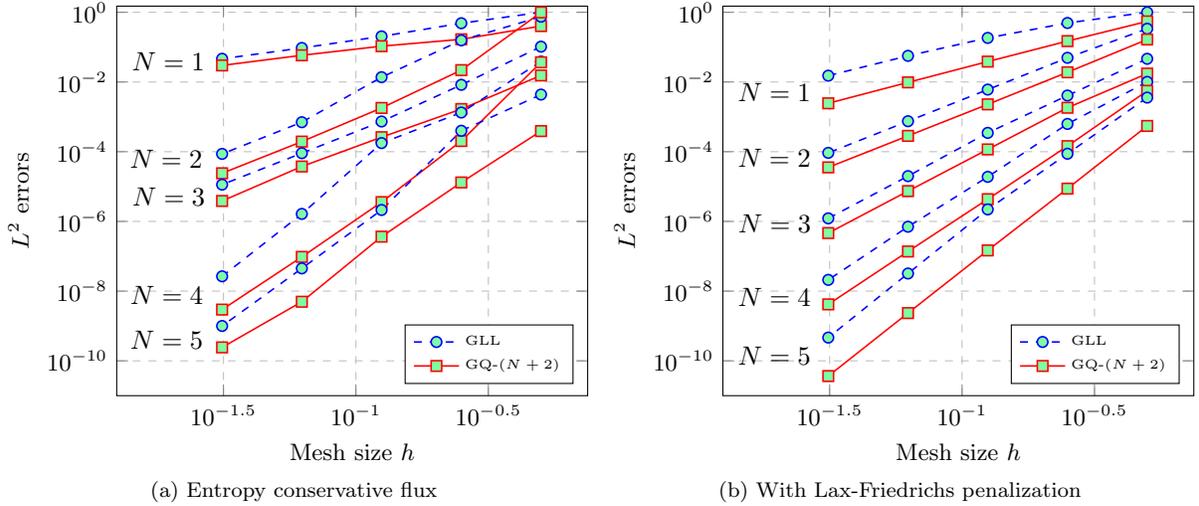
\begin{figure}
\centering
\subfloat[Entropy conservative flux]{
\begin{tikzpicture}
\begin{loglogaxis}[
    width=.475\textwidth,
    xlabel={Mesh size $h$},
    ylabel={$L^2$ errors}, 
    xmin=.0125, xmax=.75,
    ymin=1e-11, ymax=1.5,
    legend pos=south east, legend cell align=left, legend style={font=\tiny},	
    xmajorgrids=true, ymajorgrids=true, grid style=dashed,
    legend entries={GLL,GQ-$(N+2)$}    
]
\pgfplotsset{
cycle list={{blue, dashed, mark=*}, {red, mark=square*}}
}
%\addlegendimage{no markers,blue}
%\addlegendimage{no markers,red}

\addplot+[semithick, mark options={solid, fill=markercolor}]
% N = 1, tau = 0.000000 =======================
coordinates{(0.5,1)(0.25,0.485059)(0.125,0.203599)(0.0625,0.0947163)(0.03125,0.0463705)};
\addplot+[semithick, mark options={solid, fill=markercolor}]
%N = 1, tau = 0.000000 =======================
coordinates{(0.5,0.402314)(0.25,0.167917)(0.125,0.106574)(0.0625,0.058359)(0.03125,0.0298728)}
[yshift=2pt] node[left, pos=1.025, color=black] {$N = 1$};

\addplot+[semithick, mark options={solid, fill=markercolor}]
% N = 2, tau = 0.000000 =======================
coordinates{(0.5,0.746606)(0.25,0.156701)(0.125,0.0137392)(0.0625,0.000701926)(0.03125,8.64531e-05)};
\addplot+[semithick, mark options={solid, fill=markercolor}]
%N = 2, tau = 0.000000 =======================
coordinates{(0.5,0.993771)(0.25,0.0219437)(0.125,0.00180028)(0.0625,0.000194939)(0.03125,2.4045e-05)}
[yshift=7pt] node[left, pos=1.025, color=black] {$N = 2$};

\addplot+[semithick, mark options={solid, fill=markercolor}]
% N = 3, tau = 0.000000 =======================
coordinates{(0.5,0.103299)(0.25,0.00829887)(0.125,0.00073573)(0.0625,9.05975e-05)(0.03125,1.13596e-05)};
\addplot+[semithick, mark options={solid, fill=markercolor}]
%N = 3, tau = 0.000000 =======================
coordinates{(0.5,0.0154054)(0.25,0.00167426)(0.125,0.000260859)(0.0625,3.76182e-05)(0.03125,3.86238e-06)}
[yshift=3pt] node[left, pos=1.025, color=black] {$N = 3$};

\addplot+[semithick, mark options={solid, fill=markercolor}]
% N = 4, tau = 0.000000 =======================
coordinates{(0.5,0.0385542)(0.25,0.00133048)(0.125,0.000176663)(0.0625,1.64135e-06)(0.03125,2.66024e-08)};
\addplot+[semithick, mark options={solid, fill=markercolor}]
%N = 4, tau = 0.000000 =======================
coordinates{(0.5,0.0367592)(0.25,0.000202817)(0.125,3.57758e-06)(0.0625,9.58294e-08)(0.03125,2.94985e-09)}[yshift=8pt] node[left, pos=1.025, color=black] {$N = 4$};

\addplot+[semithick, mark options={solid, fill=markercolor}]
% N = 5, tau = 0.000000 =======================
coordinates{(0.5,0.00436131)(0.25,0.00039846)(0.125,2.1282e-06)(0.0625,4.49046e-08)(0.03125,9.99912e-10)};
\addplot+[semithick, mark options={solid, fill=markercolor}]
%N = 5, tau = 0.000000 =======================
coordinates{(0.5,0.000390565)(0.25,1.31188e-05)(0.125,3.6544e-07)(0.0625,4.95271e-09)(0.03125,2.42763e-10)}
[yshift=5pt] node[left, pos=1.025, color=black] {$N = 5$};

%\legend{$N=1$,$N=2$,$N=3$,$N=4$,$N=5$}
\end{loglogaxis}
\end{tikzpicture}
}
\subfloat[With Lax-Friedrichs penalization]{
\begin{tikzpicture}
\begin{loglogaxis}[
    width=.475\textwidth,
    xlabel={Mesh size $h$},  ylabel={$L^2$ errors}, 
    xmin=.0125, xmax=.75,
    ymin=1e-11, ymax=1.5,
    legend pos=south east, legend cell align=left, legend style={font=\tiny},	
    xmajorgrids=true, ymajorgrids=true, grid style=dashed,
    legend entries={GLL,GQ-$(N+2)$}
] 
%\pgfplotsset{cycle list={{blue, dashed, mark=*}, {red, mark=square*}}}
\pgfplotsset{
cycle list={{blue, dashed, mark=*}, {red, mark=square*}}
}

\addplot+[semithick, mark options={solid, fill=markercolor}]
% N = 1, tau = 0.500000 =======================
coordinates{(0.5,1)(0.25,0.4932)(0.125,0.183839)(0.0625,0.0562398)(0.03125,0.0151873)};
\addplot+[semithick, mark options={solid, fill=markercolor}]
%N = 1, tau = 0.500000 =======================
coordinates{(0.5,0.547558)(0.25,0.148981)(0.125,0.0384647)(0.0625,0.00974763)(0.03125,0.00244539)}
[yshift=5pt] node[left, pos=1.025, color=black] {$N = 1$};

\addplot+[semithick, mark options={solid, fill=markercolor}]	
% N = 2, tau = 0.500000 =======================
coordinates{(0.5,0.336817)(0.25,0.04941)(0.125,0.00605428)(0.0625,0.000748842)(0.03125,9.25456e-05)};
\addplot+[semithick, mark options={solid, fill=markercolor}]
%N = 2, tau = 0.500000 =======================
coordinates{(0.5,0.165807)(0.25,0.0190013)(0.125,0.00227903)(0.0625,0.00028425)(0.03125,3.54865e-05)}
[yshift=5pt] node[left, pos=1.025, color=black] {$N = 2$};

% N = 3, tau = 0.500000 =======================
\addplot+[semithick, mark options={solid, fill=markercolor}]
coordinates{(0.5,0.0463242)(0.25,0.00408748)(0.125,0.000346831)(0.0625,1.99064e-05)(0.03125,1.22357e-06)};
\addplot+[semithick, mark options={solid, fill=markercolor}]
%N = 3, tau = 0.500000 =======================
coordinates{(0.5,0.0174194)(0.25,0.00182234)(0.125,0.000116147)(0.0625,7.39839e-06)(0.03125,4.6305e-07)}
[yshift=5pt] node[left, pos=1.025, color=black] {$N = 3$};

\addplot+[semithick, mark options={solid, fill=markercolor}]
% N = 4, tau = 0.500000 =======================
coordinates{(0.5,0.0100716)(0.25,0.000625923)(0.125,1.89866e-05)(0.0625,7.03865e-07)(0.03125,2.10265e-08)};
\addplot+[semithick, mark options={solid, fill=markercolor}]
%N = 4, tau = 0.500000 =======================
coordinates{(0.5,0.00556743)(0.25,0.000144595)(0.125,4.33972e-06)(0.0625,1.37151e-07)(0.03125,4.16335e-09)}
[yshift=5pt] node[left, pos=1.025, color=black] {$N = 4$};

\addplot+[semithick, mark options={solid, fill=markercolor}]
% N = 5, tau = 0.500000 =======================
coordinates{(0.5,0.00356628)(0.25,8.73125e-05)(0.125,2.20528e-06)(0.0625,3.20127e-08)(0.03125,4.63639e-10)};
\addplot+[semithick, mark options={solid, fill=markercolor}]
%N = 5, tau = 0.500000 =======================
coordinates{(0.5,0.000547621)(0.25,8.7194e-06)(0.125,1.47105e-07)(0.0625,2.34345e-09)(0.03125,3.65306e-11)}
[yshift=10pt] node[left, pos=1.025, color=black] {$N = 5$};

\end{loglogaxis}
\end{tikzpicture}
}
\caption{$L^2$ errors under mesh refinement for entropy conservative and Lax-Friedrichs fluxes under both Gauss-Legendre-Lobatto (GLL) and over-integrated $(N+2)$ point Gauss quadrature (GQ-$(N+2)$).  Both sets of errors are evaluated using a GQ-$(N+5)$ quadrature rule. }
\label{fig:convergence}
\end{figure}
We begin by verifying the high order accuracy of the proposed methods using a periodic entropy wave solution
\begin{equation}
\rho(x,t) = 2 + \sin\LRp{\pi (x - t)}, \qquad u(x,t) = 1, \qquad p(x,t) = 1.
\end{equation}
We compute the $L^2$ error in the conservative variables 
\begin{equation}
\nor{\bm{u} - \bm{u}_h}_{L^2}^2 = \nor{\rho - \rho_h}_{L^2}^2 + \nor{\rho u - \rho u_h}_{L^2}^2 + \nor{E - E_h}_{L^2}^2
\end{equation}
at final time $T = .7$ using both the entropy conservative and local Lax-Friedrichs fluxes and a CFL of .125.  For these experiments, we compare two quadrature rules: 
\begin{enumerate}
\item the $(N+1)$ point Gauss-Legendre-Lobatto rule (referred to as ``GLL''),
%\item a $(N+1)$ point Gauss quadrature rule (referred to as ``GQ-$(N+1)$''),
\item an over-integrated $(N+2)$ point Gauss quadrature rule (referred to as ``GQ-$(N+2)$'').  
\end{enumerate}
The $L^2$ error is evaluated using a more accurate $N+5$ point Gauss quadrature rule.  

\begin{table}[h!]
\centering
 \begin{tabular}{||c | c |c |c |c | c|} 
 \hline
 &  $N=1$ &  $N=2$ &  $N=3$ &  $N=4$ &  $N=5$\\
 \hline
 GLL (entropy conservative) & 1.0672 & 3.6561  &  3.0086 &   6.3486 &   5.5278\\ 
 \hline
 GQ-$(N+2)$ (entropy conservative) &  0.9175  &  3.1132  &  3.0388 &   5.1221 &   5.2779 \\ 
 \hline
 GLL (Lax-Friedrichs) & 1.8887  &  3.0164  &  4.0241 &   5.0650 &   6.1095\\ 
 \hline
 GQ-$(N+2)$ (Lax-Friedrichs) & 1.9950  &  3.0018  &  3.9980  &  5.0419  &  6.0034\\ 
 \hline
 \end{tabular}
 \caption{Computed asymptotic convergence rates of $L^2$ errors under mesh refinement.  {Even-odd decoupling of the convergence rates are observed for the entropy conservative flux, while optimal convergence rates are observed for both GLL and GQ-$(N+2)$ quadratures when using Lax-Friedrichs penalization.}}
 \label{tab:rates}
\end{table}

Figure~\ref{fig:convergence} shows computed $L^2$ errors under mesh refinement for both GLL and GQ-$(N+2)$ quadrature rules with and without Lax-Friedrichs penalization.  We do not compare against the $(N+1)$ point Gauss quadrature rule or quadrature rules with $N_q > (N+2)$, as the $L^2$ errors are very similar to the GQ-$(N+2)$ case.  Computed asymptotic convergence rates are reported in Table~\ref{tab:rates}.  It can be observed that, for both GLL and GQ-$(N+2)$ quadratures, the entropy conservative flux exhibits suboptimal convergence rates for odd orders, while the inclusion of Lax-Friedrichs penalization restores the optimal $O(h^{N+1})$ convergence rate for both quadrature choices.

\subsubsection{Discontinuous profile on a periodic domain}

Next, we examine the discrete evolution of entropy by evolving a discontinuous initial profile to final time $T=2$ on the domain $[-1,1]$.  We initialize the density and velocity to be 
\begin{equation}
\rho(x,t) = \begin{cases}
3 & \LRb{x} < 1/2\\
2 & \text{otherwise},
\end{cases} \qquad 
u(x,t) = 0, \qquad
p(x,t) = \rho^\gamma.
\label{eq:discontin}
\end{equation}
Periodic boundary conditions are enforced in order to examine the evolution of entropy over longer time periods.  Figure~\ref{subfig:sol} shows $\rho, u$ at time $T = 1/10$ using both entropy conservative and Lax-Friedrichs fluxes (referred to in the figure as ``EC'' and ``LF'', respectively) and GQ-$(N+2)$ quadrature.  As expected, the entropy conservative flux results in spurious high frequency oscillations, which are significantly damped under Lax-Friedrichs penalization.  

We next examine the change in entropy over time.  While the proof of conservation of entropy holds at the semi-discrete level, it does not take into account the time discretization and thus does not hold at the fully discrete level.  This is reflected in the observation that the numerical change in entropy 
\begin{equation}
\Delta U(t) = U\LRp{\bm{u}(x,t)} - U\LRp{\bm{u}(x,0)}
\end{equation}
increases as $t$ increases.  However, numerical experiments in \cite{gassner2016well} suggest that the discrete change in entropy over time should converge to zero as the timestep decreases.%, which (under appropriate assumptions) implies that the fully discrete scheme converges to the semi-discrete scheme.  

Figure~\ref{subfig:dS} shows the evolution of the integral of $\Delta U(t)$ to final time $T=2$ for both entropy conservative and Lax-Friedrichs fluxes at various CFL numbers {using a GQ-$(N+2)$ quadrature rule.  We focus on this rule, as the conservation of entropy for the GLL quadrature rule has been established in the literature theoretically and numerically for the compressible Euler equations \cite{fisher2013high, gassner2017br1, chen2017entropy}.  Additionally, we note that the semi-discrete conservation of the integrated entropy depends on the strength of the quadrature rule utilized.  Comparisons between GLL and GQ-$(N+2)$ quadrature rules introduce inconsistency due to the fact that the former rule is exact for degree $2N-1$ polynomials, while the latter is a stronger rule and is exact for $2N+3$ polynomials. }
   
For the entropy conservative flux, we observe that $\Delta U(t)$ decreases as the CFL and timestep $\Delta t$ decrease.  This is expected, since the discrete time problem should converge to the continuous semi-discrete problem (for which $\Delta U(t) = 0$) as $\Delta t\rightarrow 0$.  We also observe that $\Delta U(t)$ does not change significantly as a function of the CFL for the Lax-Friedrichs flux, indicating that the non-zero change in entropy in this case is due to the effect of the dissipative flux rather than the time discretization.  {We also compute the convergence rate of $\Delta \bm{U}(T)$ to zero with respect to the timestep $\Delta t$, as shown in Figure~\ref{subfig:dURate}.  Despite the fact that a fourth order time-stepper is used, we observe nearly fifth order convergence.  This phenomena is not observed in two dimensions, as we show in Section~\ref{sec:2d}.}

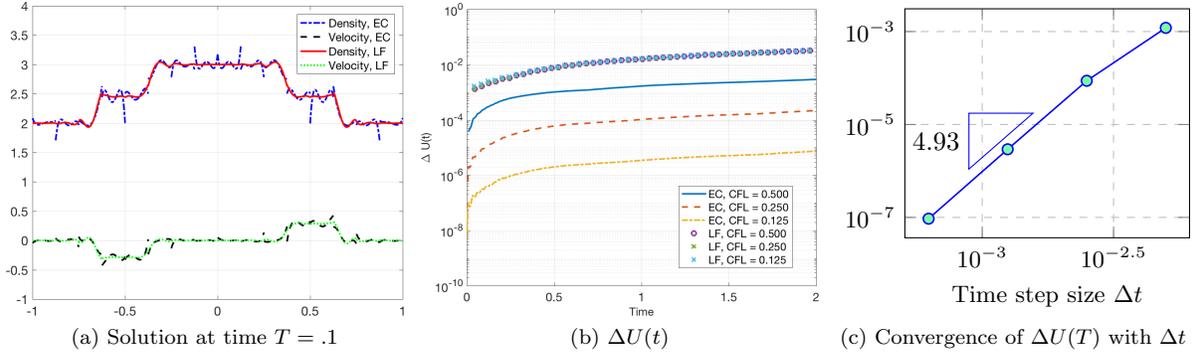
\begin{figure}
\centering
\subfloat[Solution at time $T = .1$]{
\includegraphics[width=.32\textwidth]{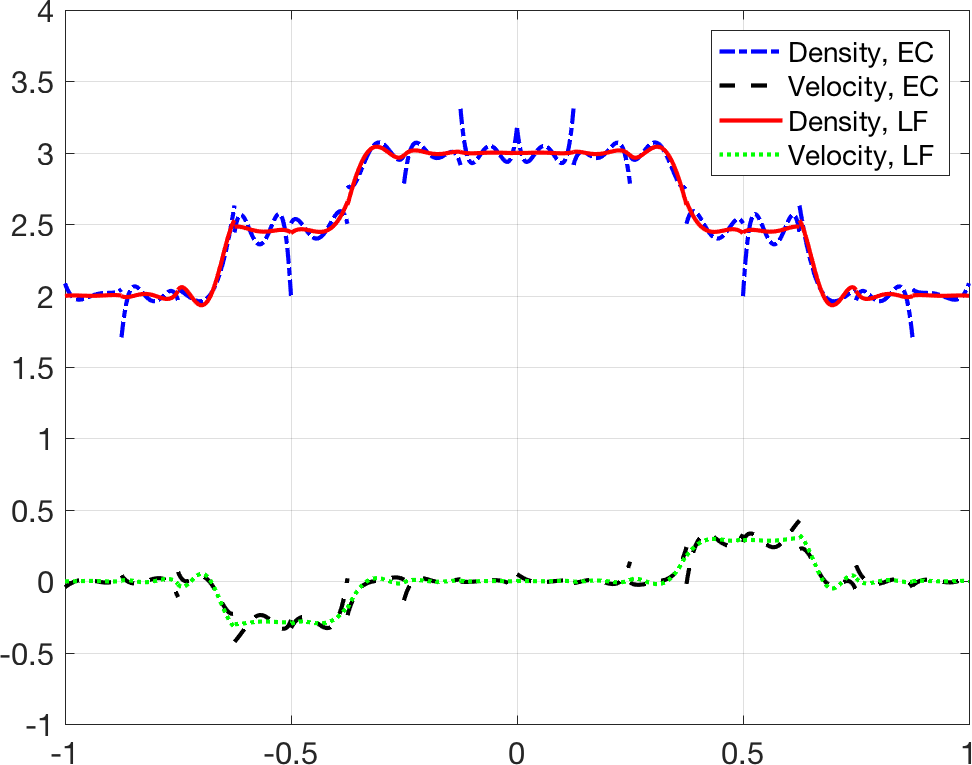}
\label{subfig:sol}}
\subfloat[$\Delta U(t)$]{
\includegraphics[width=.32\textwidth]{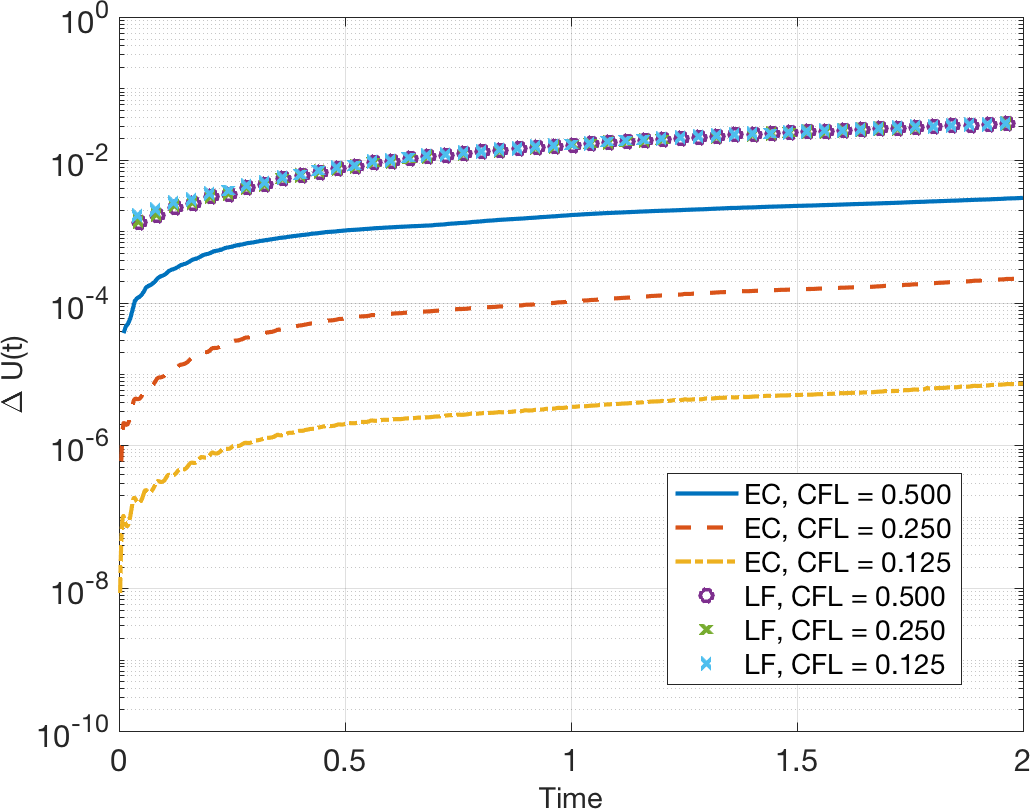}
\label{subfig:dS}}
\subfloat[Convergence of $\Delta U(T)$ with $\Delta t$]{
\begin{tikzpicture}
\begin{loglogaxis}[
    width=.325\textwidth,
    xlabel={Time step size $\Delta t$},
%    ylabel={ $\Delta U(T)$}, 
%    xticklabels={$0.005$, $0.0025$, $0.00125$},
%    xmin=.0625, xmax=.75,
%    ymin=1e-6, ymax=.01,
    legend pos=south east, legend cell align=left, legend style={font=\tiny},	
    xmajorgrids=true, ymajorgrids=true, grid style=dashed,
]
%\pgfplotsset{
%cycle list={{blue, dashed, mark=*}, {red, mark=square*}}
%}
%\addlegendimage{no markers,blue}
%\addlegendimage{no markers,red}

\addplot+[semithick, mark options={solid, fill=markercolor}]
coordinates{(0.005,0.0012)(0.0025,8.7262e-05)(0.00125,2.9252e-06)(0.000625,9.39239e-08)};
\logLogSlopeTriangleFlip{0.45}{0.225}{0.3}{4.93}{blue}
\end{loglogaxis}
\end{tikzpicture}
\label{subfig:dURate}
}
\caption{Solution snapshot and change in entropy $\Delta U(t)$ for both entropy conservative (EC) and Lax-Friedrichs (LF) fluxes {using a GQ-$(N+2)$ quadrature rule}.  The convergence of the change in entropy $\Delta U(t)$ at the final time $T = 2$ converges to zero as $O\LRp{\Delta t^{4.93}}$, which is greater than the order of the 4th order time-stepper used. }
\label{fig:dS}
\end{figure}

We also numerically evaluate the spatial formulation tested against the projected entropy variables
\begin{equation}
\delta(t) = \LRb{\LRp{\left.\LRp{D^x_h \bm{f}_S(\tilde{\bm{u}}(x),\tilde{\bm{u}}(y))}\right|_{y=x},\Pi_N \bm{v}}_{\Omega}}, \qquad  \qquad 0\leq t \leq T.
\end{equation}
We utilize the non-dissipative entropy conservative flux and evolve the initial profile (\ref{eq:discontin}) until time $T = 1$ using a GQ-$(N+2)$ quadrature rule.  From the proof of entropy conservation, we expect $\delta_{\max} = \max_{t\in (0,T)}{\delta}$ to be machine zero.  In practice, we have found that $\delta_{\max}$ depends on the tolerance $\epsilon$ used in evaluation of the logarithmic mean.  For a simulation to time $T=4$ using $N=4$ and $K = 16$ with a CFL of $1/2$, we observe that using $\epsilon = 10^{-2}$ (as recommended in \cite{ismail2009affordable}) results in $\delta_{\max} = O\LRp{10^{-10}}$.  Decreasing $\epsilon$ to $10^{-3}$ reduces $\delta_{\max}$ to $O\LRp{10^{-14}}$, and decreasing $\epsilon$ further to $1\times 10^{-4}$ reduces $\delta_{\max}$ to $10^{-15}$.  Smaller values of $\epsilon$ do result in observable changes to $\delta_{\max}$, and we do not observe any significant dependence of $\delta_{\max}$ on other discretization parameters.  

\begin{figure}
\centering
\includegraphics[width=.525\textwidth]{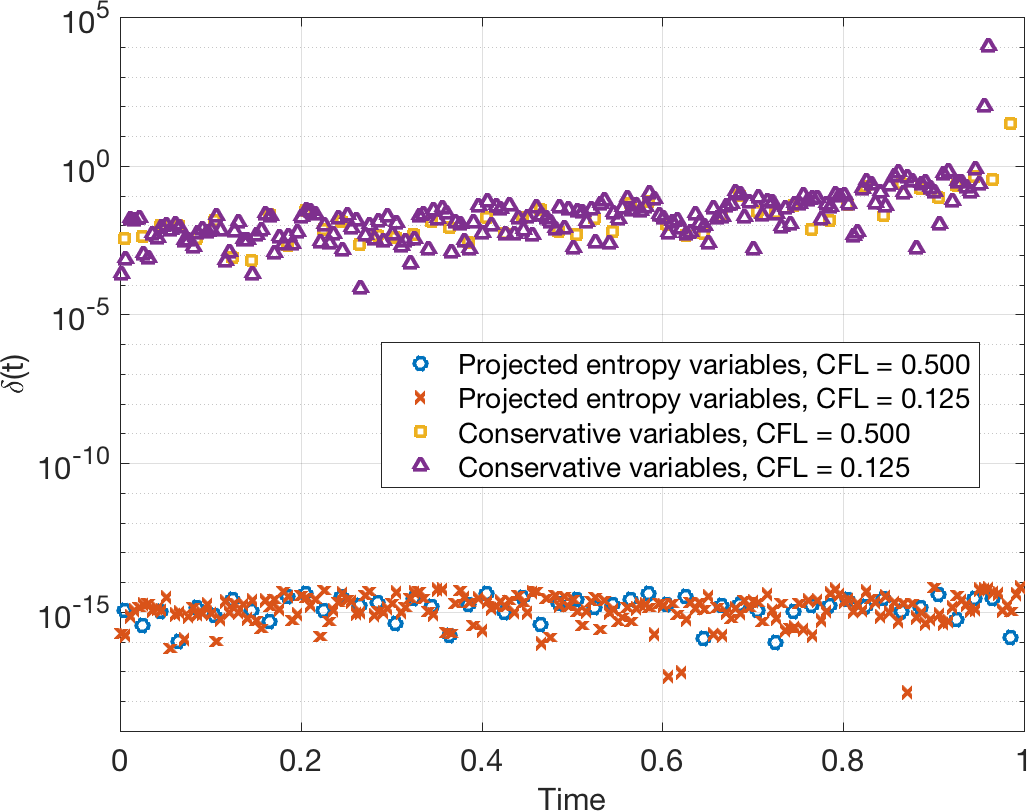}
\caption{Comparison of $\delta(t)$ when evaluating the flux function $\bm{f}_S$ using conservative variables and projected entropy variables.  {A GQ-$(N+2)$ quadrature rule is used.} }
\label{fig:rhstest}
\end{figure}
Next, in order to illustrate the importance of evaluating the flux in terms of the projected entropy variables $\Pi_N \bm{v}$, we compute $\delta(t)$ while evaluating the flux function $\bm{f}_S$ directly in terms of the conservative variables $\bm{u}(x)$ and in terms of the entropy-projected conservative variables $\bm{u}\LRp{(\Pi_N \bm{v})(x)}$.  It can be observed from Figure~\ref{fig:rhstest} that $\delta(t)$ is near machine precision when evaluating the flux in terms of the projected entropy variables.  When evaluating the flux function directly in terms of the conservative variables, $\delta(t)$ begins near $10^{-4}$ but grows steadily, blowing up exponentially near $t = 1$.  

\subsubsection{Sod shock tube}

We now examine the behavior of the proposed DG methods for some common one-dimensional test problems.  We begin with the Sod shock tube, which is posed on the domain $[-1/2,1/2]$ with initial conditions
\begin{equation}
\rho = \begin{cases}
1 & x < 0\\
.125 & x \geq 0,
\end{cases} 
\qquad
u = 0, \qquad
p = \begin{cases}
1 & x < 0\\
.1 & x \geq 0.
\end{cases}
\end{equation}
Boundary conditions are enforced by taking the external value $\bm{u}^+$ in the numerical flux to be that of the initial condition at $x = \pm 1$.  The solution develops a left-moving rarefaction, as well as a right moving shock wave and a contact discontinuity.  We simulate the solution until time $T = .2$, without the use of positivity preserving or TVD-type limiters.  For all choices of quadrature tested, the solution diverges when using the entropy conservative flux, which is a result of oscillations in the solution and density and temperature becoming negative.  This is remedied when using the dissipative Lax-Friedrichs flux, for which we do not observe blowup of the solution.  
\begin{figure}
\centering
\subfloat[GLL quadrature]{\includegraphics[width=.475\textwidth]{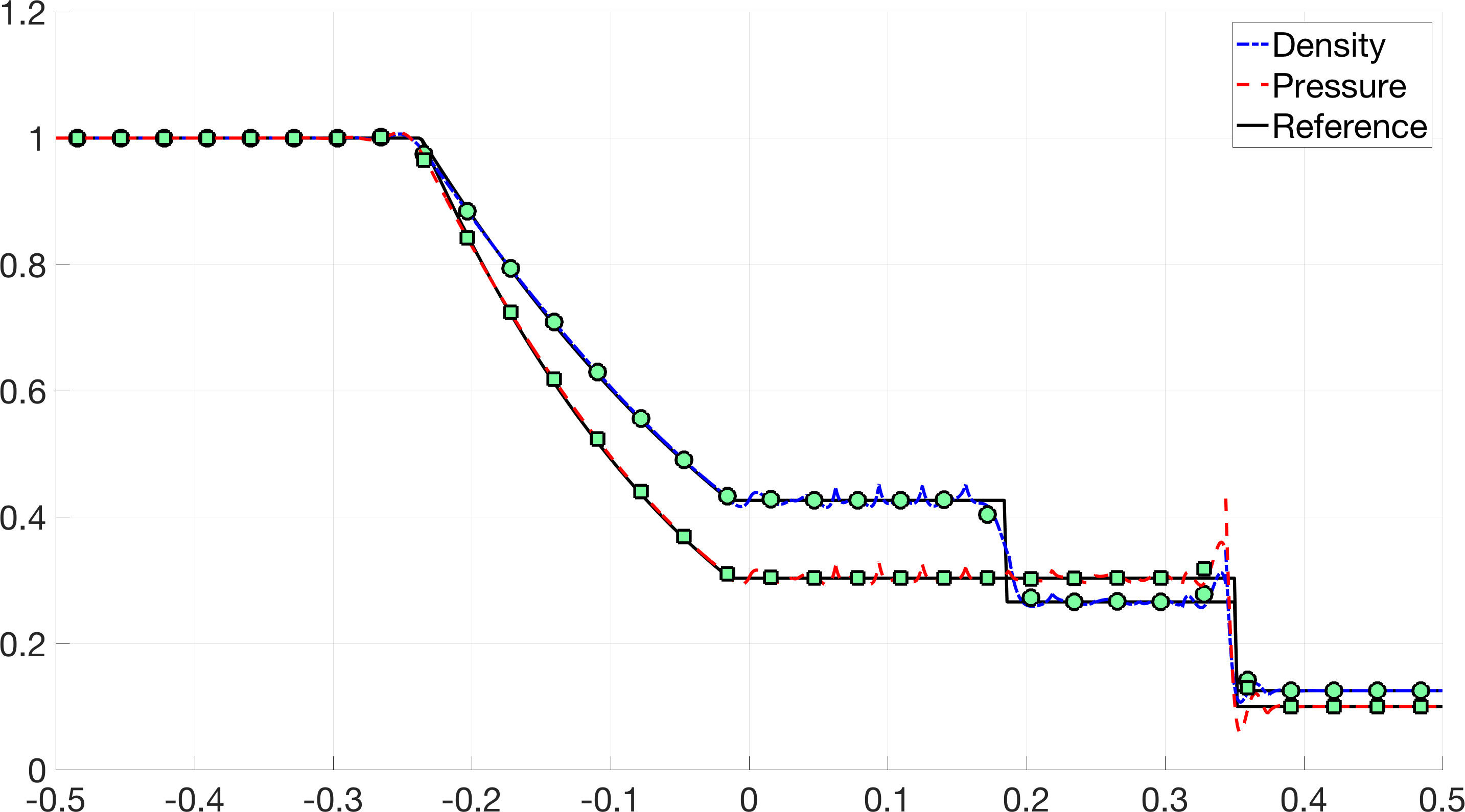}}
\hspace{.1em}
\subfloat[GQ-$(N+2)$ quadrature]{\includegraphics[width=.475\textwidth]{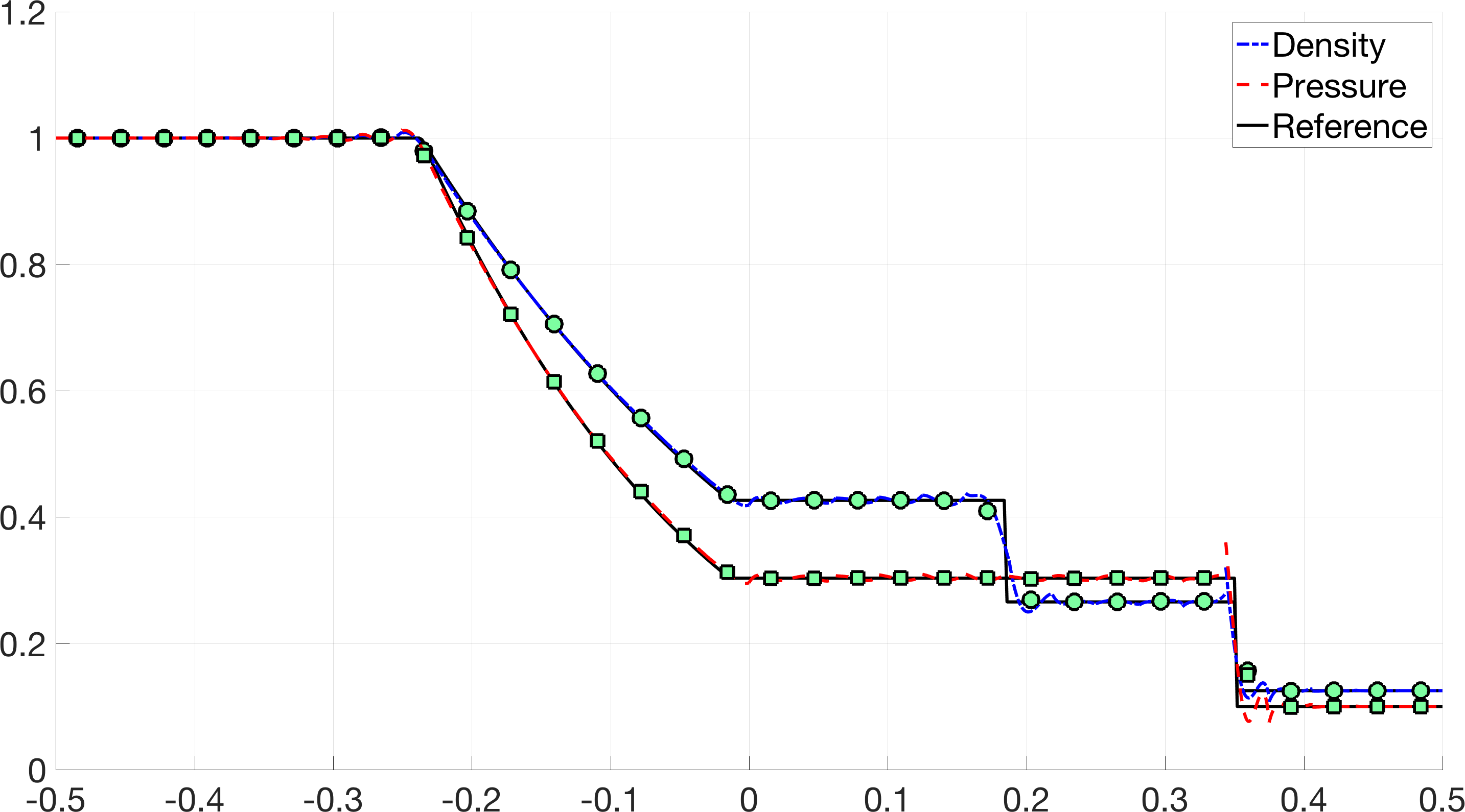}}
\caption{Density and pressure solutions for the Sod shock tube at $T = .2$ for $N = 4$ and $K = 32$ elements.  Results for both GLL and GQ-$(N+2)$ quadratures are shown.  Cell averages are overlaid as filled circles.   Both simulations are run using a CFL of $.125$. }
\label{fig:sod}
\end{figure}

Figure~\ref{fig:sod} shows both the exact solution and the computed density and pressure along with their cell averages.  These results were obtained using a CFL of $.125$ and GLL and GQ-$(N+2)$ quadrature rules.  For both quadratures, the cell averages agree relatively well with the exact solution.  Both solutions contain spurious oscillations, though the oscillations under GQ-$(N+2)$ quadrature appear qualitatively smoother and smaller.  

%\note{Add discussion of Figure~\ref{fig:sod} - reference solution is exact solution computed using a Riemann solver (``PRINCIPLES OF COMPUTATIONAL FLUID DYNAMICS'' by P. Wesseling).  }

\subsubsection{Sine-shock interaction}
\label{sec:sineshock}

The next benchmark problem we consider is the sine-shock interaction problem, which is posed on domain $[-5,5]$ with initial conditions
\begin{align}
\rho(x,0) &= \begin{cases}
3.857143 & x < -4\\
1 + .2\sin(5x) & x \geq -4,
\end{cases} \\
u(x,0) &= \begin{cases}
2.629369 & x < -4\\
0 & x \geq -4,
\end{cases}
\qquad
p(x,0) = \begin{cases}
10.3333 & x < -4\\
1 & x \geq -4.
\end{cases}\nonumber
\end{align}
We simulate the solution using different quadrature rules with $N=4$, $K=40$ elements, and a Lax-Friedrichs flux.  No TVD or positivity-preserving limiters are applied.  A smaller CFL of $.05$ is used, and is necessary to avoid solution divergence when using GQ-$(N+2)$ quadrature.  It should be pointed out that, for GLL quadrature, it is possible to use a larger CFL of $.125$ without observing solution blowup.  The reason for this discrepancy is the sensitivity of the evaluation $\bm{u}\LRp{\Pi_N \bm{v}}$ when $\Pi_N \bm{v}$ differs significantly from $\bm{v}$, and is described in more detail in Section~\ref{sec:instab}.  
\begin{figure}
\centering
%\subfloat[GLL quadrature]{\includegraphics[width=.475\textwidth]{sineShockGLLrhop.png}}
\subfloat[GLL quadrature]{\includegraphics[width=.475\textwidth]{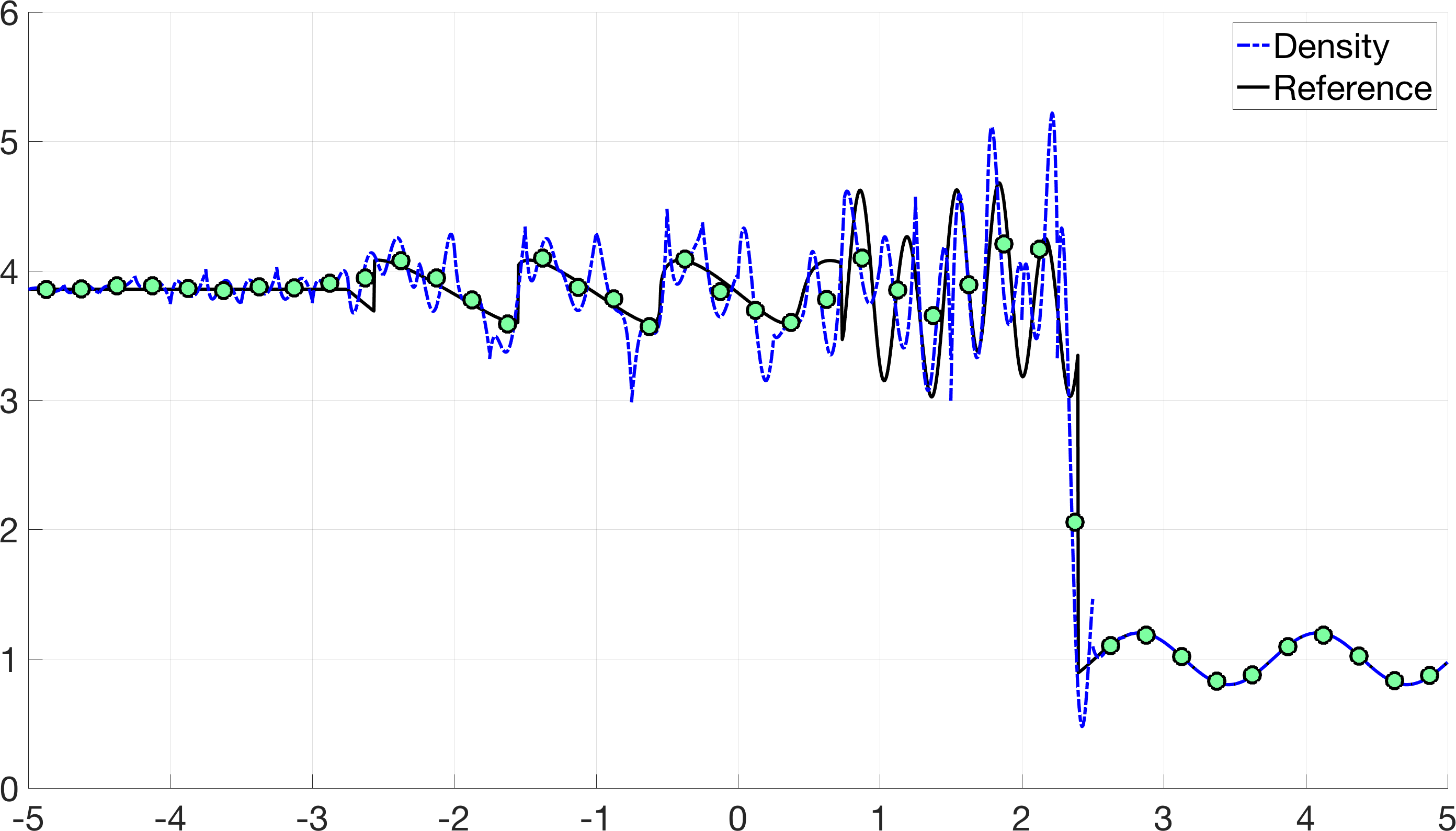}}
\hspace{.1em}
%\subfloat[GQ-$(N+2)$ quadrature]{\includegraphics[width=.475\textwidth]{sineShockGQ2rhop.png}}
\subfloat[GQ-$(N+2)$ quadrature]{\includegraphics[width=.475\textwidth]{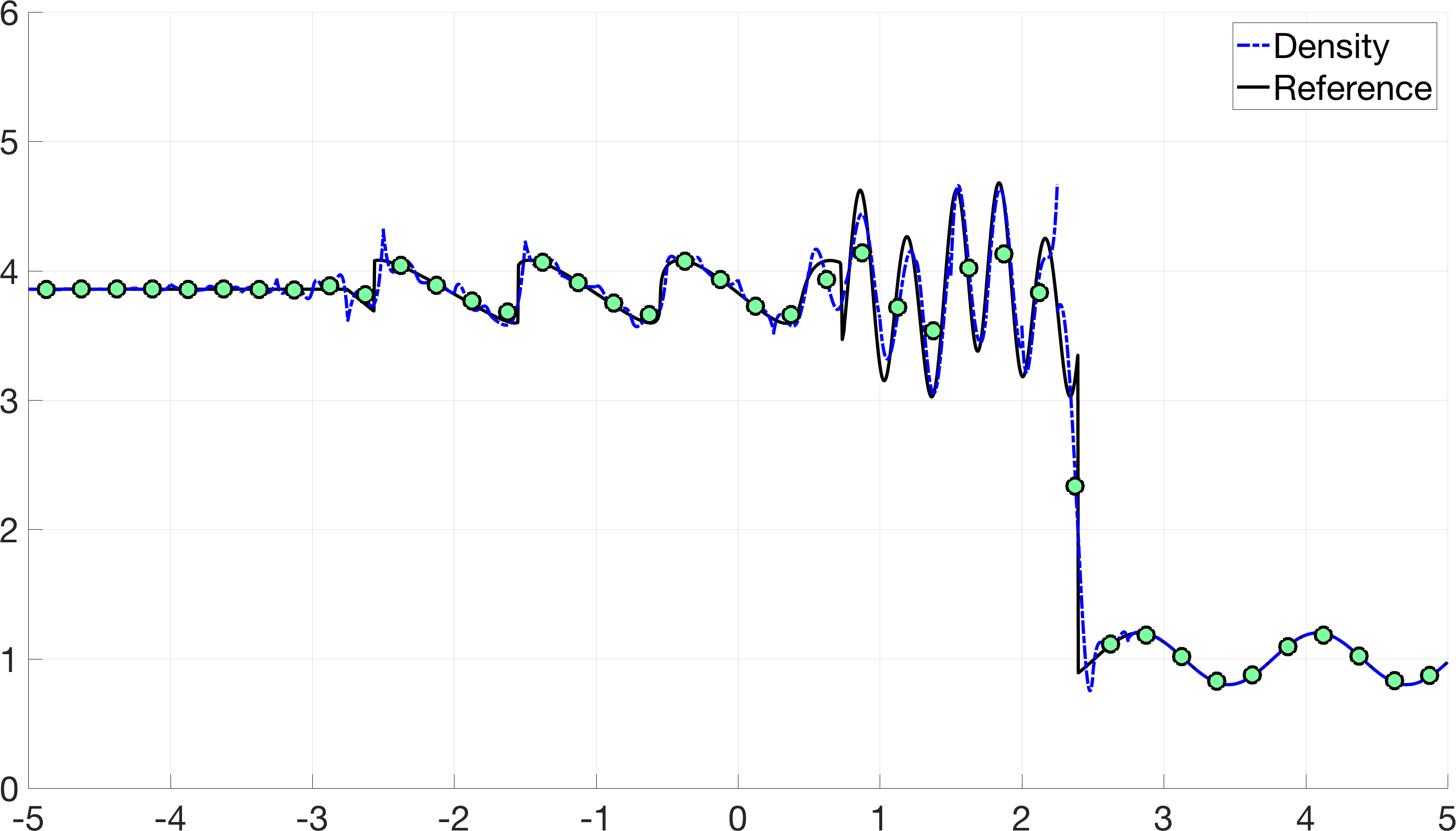}}
\caption{Density solutions for the sine-shock interaction problem at $T = 1.8$ for $N = 4$ and $K = 40$ elements.  Results for both GLL and GQ-$(N+2)$ quadratures are shown.  Cell averages are overlaid as filled circles.  Both simulations are run using a CFL of $.05$, though the GLL simulation is stable for a larger CFL of $.125$.}
\label{fig:sineshock}
\end{figure}

Figure~\ref{fig:sineshock} shows snapshots of the density at final time $T= 1.8$ for both GLL and GQ-$(N+2)$ quadrature, along with a reference solution computed using a 5th order WENO scheme with 25000 cells \cite{shu2009high}.  For both GLL and GQ-$(N+2)$ quadrature, the cell averages are close to the reference solution, though the solutions still contain spurious oscillations resulting from the presence of shocks and discontinuities.  However, as with the Sod shock tube, we observe that these oscillations are smoother and smaller in amplitude for the choice of GQ-$(N+2)$ quadrature.  

%\note{Add reference solution using 25000 cell 5th order WENO method \cite{shu2009high}.  Simulations in Figure~\ref{fig:sineshock} use CFL of $.05$; necessary for GQ-$(N+2)$.  GLL can still run at CFL of $.125$.  }

\subsubsection{Sensitivity of evaluation in terms of the projected entropy variables}
\label{sec:instab}

The numerical experiments in previous sections suggest that solutions computed using Gauss quadrature rules can be more accurate than those computed using GLL quadratures.  However, we also observe that, for the sine-shock interaction problem, a larger CFL of $.125$ can be taken when using GLL quadrature, whereas a much smaller CFL of $.05$ is required to prevent solution blowup when using GQ-$(N+1)$ and GQ-$(N+2)$ quadratures.  %The reason for this discrepancy between Gauss-Lobatto and Gauss quadratures is the presence of boundary nodes in GLL quadratures.  
For GQ-$(N+1)$ and GQ-$(N+2)$ quadratures, solution spikes can occur when evaluating the entropy-projected conservative variables at surface points.  The use of GLL quadrature avoids this phenomena because of two factors: the equivalence between interpolation and projection under a $(N+1)$ point quadrature rule, and the presence of boundary points in GLL quadrature.  

Figure~\ref{fig:proj} shows snapshots of different variables for $N=2$ and $K = 20$ elements at the fifth Runge-Kutta stage of the first timestep (just prior to the detection of negative density and pressure values) using a GQ-$(N+2)$ quadrature rule.  Discrepancies between the evaluated and projected values of the entropy variables $v_1, v_2, v_3$ at element boundaries are present, and while these discrepancies are not extremely large, they produce large spikes in the conservative variables due to the sensitivity of the nonlinear evaluation $\bm{u}\LRp{\Pi_N \bm{v}}$.  In particular, because $v_3$ appears in the denominators of entropy and $\rho e$ (as functions of the entropy variables), values of $v_3$ near zero result in large values of density and energy.  These spikes result in large oscillations in the solution if the CFL is too large, which eventually cause the density and pressure to become negative at quadrature or boundary points.  %\footnote{For the sine-shock interaction problem, these spikes appear to be related to the large jump in the pressure initial condition.  Instabilities were not observed for the sine-shock interaction when using an initial pressure condition with a smaller magnitude jump, or for other problems with discontinuous non-zero initial profiles, such as the modified Sod problem \cite{ranocha2017comparison}. }
\begin{figure}
\centering
\subfloat[$v_1$ and $\Pi_N(v_1)$]{\includegraphics[width=.32\textwidth]{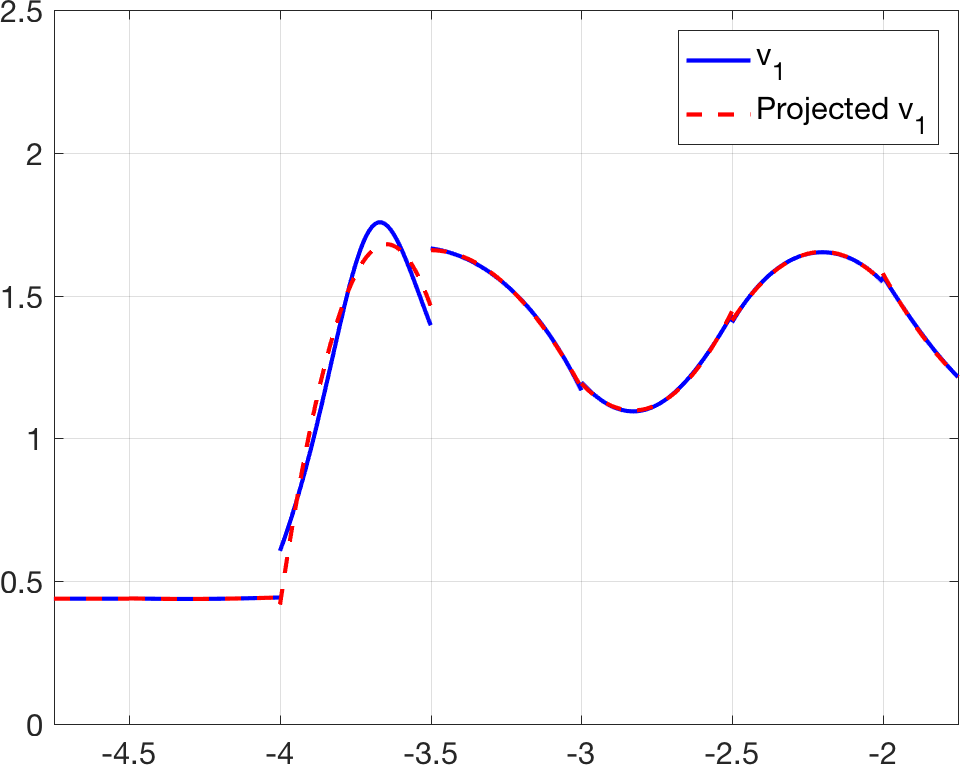}}\label{subfig:v1}
\hspace{.1em}
\subfloat[$v_2$ and $\Pi_N(v_2)$]{\includegraphics[width=.32\textwidth]{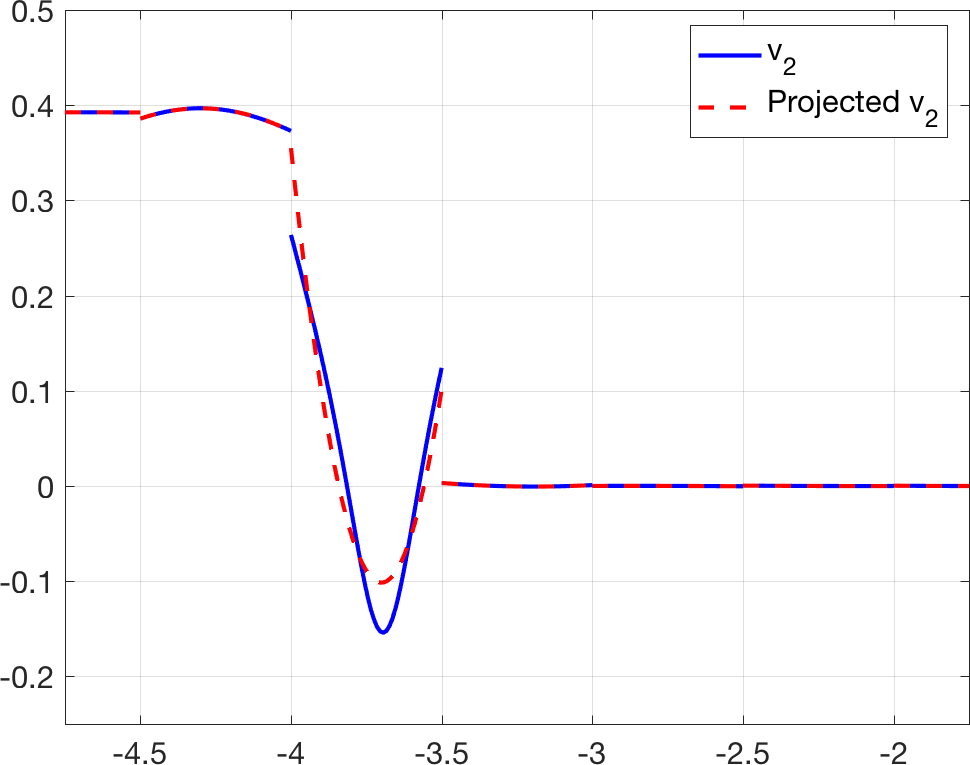}}\label{subfig:v2}
\hspace{.1em}
\subfloat[$v_3$ and $\Pi_N(v_3)$]{\includegraphics[width=.32\textwidth]{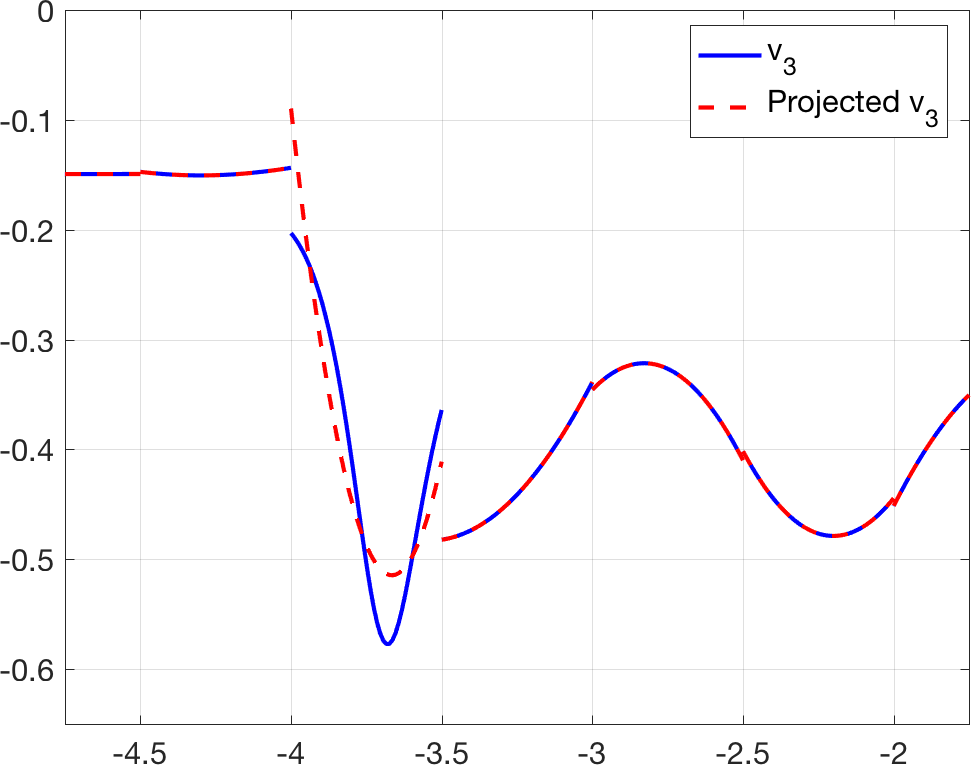}}\label{subfig:v3}\\
\subfloat[$\rho(x)$ and $\rho\LRp{\Pi_N \bm{v}}$]{\includegraphics[width=.32\textwidth]{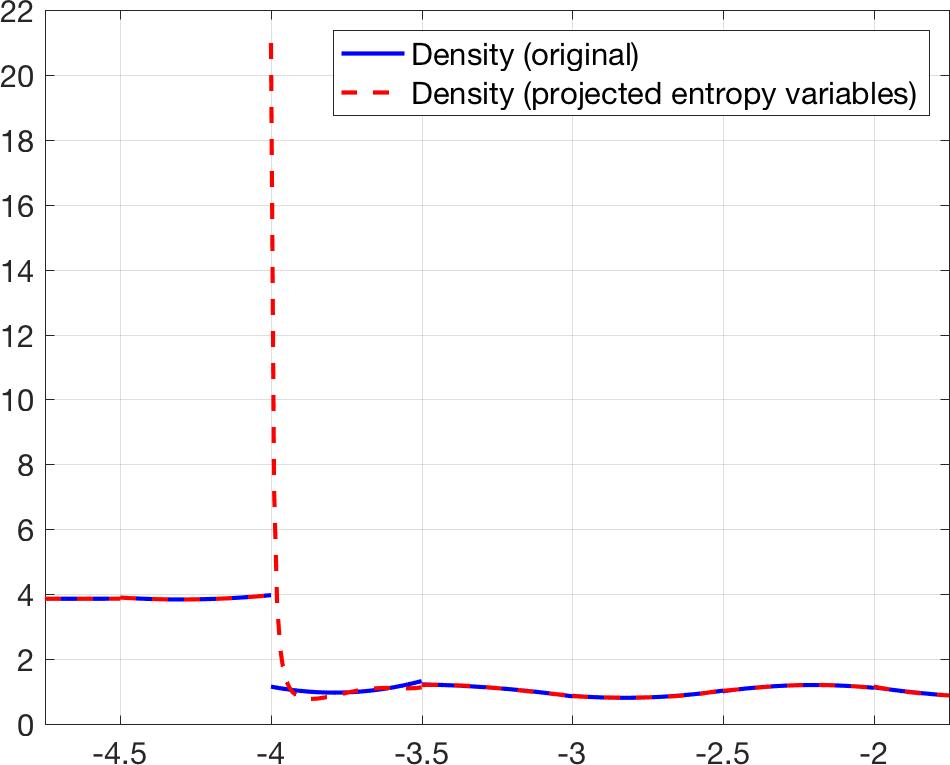}}\label{subfig:u1}
\hspace{.1em}
\subfloat[$u(x)$ and $u\LRp{\Pi_N \bm{v}}$]{\includegraphics[width=.32\textwidth]{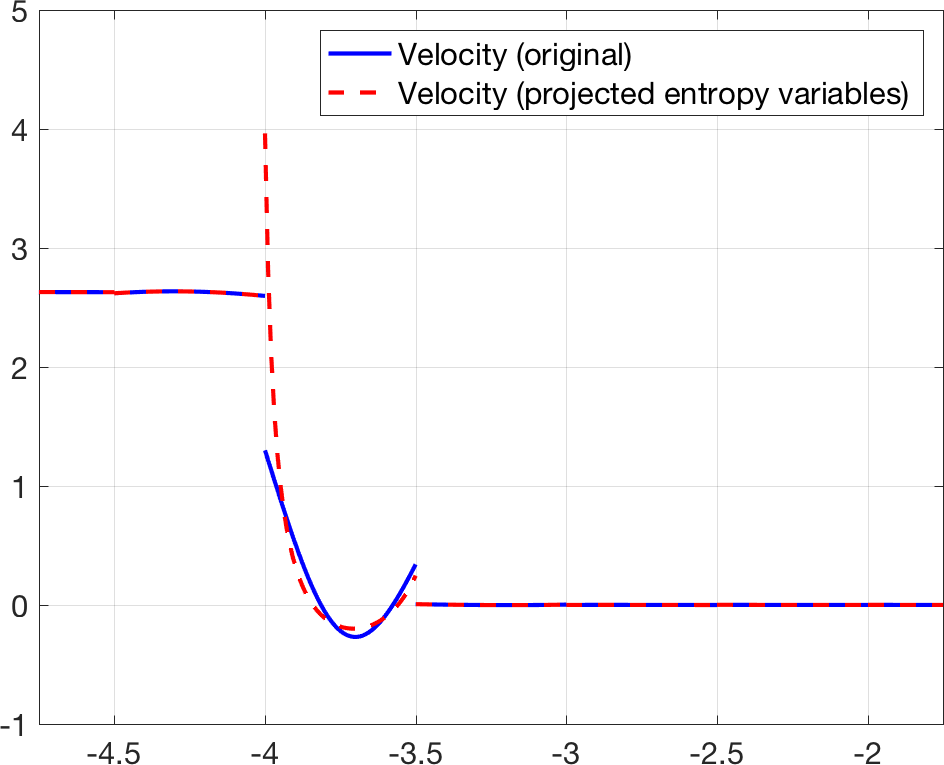}}\label{subfig:u2}
\hspace{.1em}
\subfloat[$p(x)$ and $p\LRp{\Pi_N \bm{v}}$]{\includegraphics[width=.32\textwidth]{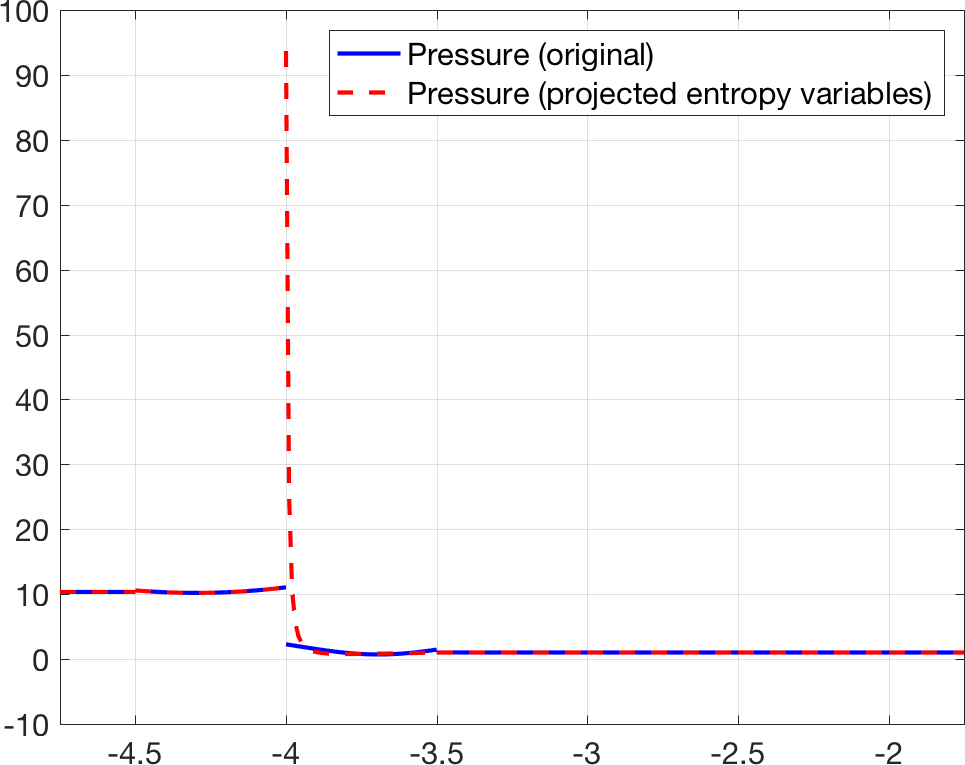}}\label{subfig:u3}
\caption{Top row: comparisons of entropy variables (evaluated directly as functions of conservative variables) and projected entropy variables (using $N=2$, $K=20$ elements, and a GQ-$(N+2)$ quadrature rule) for the sine-shock interaction before the detection of negative density and pressures.  Bottom row: comparisons of the conservative variables and their evaluations in terms of the projected entropy variables.  {The overshoot in the $L^2$ projection of $v_3$ produces a large spike in the entropy-projected conservation variables $\bm{u}\LRp{\Pi_N \bm{v}}$.}}
\label{fig:proj}
\end{figure}

Qualitatively similar behavior is observed when using a GQ-$(N+1)$ quadrature rule, though negative density and pressure values occur at a slightly later time.  In this case, the $L^2$ projection reduces to interpolation at interior Gauss points.  This still implies that $\bm{u}\LRp{\Pi_N\bm{v}}$ does not necessarily agree with $\bm{u}$ at element boundaries, and this discrepancy again leads to spikes in the conservative variables similar to those presented in Figure~\ref{fig:proj}.  These experiments suggest that the lack of control of boundary values in $\Pi_N \bm{v}$ leads to large oscillations or spikes in the boundary values of $\bm{u}\LRp{\Pi_N\bm{v}}$.  The presence of boundary nodes in {the $(N+1)$ point} GLL quadrature avoids this issue: $\bm{u}\LRp{\Pi_N\bm{v}}$ and $\bm{u}$ agree at element boundaries due to the fact that quadrature-based $L^2$ projection under a GLL quadrature is equivalent to interpolation at GLL points.\footnote{{We note that the presence of boundary nodes in a quadrature does not necessarily reduce the sensitivity of the evaluation of the entropy-projected conservative variables $\bm{u}\LRp{\Pi_N\bm{v}}$.  This due to the fact that the projection is computed using weighted averages of the points: while GLL quadrature rules contain endpoints, the quadrature weights are small at boundary points, implying that boundary values factor less  strongly into the quadrature-based projection than interior points.  Additionally, while Gauss quadrature rules do not contain endpoints, the variation of the Gauss quadrature weights is slightly less than that of GLL quadrature weights. }}

This scenario illustrates some of the more subtle differences between the use of GLL and GQ quadratures.  We note that bound-preserving and TVD limiters \cite{zhang2010positivity, zhang2012maximum} are typically implemented in practice to detect and control spikes of the kind observed in these numerical experiments.  However, slope-limiting the conservative variables $\bm{u}$ may still result in spikes being generated when evaluating $\bm{u}\LRp{\Pi_N\bm{v}}$ (unless the slope is set to zero such that $\bm{u}$ is constant).  A more nuanced approach will likely require modifying the conservative variables $\bm{u}$ based on the projected entropy variables $\Pi_N \bm{v}$ to control oscillations while maintaining entropy conservation.  These strategies will be the focus in future work.  

\subsection{Two-dimensional experiments}
\label{sec:2d}

{
We now present numerical experiments in two dimensions to verify the semi-discrete entropy conservation and accuracy of the presented schemes for the two-dimensional compressible Euler equations:
\begin{align}
\pd{\rho}{t} + \pd{\LRp{\rho u}}{x} + \pd{\LRp{\rho v}}{y} &= 0,\\
\pd{\rho u}{t} + \pd{\LRp{\rho u^2 + p }}{x} + \pd{\LRp{\rho uv}}{y} &= 0,\nonumber\\
\pd{\rho v}{t} + \pd{\LRp{\rho uv}}{x} + \pd{\LRp{\rho v^2 + p }}{y} &= 0,\nonumber\\
\pd{E}{t} + \pd{\LRp{u(E+p)}}{x} + \pd{\LRp{v(E+p)}}{x}&= 0.\nonumber
\end{align}
The following experiments use a triangular volume quadrature rule from \cite{xiao2010quadrature} which is exact for polynomials of degree $2N$, with $(N+1)$ point one-dimensional Gauss quadrature rules on the faces.  $L^2$ errors are evaluated using a volume quadrature rule of degree $2N+2$.  

In two dimensions, the pressure is $p = (\gamma-1)\LRp{E - \frac{1}{2}\rho (u^2+v^2)}$, and the specific internal energy is $\rho e = E - \frac{1}{2}\rho (u^2+v^2)$.  The formula for the entropy $U(\bm{u})$ is the same as in one dimension.  
 The entropy variables in two dimensions are 
\begin{align}
v_1 = \frac{\rho e (\gamma + 1 - s) - E}{\rho e}, \qquad v_2 = \frac{\rho u}{\rho e}, \qquad v_3 = \frac{\rho v}{\rho e}, \qquad v_4 = -\frac{\rho}{\rho e}.
\end{align}
The conservation variables in terms of the entropy variables are given by
\begin{equation}
\rho = -(\rho e) v_4, \qquad \rho u = (\rho e) v_2, \qquad \rho v = (\rho e) v_3, \qquad E = (\rho e)\LRp{1 - \frac{{v_2^2+v_3^2}}{2 v_4}},
\end{equation}
where $\rho e$ and $s$ in terms of the entropy variables are 
\begin{equation}
\rho e = \LRp{\frac{(\gamma-1)}{\LRp{-v_4}^{\gamma}}}^{1/(\gamma-1)}e^{\frac{-s}{\gamma-1}}, \qquad s = \gamma - v_1 + \frac{{v_2^2+v_3^2}}{2v_4}.
\end{equation}
The entropy conservative numerical fluxes for the two-dimensional compressible Euler equations are given by
\begin{align}
&f^1_{1,S}(\bm{u}_L,\bm{u}_R) = \avg{\rho}^{\log} \avg{u},& &f^1_{2,S}(\bm{u}_L,\bm{u}_R) = \avg{\rho}^{\log} \avg{v},&\\
&f^2_{1,S}(\bm{u}_L,\bm{u}_R) = f^1_{1,S} \avg{u} + p_{\rm avg},&  &f^2_{2,S}(\bm{u}_L,\bm{u}_R) = f^1_{2,S} \avg{u},&\nonumber\\
&f^3_{1,S}(\bm{u}_L,\bm{u}_R) = f^2_{2,S},& &f^3_{2,S}(\bm{u}_L,\bm{u}_R) = f^1_{2,S} \avg{v} + p_{\rm avg},&\nonumber\\
&f^4_{1,S}(\bm{u}_L,\bm{u}_R) = \LRp{\frac{p_{\rm avg}^{\log}}{\gamma -1} + p_{\rm avg} + \frac{\nor{\bm{u}}^2_{\rm avg}}{2}}\avg{u},& &f^4_{2,S}(\bm{u}_L,\bm{u}_R) = \LRp{\frac{p_{\rm avg}^{\log}}{\gamma -1} + p_{\rm avg} + \frac{\nor{\bm{u}}^2_{\rm avg}}{2}}\avg{v},& \nonumber
\end{align}
where we have defined
\begin{equation}
p_{\rm avg} = \frac{\avg{\rho}}{2\avg{\beta}}, \qquad p_{\rm avg}^{\log} = \frac{\avg{\rho}^{\log}}{2\avg{\beta}^{\log}}, \qquad \nor{\bm{u}}^2_{\rm avg} = 2(\avg{u}^2 + \avg{v}^2) - \LRp{\avg{u^2} +\avg{v^2}}.  
\end{equation}

\subsubsection{Discontinuous profile on a periodic domain}

We again examine semi-discrete conservation of entropy by evolving a discontinuous initial profile to final time $T=2$ on the square domain $[-1,1]^2$.  We set the initial velocities to be zero, and initialize the density and pressure as a discontinuous square pulse
\begin{equation}
\rho(\bm{x},t) = \begin{cases}
3 & \LRb{x} < 1/2 \text{ and } \LRb{y} < 1/2\\
2 & \text{otherwise},
\end{cases} \qquad 
u(\bm{x},t) = v(\bm{x},t) = 0, \qquad
p(\bm{x},t) = \rho^\gamma.
\label{eq:discontin}
\end{equation}
Figure~\ref{fig:ec2d} shows the evolution of entropy over time for an entropy conservative scheme without additional interface dissipation.  As in the one-dimensional case, we observe that the entropy increases over time, but decreases with the time-step. Unlike the one-dimensional case, we observe a jump in the change in entropy near time $t = 1.86$, though this does not appear to affect the order of convergence of the change in entropy at the final time $\Delta U(T)$, which converges to zero as $O(\Delta t^4)$ (which corresponds to the order of the time-stepping scheme used).
}
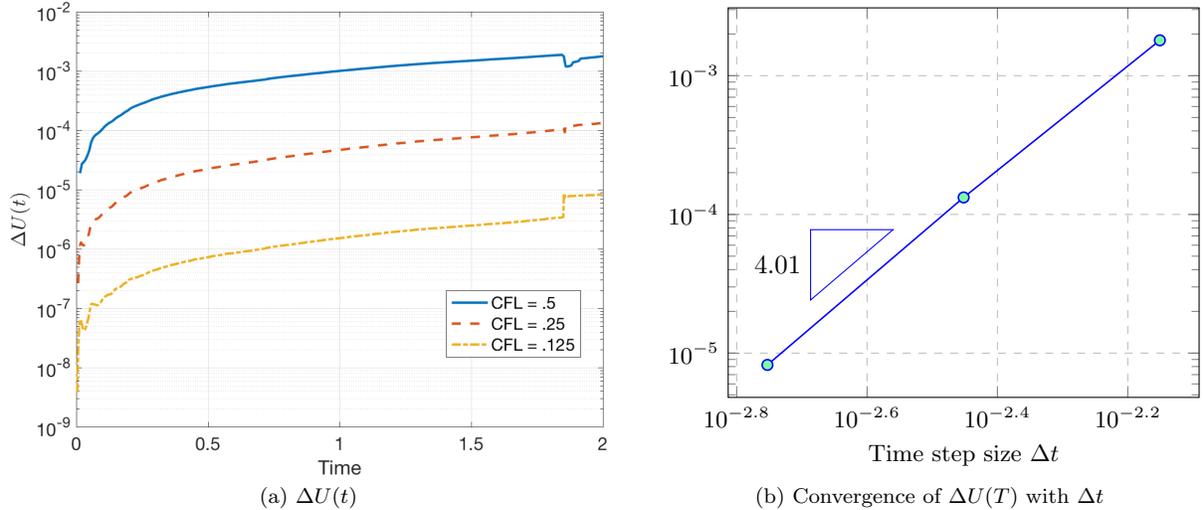
\begin{figure}
\centering
\subfloat[$\Delta U(t)$]{
\includegraphics[width=.48\textwidth]{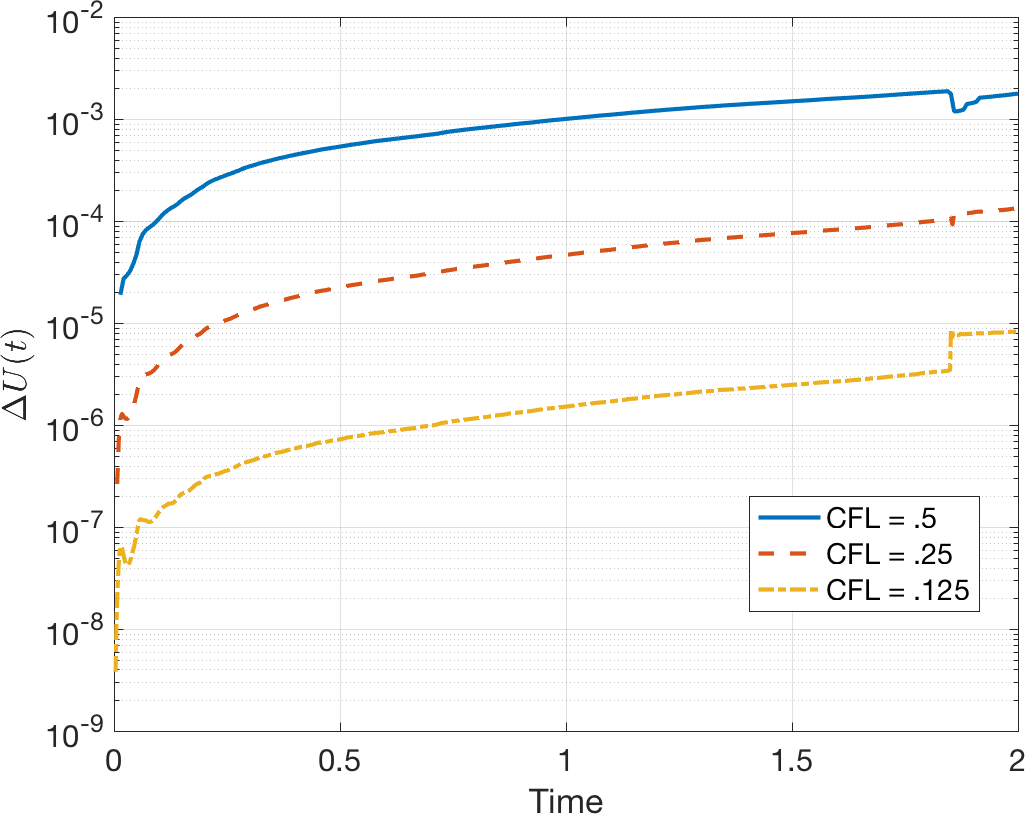}
}
\hspace{1em}
\subfloat[Convergence of $\Delta U(T)$ with $\Delta t$]{
\begin{tikzpicture}
\begin{loglogaxis}[
    width=.475\textwidth,
    xlabel={Time step size $\Delta t$},
%    ylabel={ $\Delta U(t)$}, 
%    xticklabels={$0.0071$, $0.0035$, $0.0018$},
%    xmin=.0625, xmax=.75,
%    ymin=1e-6, ymax=.01,
    legend pos=south east, legend cell align=left, legend style={font=\tiny},	
    xmajorgrids=true, ymajorgrids=true, grid style=dashed,
]
%\pgfplotsset{
%cycle list={{blue, dashed, mark=*}, {red, mark=square*}}
%}
%\addlegendimage{no markers,blue}
%\addlegendimage{no markers,red}

\addplot+[semithick, mark options={solid, fill=markercolor}]
coordinates{(0.0070671378 ,0.0018)(0.0035335689 ,0.00013255)(0.0017667845,8.2229e-06)};
\logLogSlopeTriangleFlip{0.35}{0.175}{0.25}{4.01}{blue}

\end{loglogaxis}
\end{tikzpicture}
}
\caption{Change in entropy $\Delta U(t) = \LRb{U(t)-U(0)}$ over time and convergence of $\Delta U(T)$ at the final time $T$ for a $N=4$ and a mesh of $8\times 8$ subdivided quadrilaterals.  These results use a volume quadrature rule exact for degree $2N$ polynomials, and surface quadrature rules constructed using one-dimensional $(N+1)$ point Gauss quadratures.   The convergence of the change in entropy $\Delta U(t)$ at the final time $T = 2$ converges to zero as $O\LRp{\Delta t^{4}}$, matching the order of the 4th order time-stepper used.}
\label{fig:ec2d}
\end{figure}

\subsubsection{Isentropic vortex problem}

{
Next, we examine high order convergence in two dimensions using the vortex problem as set up in \cite{shu1998essentially, crean2017high}.  The analytical solution is given as 
\begin{align}
\rho(\bm{x},t) &= \LRp{1 - \frac{\frac{1}{2}(\gamma-1)(\beta e^{1-r(\bm{x},t)^2})^2}{8\gamma \pi^2}}^{\frac{1}{\gamma-1}}, \qquad p = \rho^{\gamma},\\
u(\bm{x},t) &= 1 - \frac{\beta}{2\pi} e^{1-r(\bm{x},t)^2}(y-y_0), \qquad v(\bm{x},t) = \frac{\beta}{2\pi} e^{1-r(\bm{x},t)^2}(y-y_0),\nonumber
\end{align}
where $u, v$ are the $x$ and $y$ velocity and $r(\bm{x},t) = \sqrt{(x-x_0-t)^2 + (y-y_0)^2}$.  Here, we take $x_0 = 5, y_0 = 0$ and $\beta = 5$.  

We solve on a periodic rectangular domain $[0, 20] \times [-5,5]$ until final time $T=5$.  The mesh is constructed by first building a mesh of uniform quadrilateral elements and subdividing each quadrilateral into two uniform triangles.  We estimate the $L^2$ errors and their rates of convergence, which are shown in Figure~\ref{fig:converge2d}.  We observe $L^2$ optimal $O(h^{N+1})$ rates of convergence for $N = 1,\ldots, 3$, while for $N = 4$ we observe a rate of convergence between $O(h^{N+1})$ and $O(h^{N+1/2})$.  We note that the rate of $O(h^{N+1/2})$ is the theoretically proven rate of convergence for upwind DG methods on general meshes applied to linear hyperbolic problems  \cite{johnson1986analysis,cockburn2008optimal}.  We note that this observed rate does not change significantly if the time-step is halved (improving from $4.785$ to $4.8$), suggesting that this slight degradation in convergence rate is not due to temporal errors.  
}

\begin{figure}
\centering
\begin{tikzpicture}
\begin{loglogaxis}[
    width=.525\textwidth,
    xlabel={Mesh size $h$},  ylabel={$L^2$ errors}, 
    xmin=.1, xmax=7,
    ymin=1e-5, ymax=1,   
    legend pos = south east, legend cell align=left, legend style={font=\tiny},	
    xmajorgrids=true, ymajorgrids=true, grid style=dashed
] 
\pgfplotsset{
cycle list={{blue,  mark=*},{red, mark=square*},{cyan, mark=triangle*},{black, mark=diamond*}}
}
%\pgfplotsset{cycle list={{blue, dashed, mark=*}, {red, mark=square*}}}
\addplot+[semithick, mark options={solid, fill=markercolor}]
coordinates{(5,0.152443)(2.5,0.509207)(1.25,0.445656)(0.625,0.166519)(0.3125,0.036208)};
%[yshift=-2pt] node[left, pos=1.025, color=black] {$N = 1$};
\logLogSlopeTriangle{0.4}{0.1}{0.72}{2.2}{blue}

\addplot+[semithick, mark options={solid, fill=markercolor}]
coordinates{(5,0.484478)(2.5,0.46951)(1.25,0.125591)(0.625,0.015585)(0.3125,0.001791)};
%[yshift=-3pt] node[left, pos=1.025, color=black] {$N = 2$};
\logLogSlopeTriangleFlip{0.375}{0.1}{0.49}{3.12}{red}

\addplot+[semithick, mark options={solid, fill=markercolor}]
coordinates{(5,0.557086)(2.5,0.260022)(1.25,0.042815)(0.625,0.003611)(0.3125,0.000213)};
%[yshift=-4pt] node[left, pos=1.025, color=black] {$N = 3$};
\logLogSlopeTriangleFlip{0.35}{0.0785}{0.305}{4.084}{cyan}

\addplot+[semithick, mark options={solid, fill=markercolor}]
coordinates{(5,0.463269)(2.5,0.169843)(1.25,0.017172)(0.625,0.000965)(0.3125,3.5e-05)};
\logLogSlopeTriangle{0.4}{0.1}{0.1275}{4.785}{black}

\legend{$N=1$, $N=2$, $N=3$, $N=4$}

\end{loglogaxis}
\end{tikzpicture}
\caption{Convergence of $L^2$ errors for the isentropic vortex problem at final time $T = 5$ for various $N$.  }
\label{fig:converge2d}
\end{figure}
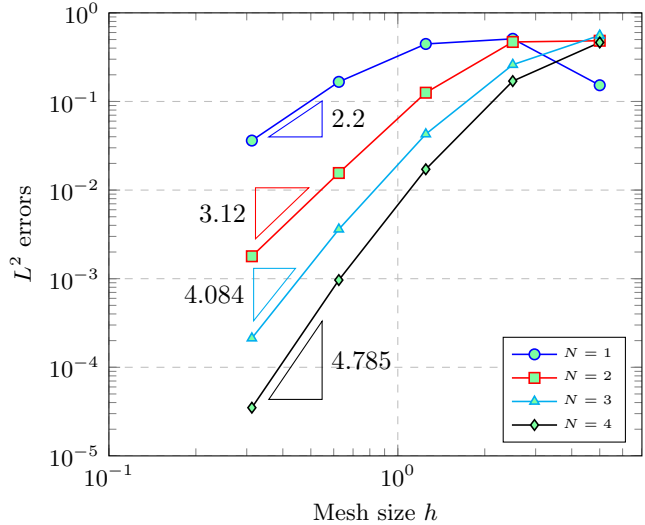

\subsubsection{A two-dimensional Riemann problem}

{
Finally, we present numerical results for a two-dimensional Riemann problem \cite{lax1998solution, kurganov2002solution, liska2003comparison} using an entropy stable high order DG scheme using Lax-Friedrichs penalization.  No additional stabilization, artificial dissipation, or limiting is applied.  The problem is posed on the square domain $[-.5, .5]^2$ with piecewise constant initial conditions
\begin{align}
&\rho(\bm{x}) = .5313,&  &u(\bm{x}) = 0,& &v(\bm{x}) = 0,&  &p(\bm{x}) = .4,&  &\bm{x} \in [0, .5]\times[0, .5],&\\
&\rho(\bm{x}) = 1,& &u(\bm{x}) = .7276,&  &v(\bm{x}) = 0,&  &p(\bm{x}) = 1,&  &\bm{x} \in [-.5, 0]\times[0, .5],&\nonumber\\
&\rho(\bm{x}) = .8,&  &u(\bm{x}) = 0,&  &v(\bm{x}) = 0,&  &p(\bm{x}) = 1,&  &\bm{x} \in [-.5, 0]\times[-.5,0],&\nonumber\\
&\rho(\bm{x}) = 1,&  &u(\bm{x}) = 0,&  &v(\bm{x}) = .7276,&  &p(\bm{x}) = 1,&  &\bm{x} \in [0,.5]\times[-.5,0].&\nonumber
\end{align}
In \cite{chen2017entropy}, the boundary condition is imposed by computing the exact shock speed using a one-dimensional Riemann solver on each boundary.  Here, we solve instead on an enlarged quadrilateral domain $[-1,1]^2$ with periodic boundary conditions, but consider the solution only on the smaller physical domain $[-.5, .5]^2$.  The size of the enlarged domain is chosen such that solution within the fictitious portion of the domain $[-1,1]^2\setminus [-.5,.5]^2$ does not pollute the solution within the smaller subdomain $[-.5, .5]^2$ at the final time $T = 1/4$.  The piecewise constant initial conditions are initialized as piecewise constants in each quadrant of the enlarged domain.  

The enlarged domain $[-1,1]^2$ is meshed using $64$ uniform quadrilaterals per direction.  A uniform triangular mesh of $8192$ elements is constructed by subdividing each quadrilateral into two triangles.  Figure~\ref{fig:riemann2d} shows the numerical solution on this mesh for $N=3$ using a CFL of $.125$ on both the enlarged and physical domains.  The result on the physical domain shows qualitative agreement with results presented in the literature.  
}

\begin{figure}
\centering
\subfloat[Enlarged domain $\LRs{-1,1}^2$ containing $64\times 64$ elements]{\includegraphics[width=.45\textwidth]{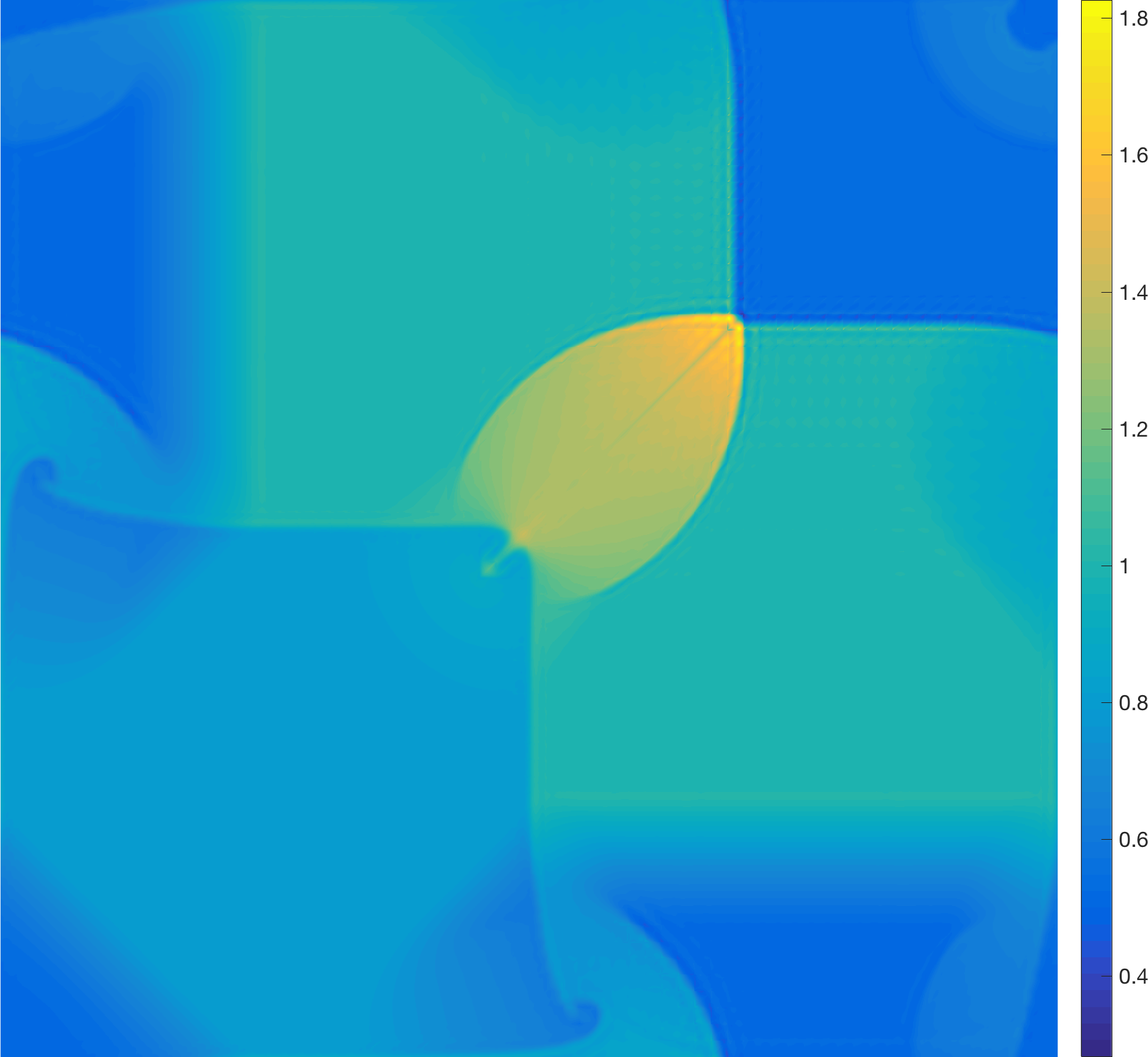}}
\hspace{2em}
\subfloat[Physical domain $\LRs{-.5,.5}^2$, $32\times 32$ elements]{\includegraphics[width=.45\textwidth]{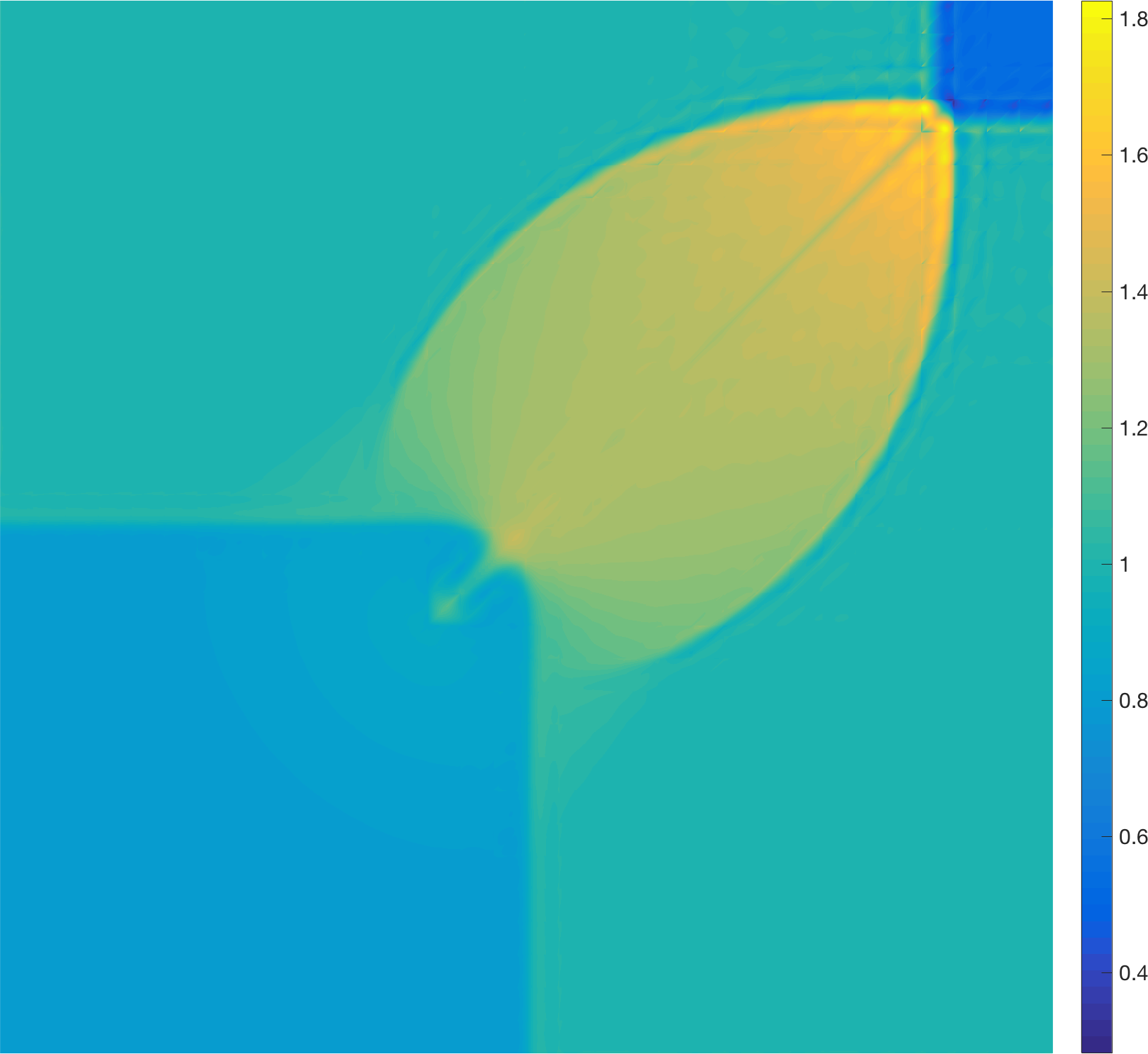}}
\caption{Numerical solution of the Riemann problem at time $T = .25$ using an entropy stable DG method with $N=3$.  These results use a volume quadrature rule of degree $2N$ and surface quadrature rules using one-dimensional GQ-$(N+1)$ quadratures.  }
\label{fig:riemann2d}
\end{figure}

\subsection{{On accuracy and computational cost}}

{While we have shown how to construct high order schemes which are discretely entropy conservative or entropy stable, we have not shown theoretically that these methods are high order accurate for general conservation laws.  We first note that, for conservation laws where $\bm{v} = \bm{u}$ (i.e.\ the entropy variables are the same as the conservative variables), the entropy-projected conservation variables $\tilde{\bm{u}} = \bm{u}\LRp{\Pi_N \bm{v}}$ reduce to the conservative variables.  By (\ref{eq:Drecovery}), 
\begin{equation}
 \LRs{\begin{array}{cc} 
\bm{P}_q & \bm{L}_q\end{array}} \bm{D}_N 
\end{equation}
is at least a degree $N$ approximation to the derivative.  Using arguments in \cite{chen2017entropy}, the truncation error of the proposed schemes is then at least $O\LRp{h^N}$ in smooth regions.  

When $\bm{v}\neq \bm{u}$, we require additional conditions for high order accuracy.  Ensuring that the truncation error is degree $r$ accurate in this more general case would require at least that 
\begin{equation}
\LRb{\bm{u} - \bm{u}\LRp{\Pi_N\bm{v}}} = O\LRp{h^{r}}
\end{equation}
in smooth regions.  In other words, we would require that the difference between $\tilde{\bm{u}}$ and $\bm{u}$ is high order accurate where $\bm{u}$ is a high order accurate approximation to the solution.  While we have not yet proven this, numerical experiments indicate that this is satisfied with $r = N+1$ in both one and two dimensions.  

We compute the $L^2$ error between $\bm{u}$ and $\bm{u}\LRp{\Pi_N\bm{v}}$ on the one-dimensional interval $[-1,1]$ and the two-dimensional domain $[-1,1]^2$.  The conservation variables are set to the polynomial $L^2$ projections of smooth functions.  In one dimension, we set
\begin{equation}
\rho(\bm{x}) = \rho_0 + e^{x/2}\sin(\pi x), \qquad \rho u(\bm{x}) = \sin(\pi x), \qquad E(\bm{x}) = E_0 + \frac{\rho}{2} u^2,
\end{equation}
while in two dimensions, we set 
\begin{align}
\rho(\bm{x}) &= \rho_0 + e^{(x+y)/2}\sin(\pi x)\sin(\pi y), \\
 \rho u(\bm{x}) &= \rho v(\bm{x}) = \sin(\pi x)\sin(\pi y), \qquad E(\bm{x}) = E_0 + \frac{\rho}{2} (u^2 + v^2).\nonumber
\end{align}
Uniform meshes of $K_{\rm 1D}$ elements are used in both cases.  The two-dimensional domain is first meshed using a uniform quadrilateral mesh of $K_{\rm 1D}\times K_{\rm 1D}$ elements in each direction, and a uniform triangular mesh is constructed by bisecting each quadrilateral element into two triangles.  In one dimension, a $(N+2)$ point Gauss quadrature rule is used, while in two dimensions, a quadrature rule exact for degree $2N$ polynomials is used.  The error is evaluated using a quadrature rule of two degrees higher.  

Figure~\ref{fig:accUV} plots the $L^2$ error $\nor{\bm{u} - \bm{u}\LRp{\Pi_N\bm{v}}}_{L^2\LRp{\Omega}}$ between the conservative and entropy-projected conservative variables in one and two dimensions for $\rho_0 = E_0 = 2$.  Convergence rates of $O(h^{N+1})$ are observed for  $N = 1,\ldots, 5$.  This implies that the entropy-projected conservative variables approximate the conservative variables with high order accuracy, and that the flux is approximated with high order accuracy.
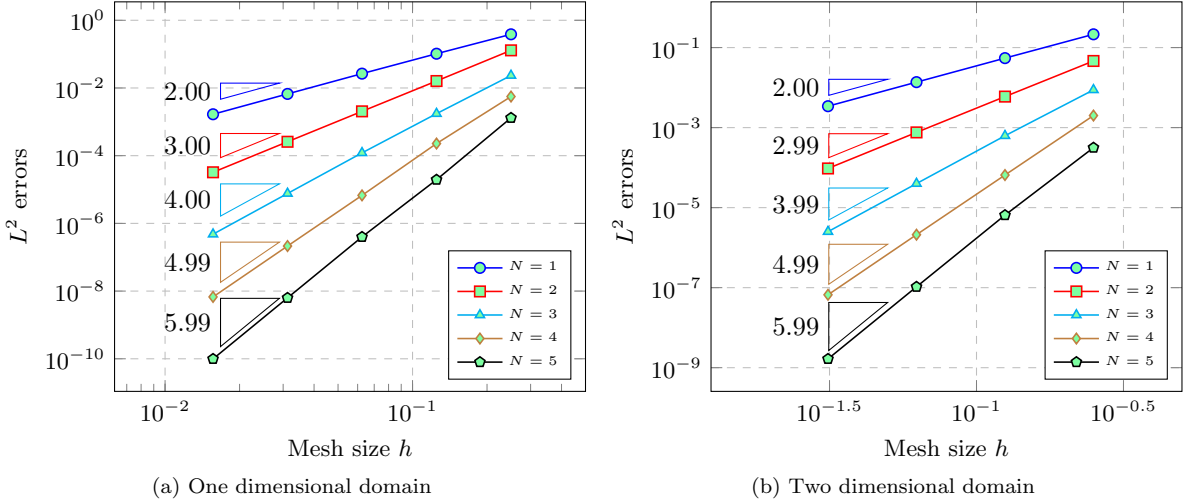
\begin{figure}
\centering
\subfloat[One dimensional domain]{
\begin{tikzpicture}
\begin{loglogaxis}[
    width=.475\textwidth,
    xlabel={Mesh size $h$},
    ylabel={$L^2$ errors}, 
    xmin=.00625, xmax=.5,
%    ymin=1e-10, ymax=.5,
    legend pos=south east, legend cell align=left, legend style={font=\tiny},	
    xmajorgrids=true, ymajorgrids=true, grid style=dashed,
    legend entries={$N=1$,$N=2$,$N=3$,$N=4$,$N=5$}    
]
\pgfplotsset{
cycle list={{blue,  mark=*},{red, mark=square*},{cyan, mark=triangle*},{brown, mark=diamond*},{black, mark=pentagon*}}
}

\addplot+[semithick, mark options={solid, fill=markercolor}]
coordinates{(0.25,0.384932)(0.125,0.102767)(0.0625,0.02652)(0.03125,0.00667795)(0.015625,0.00167262)};
\logLogSlopeTriangleFlip{0.35}{0.125}{0.75}{2.00}{blue}
\addplot+[semithick, mark options={solid, fill=markercolor}]
coordinates{(0.25,0.129276)(0.125,0.0160502)(0.0625,0.00203945)(0.03125,0.000257735)(0.015625,3.23072e-05)};
\logLogSlopeTriangleFlip{0.35}{0.125}{0.6}{3.00}{red}
\addplot+[semithick, mark options={solid, fill=markercolor}]
coordinates{(0.25,0.0233657)(0.125,0.0017421)(0.0625,0.000120953)(0.03125,7.63711e-06)(0.015625,4.78665e-07)};
\logLogSlopeTriangleFlip{0.35}{0.125}{0.45}{4.00}{cyan}
\addplot+[semithick, mark options={solid, fill=markercolor}]
coordinates{(0.25,0.00557888)(0.125,0.000228086)(0.0625,6.71799e-06)(0.03125,2.1405e-07)(0.015625,6.72003e-09)};
\logLogSlopeTriangleFlip{0.35}{0.125}{0.28}{4.99}{brown}
\addplot+[semithick, mark options={solid, fill=markercolor}]
coordinates{(0.25,0.00130749)(0.125,1.94884e-05)(0.0625,3.98247e-07)(0.03125,6.30018e-09)(0.015625,9.90995e-11)};
\logLogSlopeTriangleFlip{0.35}{0.125}{0.115}{5.99}{black}

\end{loglogaxis}
\end{tikzpicture}
}
\subfloat[Two dimensional domain]{
\begin{tikzpicture}
\begin{loglogaxis}[
    width=.475\textwidth,
    xlabel={Mesh size $h$},
    ylabel={$L^2$ errors}, 
    xmin=.0125, xmax=.5,
%    ymin=1e-10, ymax=.5,
    legend pos=south east, legend cell align=left, legend style={font=\tiny},	
    xmajorgrids=true, ymajorgrids=true, grid style=dashed,
    legend entries={$N=1$,$N=2$,$N=3$,$N=4$,$N=5$}    
]
\pgfplotsset{
cycle list={{blue,  mark=*},{red, mark=square*},{cyan, mark=triangle*},{brown, mark=diamond*},{black, mark=pentagon*}}
}

\addplot+[semithick, mark options={solid, fill=markercolor}]
% N = 1, tau = 0.000000 =======================
coordinates{(0.25,0.214705)(0.125,0.0545393)(0.0625,0.0136973)(0.03125,0.00342883)};
\logLogSlopeTriangleFlip{0.375}{0.125}{0.76}{2.00}{blue}

\addplot+[semithick, mark options={solid, fill=markercolor}]
% N = 2, tau = 0.000000 =======================
coordinates{(0.25,0.0462894)(0.125,0.00595998)(0.0625,0.000759106)(0.03125,9.53391e-05)};
\logLogSlopeTriangleFlip{0.375}{0.125}{0.6}{2.99}{red}

\addplot+[semithick, mark options={solid, fill=markercolor}]
% N = 3, tau = 0.000000 =======================
coordinates{(0.25,0.00876667)(0.125,0.000625609)(0.0625,3.99053e-05)(0.03125,2.50774e-06)};
\logLogSlopeTriangleFlip{0.375}{0.125}{0.44}{3.99}{cyan}

\addplot+[semithick, mark options={solid, fill=markercolor}]
% N = 4, tau = 0.000000 =======================
coordinates{(0.25,0.00200257)(0.125,6.52844e-05)(0.0625,2.10741e-06)(0.03125,6.63717e-08)};
\logLogSlopeTriangleFlip{0.375}{0.125}{0.275}{4.99}{brown}

\addplot+[semithick, mark options={solid, fill=markercolor}]
% N = 5, tau = 0.000000 =======================
coordinates{(0.25,0.000314991)(0.125,6.49805e-06)(0.0625,1.04992e-07)(0.03125,1.65675e-09)};
\logLogSlopeTriangleFlip{0.375}{0.125}{0.105}{5.99}{black}

\end{loglogaxis}
\end{tikzpicture}
}
\caption{Convergence of the $L^2$ difference $\nor{\bm{u} - \bm{u}\LRp{\Pi_N\bm{v}}}_{L^2\LRp{\Omega}}$ between conservative variables and conservative variables as a function of the projected entropy variables with mesh refinement.  A GQ-$(N+2)$ quadrature rule is used in one dimension, while a volume quadrature rule of degree $2N$ is used in two dimensions.  The $L^2$ norm is evaluated using a quadrature of two degrees higher.   }
\label{fig:accUV}
\end{figure}
Numerical experiments also suggest that the constant in the $O(h^{N+1})$ error rate increases if $\rho_0$ and $E_0$ are decreased.  This is likely due to the fact that the error $\nor{\bm{u} - \bm{u}\LRp{\Pi_N\bm{v}}}_{L^2\LRp{\Omega}}$ depends on the nonlinearity of the mapping between conservation and entropy variables.  Because this mapping is non-invertible when $\rho \leq 0$ or $E-\frac{1}{2}\rho \LRb{\bm{u}}^2 \leq 0$, we expect that it becomes more nonlinear the closer the minimum values of the thermodynamic variables are to $0$.  

Finally, we briefly discuss the computational cost of the entropy stable method described in this work.  It was mentioned previously that the method proposed in this work can utilize high order polynomial approximation spaces with $N_p \leq N_q$, i.e.\ fewer degrees of freedom than the dimension of the underlying quadrature space.  However, this comes at the additional cost of computing $L^2$ projections of entropy variables and applying projection and lifting matrices for each right hand side evaluation.  Additionally, the main computational steps involving the Hadamard product and the application of the operator $\bm{D}^i_N$ depend only on the number of quadrature points $N_q$ and not $N_p$.  

Thus, the computational cost of the proposed methods is higher than that of SBP methods based on under-integrated quadrature rules.  However, the proposed work also allows for the construction of entropy conservative and entropy stable schemes for more general pairings of approximation spaces and quadrature.  Additionally, numerical experiments indicate that, in certain cases, the use of over-integrated quadrature rules improves solution accuracy.  Future work will compare the qualitative and quantitative behavior of entropy stable diagonal-norm SBP methods with the method proposed in this work.  
}

\section{Conclusions}

This work presents a generalization of discretely entropy conservative methods, allowing for the combination of more general approximation spaces and (sufficiently accurate) quadrature rules.  {We introduce SBP-like operators using quadrature-based projection and lifting matrices, and construct high order DG schemes for conservation laws.}  The resulting schemes satisfy a discrete conservation of entropy.  Numerical results for the compressible Euler equations indicate that the resulting methods deliver high order accuracy for smooth solutions while improving stability and robustness for under-resolved and shock solutions compared to non-entropy conservative and non-entropy stable schemes.  We also describe differences between the sensitivity of methods based on Gauss-Lobatto quadrature and methods based on Gauss quadrature rules.  

We note that, while entropy conservative and entropy stable schemes improve the robustness of solvers for nonlinear hyperbolic conservation laws, they do not address problems such as spurious oscillations in high order approximations of shock solutions or positivity preservation of density and pressure variables \cite{chen2017entropy}.  These issues can be addressed through the use of regularization (e.g.\ filtering, artificial viscosity) and/or limiting.  However, this can lead to a reliance on regularization techniques as ad-hoc stabilization mechanisms.  Addressing this issue requires constructing methods which avoid the need for excessive regularization and limiting in the pursuit of stability, and is a significant motivation for the development of entropy conservative and entropy stable discretizations.  

{Future work will focus on generalizations to three dimensions (including tetrahedra, prisms, and pyramidal elements), as well as analytical estimates for the error between the conservative and entropy-conservative variables.  These estimates will be necessary to construct error estimates for the proposed entropy stable methods.}  

\section{Acknowledgments}

The author thanks Lucas Wilcox, Andrew Winters, David M.\ Williams, and Weifeng Qiu for helpful discussions.  The author also thanks David C.\ Del Rey Fernandez for the name ``decoupled SBP operator''.  Jesse Chan is supported by the National Science Foundation under awards DMS-1719818 and DMS-1712639.  

\appendix
\section{Recovery of known schemes for Burgers' equation}
\label{appendix:A}
\noteOne{
In this appendix, we show how the framework presented recovers some existing entropy conservative (stable) schemes, and describe in more detail the connection between flux differencing and split formulations \cite{gassner2016split} in the context of the Burgers' equation.  The one-dimensional Burgers equation has flux $f(u) = u^2/2$, and can be rewritten in split form as
\begin{equation}
\pd{u}{t} + \frac{1}{3}\LRp{\pd{u^2}{x} + u\pd{u}{x}} = 0.  
\end{equation}
Assuming smooth $u$, one can show that discretizing this form of Burgers' equation results in a scheme which conserves the square entropy $U(\bm{u}) = u^2/2$ \cite{ranocha2017extended, ranocha2017generalised}.   As pointed out in \cite{gassner2016split, gassner2017br1}, this split formulation may also be recovered using flux differencing under the following two-point flux
\begin{equation}
f_S(u_L,u_R) = \frac{1}{6}(u_L^2 + u_Lu_R + u_R^2).
\label{eq:burgflux}
\end{equation}
Applying (\ref{eq:fluxdiff}) yields the split form of Burgers' equation
\begin{equation}
2\left.\pd{f_S(u(x),u(y))}{x}\right|_{y=x} = \frac{1}{3}\left.\pd{\LRp{u(x)^2 + u(x)u(y) + u(y)^2}}{x}\right|_{y=x} = \frac{1}{3}\LRp{\pd{u^2}{x} + u\pd{u}{x}}.
\end{equation}
Recovering the split-form of Burgers' equation relies on the property that $\pd{\LRp{u(y)^2}}{x} = u(y)^2\pd{\LRp{1}}{x} = 0$.  More generally, recovery of the split form can be achieved by using flux differencing in combination with any differential operator $D$ such that $D1 = 0$.  

Given the entropy conservative flux (\ref{eq:burgflux}) for Burgers' equation, we can show that the decoupled SBP operator recovers existing entropy stable schemes for Burgers' equation \cite{ranocha2017extended, ranocha2017generalised}.  We will assume periodic boundary conditions for simplicity and use the energy conserving flux
\begin{equation}
f_S(u_L,u_R) = \frac{1}{6}\LRp{u_L^2 + u_Lu_R + u_R^2}.  
\end{equation}
Define the entries of the matrix $(\bm{F}_S)_{ij} = u(\hat{{x}}_i)^2 + u(\hat{{x}}_i)u(\hat{{x}}_j) + u(\hat{{x}}_j)^2$.  Then, we have that 
\begin{equation}
\bm{F}_S = \frac{1}{6}\LRp{\bm{U}_L \circ \bm{U}_L + \bm{U}_L \circ \bm{U}_R + \bm{U}_R\circ \bm{U}_R}, \qquad \bm{U}_L = 
\LRs{\begin{array}{c}\bm{u}_q\\\bm{u}_f\end{array}}
\bm{e}^T, \qquad \bm{U}_R = \bm{e}\LRs{\begin{array}{c}\bm{u}_q\\\bm{u}_f\end{array}}^T,
\end{equation}
where $\bm{e}$ is the vector of all ones and $\bm{u}_q, \bm{u}_f$ denote the vector of $u(x)$ evaluated at volume and surface quadrature points, respectively.  
Applying the one-dimensional multi-element formulation (\ref{eq:multielemformulation}) to (\ref{eq:burgflux}) gives
\begin{align}
\td{\bm{u}}{t} +  \frac{1}{3}\LRs{\begin{array}{cc} 
\bm{P}_q & \bm{L}_q\end{array}} \LRp{\bm{D}_N \circ \LRp{\bm{U}_L \circ \bm{U}_L + \bm{U}_L \circ \bm{U}_R + \bm{U}_R\circ \bm{U}_R}} \bm{1} + \bm{L}_q\diag{{\bm{n}}}\LRp{\bm{f}^* - \bm{f}(\tilde{\bm{u}}_f)} = 0.
\label{eq:burgDN}
\end{align}
Using the definition of the Hadamard product and that $\bm{D}_N\bm{1} = 0$ from Theorem~\ref{thm:sbp}, we have that 
\begin{align}
\LRp{\LRp{\bm{D}_N\circ \bm{U}_L}\bm{1}}_i &= \sum_{j=1}^{N_q+N^f_q} \LRp{\bm{D}_N}_{ij} \LRp{\LRs{\begin{array}{c}\bm{u}_q\\\bm{u}_f\end{array}}}_j = \LRp{\bm{D}_N\LRs{\begin{array}{c}\bm{u}_q\\\bm{u}_f\end{array}}}_i,\\
\LRp{\LRp{\bm{D}_N\circ \bm{U}_R}\bm{1}}_i &= \sum_{j=1}^{N_q+N^f_q} \LRp{\bm{D}_N}_{ij} \LRp{\LRs{\begin{array}{c}\bm{u}_q\\\bm{u}_f\end{array}}}_i = \LRp{{\rm diag}\LRp{\LRs{\begin{array}{c}\bm{u}_q\\\bm{u}_f\end{array}}}\bm{D}_N\bm{1}}_i = \bm{0}
\end{align}
We can similarly show that % Lemma~\ref{lemma:hadamard}, we can show that 
\begin{align}
\LRp{\bm{D}_N\circ \bm{U}_L \circ \bm{U}_L}\bm{1} &= \bm{D}_N \LRs{\begin{array}{c}\bm{u}_q\\\bm{u}_f\end{array}}^2, \\
\LRp{\bm{D}_N\circ \bm{U}_L \circ \bm{U}_R}\bm{1} &= {\rm diag}\LRp{\LRs{\begin{array}{c}\bm{u}_q\\\bm{u}_f\end{array}}}\bm{D}_N\LRs{\begin{array}{c}\bm{u}_q\\\bm{u}_f\end{array}},  \qquad
\LRp{\bm{D}_N\circ \bm{U}_R \circ \bm{U}_R}\bm{1} = \bm{0}.\nonumber
\end{align}
where the entries are $\LRs{\begin{array}{c}\bm{u}_q\\\bm{u}_f\end{array}}^2$ are the squared entries of $\LRs{\begin{array}{c}\bm{u}_q\\\bm{u}_f\end{array}}$.  Substituting these into (\ref{eq:burgDN}) simplifies the formulation
\begin{align}
\td{\bm{u}}{t} &+  \frac{1}{3}\LRs{\begin{array}{cc} 
\bm{P}_q & \bm{L}_q\end{array}} \LRp{\bm{D}_N  \LRs{\begin{array}{c}\bm{u}_q\\\bm{u}_f\end{array}}^2 + {\rm diag}\LRp{\LRs{\begin{array}{c}\bm{u}_q\\\bm{u}_f\end{array}}} \bm{D}_N \LRs{\begin{array}{c}\bm{u}_q\\\bm{u}_f\end{array}}} \\
&+ \frac{1}{3}\bm{L}_q\diag{{\bm{n}}}\LRp{\frac{1}{6}\LRp{\bm{u}_f^2 + \bm{u}_f^+\bm{u}_f + \LRp{\bm{u}_f^+}^2}  - \frac{1}{2}\bm{u}_f^2} = 0,\nonumber
\label{eq:burgDN2}
\end{align}
where $\bm{u}_f^+$ denotes the values of the solution at face quadrature points on a neighboring elements.

We can now use (\ref{eq:pqvq}) and the structure of $\bm{D}_N$ to simplify (\ref{eq:burgDN2}) 
\begin{align}
\LRs{\begin{array}{cc} \bm{P}_q & \bm{L}_q\end{array}} 
\bm{D}_N  \LRs{\begin{array}{c}\bm{u}_q\\\bm{u}_f\end{array}}^2 &= \LRs{\begin{array}{cc} \bm{P}_q & \bm{L}_q\end{array}} \LRs{\begin{array}{c}
\bm{D}_q \bm{u}_q^2 + \frac{1}{2}\bm{V}_q\bm{L}_q{\rm diag}\LRp{\hat{\bm{n}}}\LRp{\bm{u}_f^2 - \bm{V}_f\bm{P}_q\bm{u}_q^2}\\
\frac{1}{2}{\rm diag}\LRp{\hat{\bm{n}}} \LRp{\bm{u}_f^2 - \bm{V}_f\bm{P}_q\bm{u}_q^2}\\
\end{array}}\\
&= \bm{D}_i\bm{P}_q\bm{u}_q^2 + \bm{L}_q{\rm diag}\LRp{\hat{\bm{n}}}\LRp{ \bm{u}_f^2 - \bm{V}_f\bm{P}_q\bm{u}_q^2 }\nonumber
\end{align}
We can similarly simplify the second part of (\ref{eq:burgDN2})
\begin{align}
&\LRs{\begin{array}{cc} \bm{P}_q & \bm{L}_q\end{array}} 
{\rm diag}\LRp{ \LRs{\begin{array}{c}\bm{u}_q\\\bm{u}_f\end{array}}}\bm{D}_N  \LRs{\begin{array}{c}\bm{u}_q\\\bm{u}_f\end{array}} \\ &=\LRs{\begin{array}{cc} \bm{P}_q & \bm{L}_q\end{array}} \LRs{\begin{array}{c}
{\rm diag}\LRp{ \bm{u}_q}\LRp{\bm{D}_q \bm{u}_q + \frac{1}{2}\bm{V}_q\bm{L}_q{\rm diag}\LRp{\hat{\bm{n}}}\LRp{\bm{u}_f - \bm{V}_f\bm{P}_q\bm{u}_q}}\\
\frac{1}{2}{\rm diag}\LRp{\bm{u}_f\hat{\bm{n}}} \LRp{\bm{u}_f - \bm{V}_f\bm{P}_q\bm{u}_q}\\
\end{array}}\nonumber\\
&= 
\bm{P}_q {\rm diag}\LRp{ \bm{u}_q}\LRp{\bm{D}_q \bm{u}_q}, \nonumber
\end{align}
where we have used that, because $u\in P^N$, $\bm{u} = \bm{P}_q\bm{V}_q\bm{u} = \bm{P}_q\bm{u}_q$ and $\bm{u}_f = \bm{V}_f\bm{u} = \bm{V}_f\bm{P}_q\bm{u}_q$.  Combining these two yields a simplified formulation
\begin{align}
&\td{\bm{u}}{t} + \frac{1}{3}\LRp{\bm{D}_i\bm{P}_q\bm{u}_q^2  + \bm{P}_q {\rm diag}\LRp{ \bm{u}_q}\LRp{\bm{D}_q \bm{u}_q}} \\
&+ \bm{L}_q{\rm diag}\LRp{\hat{\bm{n}}}\LRp{\frac{1}{6}\LRp{\bm{u}_f^2 + \bm{u}_f^+\bm{u}_f + \LRp{\bm{u}_f^+}^2}  - \frac{1}{2}\bm{u}_f^2 + \frac{1}{3}\LRp{\bm{u}_f^2 - \bm{V}_f\bm{P}_q\bm{u}_q^2 }}  = 0.
\end{align}
Cancelling out flux terms gives 
\begin{align}
 \td{\bm{u}}{t} &+  \frac{1}{3}\LRp{\bm{D}_i\bm{P}_q\bm{u}_q^2  + \bm{P}_q {\rm diag}\LRp{ \bm{u}_q}\LRp{\bm{D}_q \bm{u}_q}} \\
 &+ \bm{L}_q {\rm diag}\LRp{\hat{\bm{n}}} \LRp{\frac{1}{6}\LRp{\bm{u}_f^+\bm{u}_f + \LRp{\bm{u}_f^+}^2} -\frac{1}{3}\bm{V}_f\bm{P}_q\bm{u}_q^2} = 0.\nonumber
\end{align}
If Gauss quadrature with $(N+1)$ points is used, and the basis is chosen to be the co-located Lagrange basis at Gauss nodes, then $\bm{V}_q = \bm{P}_q = \bm{I}$, $\bm{u} = \bm{u}_q$, and the formulation reduces to
\begin{align}
 \td{\bm{u}}{t} &+  \frac{1}{3}\LRp{\bm{D}_i \bm{u}^2  + {\rm diag}\LRp{ \bm{u}}\LRp{\bm{D}_i \bm{u}}} 
+ \bm{L}_q {\rm diag}\LRp{\hat{\bm{n}}} \LRp{\frac{1}{6}\LRp{\bm{u}_f^+\bm{u}_f + \LRp{\bm{u}_f^+}^2} -\frac{1}{3}\bm{V}_f\LRp{\bm{u}^2}} = 0.
\end{align}
This is equivalent to the weak form of the split formulation and the correction terms introduced in \cite{ranocha2017extended} for Burgers' equation with an entropy conservative flux.  Entropy conservative SBP-SATs were introduced for generalized SBP operators and the compressible Euler equations in \cite{ranocha2017comparison}; however, we have not yet determined if (\ref{eq:multielemformulation}) recovers these formulations as well. 
}

\bibliographystyle{unsrt}
\bibliography{dg}

\end{document}